\documentclass{amsart}
\usepackage{amsmath}
\usepackage{amsthm}
\usepackage{amsfonts}
\usepackage{amssymb}
\usepackage{graphicx}
\usepackage{enumerate}
\usepackage{tikz-cd}
\usepackage{array}
\usepackage[parfill]{parskip}
\usepackage{subcaption}
\usepackage{upgreek}
\usepackage{textcomp}
\usepackage{setspace}
\usepackage[utf8]{inputenc}
\usepackage[fontsize=10pt]{scrextend}
\usepackage[left=2.5cm, right=2.5cm, top=2.4cm, bottom=2.4cm]{geometry}
\theoremstyle{definition}
\newtheorem{thm}{Theorem}[section]
\newtheorem{defn}[thm]{Definition}
\newtheorem{defn-prop}[thm]{Definition-Proposition}
\newtheorem{rem}[thm]{Remark}
\newtheorem{lem}[thm]{Lemma}
\newtheorem{exmp}[thm]{Example}
\newtheorem{cor}[thm]{Corollary}
\newtheorem{prop}[thm]{Proposition}
\newtheorem{notat}[thm]{Notation}

\newcommand\hcancel[2][black]{\setbox0=\hbox{$#2$}%
	\rlap{\raisebox{.45\ht0}{\textcolor{#1}{\rule{\wd0}{1pt}}}}#2} 
\DeclareMathOperator{\pertL}{\hcancel{\mathcal{L}}}
\newcommand*{\flip}[1]{\scalebox{-1}[1]{\rotatebox[origin = c]{180}{#1}}}
\newcommand{\HMR}{\flip{\textit{W}}}

\newcommand{\hmr}{\textit{HMR}}
\makeatletter
\DeclareRobustCommand\widecheck[1]{{\mathpalette\@widecheck{#1}}}
\def\@widecheck#1#2{%
	\setbox\z@\hbox{\m@th$#1#2$}%
	\setbox\tw@\hbox{\m@th$#1%
		\widehat{%
			\vrule\@width\z@\@height\ht\z@
			\vrule\@height\z@\@width\wd\z@}$}%
	\dp\tw@-\ht\z@
	\@tempdima\ht\z@ \advance\@tempdima2\ht\tw@ \divide\@tempdima\thr@@
	\setbox\tw@\hbox{%
		\raise\@tempdima\hbox{\scalebox{1}[-1]{\lower\@tempdima\box
				\tw@}}}%
	{\ooalign{\box\tw@ \cr \box\z@}}}
\makeatother

\DeclareMathOperator{\grad}{grad}

\DeclareMathOperator{\Tr}{Tr}

\DeclareMathOperator{\SO}{SO}

\DeclareMathOperator{\Spin}{Spin}
\DeclareMathOperator{\Hom}{Hom}
\DeclareMathOperator{\Pic}{Pic}

\DeclareMathOperator{\Id}{Id}

\DeclareMathOperator{\su}{\mathfrak{su}}
\DeclareMathOperator{\genus}{\text{genus}}
\DeclareMathOperator{\End}{End}

\DeclareMathOperator{\U}{\text{U}}
\DeclareMathOperator{\frakF}{\mathfrak{F}}
\DeclareMathOperator{\frakq}{\mathfrak{q}}
\DeclareMathOperator{\frakp}{\mathfrak{p}}
\DeclareMathOperator{\frakm}{\mathfrak{m}}

\DeclareMathOperator{\fraks}{\mathfrak{s}}
\DeclareMathOperator{\fraka}{\mathfrak{a}}
\DeclareMathOperator{\frakb}{\mathfrak{b}}
\DeclareMathOperator{\frakc}{\mathfrak{c}}
\DeclareMathOperator{\calD}{\mathcal{D}}
\DeclareMathOperator{\calA}{\mathcal{A}}
\DeclareMathOperator{\calH}{\mathcal{H}}

\DeclareMathOperator{\calT}{\mathcal{T}}
\DeclareMathOperator{\calU}{\mathcal{U}}
\DeclareMathOperator{\calV}{\mathcal{V}}

\DeclareMathOperator{\calQ}{\mathcal{Q}}
\DeclareMathOperator{\calM}{\mathcal{M}}

\DeclareMathOperator{\boldB}{\boldsymbol{\calB}}

\newcommand{\bbZ}{\mathbb{Z}}
\newcommand{\gr}{\textsf{gr}}

\newcommand{\bbF}{\mathbb{F}}

\newcommand{\bbT}{\mathbb{T}}
\newcommand{\bbS}{\mathbb{S}}
\newcommand{\bbJ}{\mathbb{J}}
\newcommand{\bbR}{\mathbb{R}}
\newcommand{\red}{\text{red}}
\newcommand{\HM}{\textit{HM}}

\newcommand{\calB}{\mathcal{B}}
\newcommand{\calC}{\mathcal{C}}
\newcommand{\calE}{\mathcal{E}}

\newcommand{\calG}{\mathcal{G}}
\newcommand{\calJ}{\mathcal{J}}
\newcommand{\calK}{\mathcal{K}}
\newcommand{\calL}{\mathcal{L}}

\newcommand{\calP}{\mathcal{P}}
\newcommand{\calS}{\mathcal{S}}
\newcommand{\reals}{\mathbb R}
\newcommand{\C}{\mathbb{C}}
\newcommand{\Ftwo}{\mathbb{F}_2}

\newcommand{\del}{\ensuremath{\partial}}
\newcommand{\delbar}{\ensuremath{{\bar{\partial}}}}
\renewcommand{\Re}{\text{Re}}

\newcommand{\loc}{\text{loc}}

\renewcommand{\O}{\text{O}}
\newcommand{\spinc}{{\text{spin\textsuperscript{c} }}}

\newcommand{\SpincR}{{\textsf{RSpin\textsuperscript{c}}}}
\newcommand{\RSpin}{{\textsf{RSpin}}}

\newcommand{\Ker}{\ensuremath{{\sf{ker}}}}
\newcommand{\Tait}{\ensuremath{{\sf{Tait}}}}
\newcommand{\vol}{\ensuremath{{\sf{vol}}}}
\newcommand{\Sing}{\ensuremath{{\sf{Sing}}}}
\newcommand{\scal}{\ensuremath{{\sf{scal}}}}
\newcommand{\Ver}{\ensuremath{{\sf{Vertex}}}}

\newcommand{\Edg}{\ensuremath{{\sf{Edge}}}}
\newcommand{\Fix}{\ensuremath{{\sf{Fix}}}}
\newcommand{\Facet}{\ensuremath{{\sf{Facet}}}}

\newcommand{\sfr}{\ensuremath{{\sf{r}}}}
\newcommand{\sfc}{\ensuremath{{\sf{c}}}}

\newcommand{\Rank}{\ensuremath{{\sf{Rank}}}}

\newcommand{\ind}{\ensuremath{{\sf{ind}}}}

\newcommand{\dbcv}{\ensuremath{{\sf{\Sigma}}}}

\newcommand{\la}{\ensuremath{{\langle}}}
\newcommand{\ra}{\ensuremath{{\rangle}}}
\newenvironment{citemize}{
	\setlength{\parskip}{0pt}
	\begin{itemize}
		\setlength{\itemsep}{0.5pt}
		\setlength{\parskip}{0pt}
		\setlength{\parsep}{0pt}
	}
	{\end{itemize}\setlength{\parskip}{11pt}} 
\newcommand{\Addresses}{{
	\bigskip
	\footnotesize
	\textsc{Department of Mathematics, 
Harvard University, 
Cambridge MA, 
United States 02138}\par\nopagebreak
  \textit{E-mail address:} \texttt{jiakaili@math.harvard.edu}
}}

\title{Monopole invariants for webs and foams}

\newcommand\blfootnote[1]{%
	\begingroup
	\renewcommand\thefootnote{}\footnote{#1}%
	\addtocounter{footnote}{-1}%
	\endgroup
}

\begin{document}    
	\blfootnote{This work was partially supported by a Simons Foundation Award \#994330 (Simons Collaboration on New Structures in Low-Dimensional Topology)}
\begin{abstract}
	We develop monopole Floer-theoretic invariants for webs and foams using Seiberg-Witten theory with orbifold singularities, based on Kronheimer and Mrowka's framework of monopole Floer homology.
\end{abstract}
\author{Jiakai Li}
\maketitle
\section{Introduction}
For the scope of this paper, we adopt the following definitions of webs and foams:

\emph{Webs} are, roughly speaking, embedded trivalent graphs in $\bbR^3$.
More precisely, a \emph{web} $K$ in $\bbR^3$ is a compact $1$-dimensional subcomplex, such that the complement $K \setminus \Ver(K)$ is a smooth $1$-dimensional submanifold of $\reals^3$, where $\Ver(K)$ is a finite subset of $K$, consisting of \emph{vertices}.
Each vertex is required to have a neighbourhood in which $K$ is diffeomorphic to three distinct, coplanar rays in $\bbR^3$.
An \emph{edge} of $K$ is a connected component of $K \setminus \Ver(K)$.
In particular, following Kronheimer and Mrowka~\cite{KMweb}, we assume that edges are unoriented, in contrast with Khovanov~\cite{Khovanov2004sl3}.

\emph{Foams} can be thought of as singular cobordisms between webs.
Specifically, a \emph{(closed) foam} $\Sigma$ in $\bbR^4$ is a compact $2$-dimensional subcomplex with the requirement that around each point $x$ on $\Sigma$, the subcomplex has one of the following three local models:
\begin{enumerate}[(i)]
	\item $\reals^2$, i.e. $\Sigma$ is a smooth embedded $2$-manifold near $x$;
	\item $\reals \times K_3$, where $K_3 \subset \bbR^3$ is the union of $3$ distinct rays meeting at the origin; or
	\item a cone in $\reals^3 \subset \reals^4$ whose vertex is $x$ and whose base is the union of $4$ distinct points in $S^2$ joined pairwise by $6$ non-intersecting great-circle arcs, each of length at most $\pi$.
\end{enumerate}
A vertex $x$ of type (iii) is a \emph{tetrahedral point}. 
The union of type (ii) points form a union of arcs called \emph{seams}.
The complement of tetrahedral points and seams is a smoothly embedded 2-manifold, possibly nonorientable, whose components will be referred to as \emph{facets}.

The terms ``webs'' and ``foams'' were coined by Kuperberg~\cite{Kuperberg1996Web}, and
the type of webs and foams considered in this paper goes back to Khovanov and Rozansky \cite{KhovanovRozansky2007LGfoam}.
Evaluations of foams played a central role in Khovanov's definition of $\mathfrak{sl}(3)$-link homology \cite{Khovanov2004sl3}.
An invariant $J^{\sharp}$ for webs and foams from instanton gauge theory was introduced by Kronheimer and Mrowka \cite{KMweb}.
Notably,
Kronheimer and Mrowka proposed an alternative approach \cite{KMweb} to the four-colour theorem using $J^{\sharp}$.
This theorem in combinatorics was first proved by Appel and Haken~\cite{AH4colour}, with computer assistance by J. Koch.

The construction of $J^{\sharp}$ utilizes $\SO(3)$-connections with orbifold singularities.
It is natural to seek an ``orbifold Seiberg-Witten Floer homology'' for webs.
As it turns out, there exists no direct Seiberg-Witten analogue of $J^{\sharp}$.
Roughly speaking, this is due to the lack of an embedding of the Klein four group $\bbZ/2 \times \bbZ/2$ into the structural group $\U(1)$ of Seiberg-Witten theory.
This paper provides an approach to Seiberg-Witten gauge theory for Klein four singularities.

The first step in defining a Seiberg-Witten web invariant is to prescribe ``holonomies'' around the three edges at a vertex.
The key to our construction is to assign both  ``conjugate-linear and complex-linear holonomies''.
Similar types of singularity configurations appeared in Daemi's plane Floer homologies for webs~\cite{DaemiPrivateComm}, and
our approach can be seen as a Seiberg-Witten generalization of Daemi's invariant.


The variant of Floer homology in this paper is based on Kronheimer and Mrowka's framework of monopole Floer homology \cite{KMbook2007}.
The additional technology deployed here is a combination of real and orbifold Seiberg-Witten theory.
The former was developed in the works of \cite{TianWang2009, NNakamura2015, Kato2022, KMT2021, KMT2023}.
In particular, \cite{ljk2022} is a special case of the current setup without orbifold singularities.
The $4$-dimensional orbifold Seiberg-Witten invariant was studied by e.g. W. Chen~\cite{ChenWM2006scob,ChenWM2005Jcurves_in_4orbifolds,ChenWM2020symp_4orbifolds} and C. LeBrun~\cite{LeBrun2015_edges}.
In the $3$-dimensional literature, there are works of S. Baldridge~\cite{Baldridge2001} and W. Chen~\cite{ChenWM2012SW3orbifolds} on orbifold Seiberg-Witten invariants.

\subsection{Structure of the invariants}
\label{sec:intro_str}
This paper is only concerned with the category of  webs in $S^3$ and foams properly embedded punctured $S^4$'s, so as not to distract the readers with excessive notations.
	There are no technical difficulties in extending the definition to webs in more general $3$-manifolds.

Given a web $K \subset S^3$, we define three \emph{monopole web  Floer homologies} of $K$:
\[
	\widecheck{\HMR}_{\bullet}(K), \quad \widehat{\HMR}_{\bullet}(K), \quad
	\overline{\HMR}_{\bullet}(K).
\]
These are graded vector spaces over $\bbF_2$, where ``$\bullet$'' denotes certain completion of some grading.
We will work only over characteristic two.
\begin{rem}
The notations ``$\widecheck{\HMR}_{\bullet}(K), \widehat{\HMR}_{\bullet}(K),
	\overline{\HMR}_{\bullet}(K)$'' (pronounced ``M-to, M-from, M-bar'') follow Kronheimer and Mrowka's \cite{KMbook2007} monopole Floer homology groups $\widecheck{\HM}_{\bullet}(Y),\widehat{\HM}_{\bullet}(Y), \overline{\HM}_{\bullet}(Y)$ of a closed $3$-manifold $Y$.
They are counterparts of Ozsv\'{a}th and Szab\'{o}'s \cite{OzSz2004} Heegaard Floer homology ``$\textit{HF}^+(Y)$,$\textit{HF}^-(Y)$, $\textit{HF}^{\infty}(Y)$''.
For convenience, we will write ``$\HMR^{\circ}_{\bullet}(K)$'' for one of the flavours: ``$\widecheck{\HMR}_{\bullet}(K), \widehat{\HMR}_{\bullet}(K),
	\overline{\HMR}_{\bullet}(K)$.''
\end{rem}

Denote by $\Edg(K)$ the set of edges of $K$. A \emph{$1$-set} $s$ of $K$ is a subset of $\Edg(K)$ which contains exactly one edge at every vertex.
The construction of $\HMR$  requires an input of a $1$-set $s$. 
By definition, $\HMR^{\circ}_{\bullet}(K)$ is a direct sum
\[
	\HMR^{\circ}_{\bullet}(K) = \bigoplus_{s \in \{\text{1-sets of }K\}} \HMR^{\circ}_{\bullet}(K,s).
\] 
\begin{rem}
	The group $\HMR^{\circ}_{\bullet}(K,\emptyset)$ agrees with $\textit{HMR}^{\circ}_{\bullet}(K)$ in \cite{ljk2022} for a link $K$.
\end{rem}

\begin{defn}
A \emph{(foam) cobordism} $\Sigma: K_- \to K_+$ between two webs $K_-$ and $K_+$ is a foam with boundary $\Sigma \subset [a,b] \times \reals^3$ such that near the ends $a$ and $b$, the foam is modelled on $[a,a+\epsilon) \times K_-$ and $(b-\epsilon, b] \times K_+$, respectively.
\end{defn}
Given a foam cobordism $\Sigma: K_- \to K_+$, there is a cobordism map for each of the flavour:
\[
	\HMR^{\circ}(\Sigma):\HMR^{\circ}_{\bullet}(K_-) \to \HMR^{\circ}_{\bullet}(K_+).
\]
Its definition requires a choice of $1$-set $s_{\Sigma}$, now for the foam $\Sigma$.
Let $\Facet(\Sigma)$ be the set of facets.
A \emph{$1$-set} $s_{\Sigma}$ is a subset of $\Facet(\Sigma)$ such that among the three facets around a seam, there is a unique facet in $s_{\Sigma}$.
The cobordism map is a sum
\[
	\HMR^{\circ}(\Sigma) 
	= \bigoplus_{s_{\Sigma} \in \{\text{1-sets of }\Sigma\}} \HMR^{\circ}(\Sigma,s_{\Sigma}).
\]
A \emph{dotted foam} is a foam with marked points on the facets called ``\emph{dots}''.
The definition of a dotted foam cobordism maps involves evaluation of cohomology classes on Seiberg-Witten moduli spaces.
As a special case, a dot $\delta$ on $K \setminus \Ver(K)$ gives rise to an operator
\[
	\upsilon_{\delta}: \HMR^{\circ}_{\bullet}(K) \to \HMR^{\circ}_{\bullet}(K), 
\]
that depends only on the edge where $\delta$ lies. 
Suppose $s$ is a $1$-set and $\delta \in e$.
\begin{prop}
\label{prop:powers_upsilon}
The operator $\upsilon_{\delta}$
satisfies the following properties (see Section~\ref{subsec:foamcob_eval_module}):
\begin{citemize}
	\item $\upsilon_{\delta}=0$, if $e \in s$;
	\item $\upsilon_{\delta}^2 = U$ on $\HMR^{\circ}_{\bullet}(K,s)$, if $e \notin s$, where $U$ is independent of $e$.
\end{citemize}
Suppose $\delta_1,\delta_2,\delta_3$ are on three distinct edges that meet at a common vertex. Then
\begin{citemize}
	\item $\upsilon_{\delta_1} + \upsilon_{\delta_2} + \upsilon_{\delta_3} = 0$,
	\item $\upsilon_{\delta_1}\upsilon_{\delta_2} + \upsilon_{\delta_2}\upsilon_{\delta_3} + \upsilon_{\delta_3}\upsilon_{\delta_1} = U$, 	
	\item $\upsilon_{\delta_1}\upsilon_{\delta_2}\upsilon_{\delta_3}=0$.
\end{citemize}
\end{prop}
The above relations are reminiscent of a deformed version $J^{\sharp}(K,\Gamma)$ of $J^{\sharp}$ introduced in \cite[Proposition~5.7]{KMdefweb2019}.
Similar identities are also satisfied in $\mathfrak{sl}(3)$-homology \cite{Khovanov2004sl3} and $J^{\sharp}$ \cite{KMweb}.

Let $p$ be a point on $K$. 
The mapping cone complex of $\upsilon_{p}$ defines the \emph{framed monopole web homology} based at $p$:
\[
	\widetilde{\HMR}_{\bullet}(K,p,s),
\]
the same way as $\widetilde{\textit{HMR}}$ in~\cite{ljk2024SSKh}.
This algebraic construction is inspired by $\widetilde{\HM}$ in monopole Floer~\cite{Bloom2011} and $\widehat{\textit{HF}}$ in Heegaard Floer theories~\cite{OzSz2004}.
The framed monopole web homology satisfies functoriality with respect to based foam cobordisms.

Suppose now $\Sigma$ is a closed foam, where $\{\delta_1,\dots,\delta_n\}$ a set of dots.
We define a foam evaluation of $\Sigma$ by counting irreducible Seiberg-Witten solutions
\[
\frakm(\Sigma,s_{\Sigma}|\delta_1,\dots,\delta_n) \in \bbF_2.
\]
Note that $\frakm(\Sigma,s_{\Sigma}|\delta_1,\dots,\delta_n)$ is well-defined under certain topological condition on $\Sigma$. See Section~\ref{sec:functoriality}.
This is a special case of the ``$\overrightarrow{\HMR}$'' map.
By considering $\Sigma$ as a cobordism between the empty webs, and the closed $4$-manifold invariant $\overrightarrow{\HM}$, 
we define
\[
	\overrightarrow{\HMR}(\Sigma,s_{\Sigma}):
		\widehat{\HMR}_{\bullet}(\emptyset) \to 
		\widecheck{\HMR}_{\bullet}(\emptyset).
\]

We prove an adjunction inequality in dimension three and a vanishing theorem for webs containing \emph{embedded bridges}.
This property is also enjoyed by $J^{\sharp}$.
\begin{thm}
Suppose a web $K$ has an embedded bridge, that is, there exists an embedded $2$-sphere that intersects $K$ at a single point.
Then $\HMR_{\bullet}(K) = 0$.
\end{thm}
Finally, we prove an excision theorem for cutting along genus $g$ surfaces that intersect the web at points.
As an application, we obtain a nonvanishing theorem for a family of braids of the following form.

Let $\gamma$ be a braid of $n$ strands in the cylinder $[0,1] \times D^2$.
Suppose $\del\gamma$ consists of $n$ points lying on a circle $U_i \subset \{i\} \times D^2$, for $i=0,1$ respectively.
Let $K_{\gamma}$ be the web formed as the union of $\gamma$ and $U_0,U_1$.
Let $s_{\gamma}$ be the $1$-set comprising strands of $\gamma$.
\begin{prop}
	\label{prop:intro_excision_app}
	$\HMR_{\bullet}(K_{\gamma},s_{\gamma}) \ne 0$.
\end{prop}
\begin{figure}[h!]
	\centering
	\includegraphics[width=0.25\linewidth]{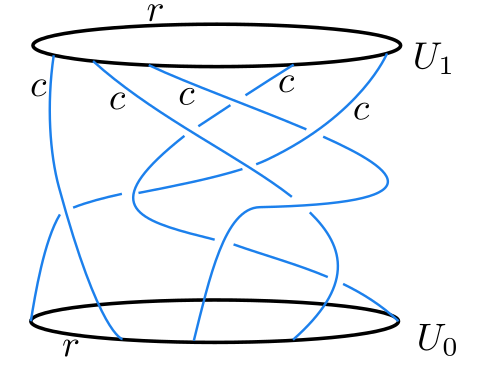}
	\caption{A web $K_{\upgamma}$ from a $5$-braid $\gamma$.}
	\label{fig:5braid}
\end{figure}

\subsection{Organization of the paper}
We describe the topology of the orbifolds of interest in Section~\ref{sec:bifold_singularities} and set up the
bifold Seiberg-Witten theory with involutions in Section~\ref{sec:bifoldspinc}.  Section~\ref{sec:analysis_of_sw_traj} is devoted to the analysis of Seiberg-Witten trajectories.
We define the Floer homologies of webs in Section~\ref{sec:floer_homology} and discuss
their functorial aspects in Section~\ref{sec:functoriality}.
Section~\ref{sec:framed} discusses the framed web homology.
In Section~\ref{sec:adj_excis}, we prove an adjunction inequality and an excision theorem, along with some applications.
We examine some examples in Section~\ref{sec:example}.
We conclude with some observations and remarks in Section~\ref{sec:remarks}.

\subsection{Acknowledgement} 
I would like to thank my advisor, Peter Kronheimer, for his guidance and encouragement during this project.
I appreciate the several conversations that I had with Ali Daemi about plane Floer homology.
I am grateful to Masaki Taniguchi for many discussions on orbifold Seiberg-Witten theory.
I would also like to thank the Simons Foundation and organizers of the ``Instantons and Foams'' workshop at MIT.
This was a satellite event of the Simons collaboration on new structures in low-dimensional topology, which partly motivated this work.

\section{Bifold singularities} 
\label{sec:bifold_singularities}
\subsection{Bifolds as orbifolds}
\label{subsec:bifolds_orbi}
Given a manifold $M$, recall that an orbifold structure $\check M$ consists of a system of charts $\{U\}$ of $M$ and for each chart $U$, a quotient map
\[
	\phi: \tilde U \to U
\]
for an action of finite group $H$. 
For $x \in U$, let $H_x$ be the stabilizer group of a preimage $\tilde x \in \tilde U$ of $x$.
We write $\check M = (M,\Sing)$ where $\Sing$ consists of $x \in H_x$ such that $H_x \ne \{1\}$.
In this paper, $M$ will either be a $3$- or $4$-manifold, and $H_x$ will be trivial, $\bbZ/2$, $V_4$, or $V_8$.
Here, $V_4 \le \SO(3)$ and $V_8 \le \SO(4)$ are subgroups realized by diagonal matrices with entries in $\{\pm 1\}$.

For a $3$-orbifold $\check Y = (Y,K)$, we require that every point $x \in K$ lies in a neighbourhood modelled on $\reals \times (\reals^2/\{\pm 1\})$ or $\reals^3/V_4$.
For a $4$-orbifold $\check X = (X,\Sigma)$, there are $3$ local models: $\reals^2 \times \reals^2/\{\pm 1\}$, $\reals \times \reals^3/V_4$, and $\reals^4/V_8$.
\begin{defn}
	A \emph{bifold} is a $3$- or $4$-dimensional orbifold whose local models belong to the types described above.
\end{defn}
Given a smooth Riemannian metric on the underlying $3$- or $4$-manifold of a bifold $(M,\Sing)$, we can modify the Riemannian metric to an orbifold metric $\check g$  in the manner of \cite[Section~2.2]{KMweb}.
If $M$ is a $3$-manifold and $\Sing$ is a web, the resulting metric will have cone-angle $\pi$ around each edge.
On the other hand, if $M$ is a $4$-manifold and $\Sing$ is a foam, then the bifold metric will have cone-angle $\pi$ along facets of the foam.

\subsection{Cohomology and line bundles}
\label{subsec:cohom_line_bundles}
An orbifold $k$-form is given in a orbifold chart $\tilde U \to U$ by an invariant $k$-form on $\tilde U$.
Let $\Omega^k(\check M; \reals)$ be the space of $\reals$-valued $1$-forms.
The orbifold de Rham complex $(\Omega^*(\check M;\reals),d)$ computes the de Rham cohomology group $H^k(\check M;\reals)$.
By Poincar\'{e} lemma, the $H^k(\check M;\reals)$ is isomorphic to $H^k( M;\reals)$ the cohomology of the underlying manifold.
Let $H^1(\check M;\bbZ)$ be the subgroup of $H^1(\check M;\reals)$ represented by bifold $1$-forms of integral periods.
There is a one-to-one correspondence between elements of $H^1(\check M;\bbZ)$ and orbifold harmonic maps from $\check M$ to $S^1$.

Let $L \to \check M$ be a bifold line bundle over a bifold $\check M$. 
The \emph{bifold first Chern class} $c_1(L) \in H^2(\check M;\bbR)$ of $L$ is represented by the de Rham class of the $2$-form $iF_A/2\pi$, where $A$ is an orbifold connection on $L$.
(In fact this $c_1(L)$ lies in rational cohomology, as $L^{\otimes 2}$ is smooth over the underlying space.)

Let $\check Y$ be a 3-bifold and suppose the singular set $K \subset Y$ has no $V_4$-point.
The isomorphism classes of bifold line bundles can be analyzed as follows.
Let $K_i$ be a component of $K$. Over a tubular neighbourhood of $K_i$, consider the trivial line bundle over its branched double cover
$\C \times (K_i \times D^2) \to (K_i \times D^2)$ 
where the involution acts by
\begin{equation}
	(\lambda,\theta, z) \mapsto (-\lambda,\theta, -z),
\end{equation}
where $z$ is a complex coordinate on the unit disc $D^2$.
Let $E_i$ be the line bundle obtained by gluing the orbifold line bundle on the neighbourhood of $K_i$ and the trivial bundle on $Y \setminus (\cup K_i)$ via the map $(\theta,z) \mapsto z$.

Given any orbifold line bundle $L$, there exists a set of integers $\beta_i \in \{0,1\}$ for which $L \otimes E_1^{\beta_1} \otimes \cdots \otimes E_n^{\beta_n}$ is a smooth line bundle on the underlying 3-manifold.
We refer to $\beta_i$ as the \emph{isotropy representation (or isotropy data)} of $L$ at $K_i$.

Let $\Pic(\check Y)$ be the (topological Picard) group of equivalence classes of bifold line bundles on $\check Y$, equipped with tensor product.
The following theorem is due to Baldridge~\cite[Theorem~4]{Baldridge2001}.
\begin{thm}
	Let the isotropy representations be fixed.
	The isomorphism classes of bifold line bundles on $\check Y$ with  along $K_i$ are in bijective correspondence with classes in $H^2(Y;\bbZ)$.
\hfill \qedsymbol
\end{thm}
In particular, orbifold line bundles over $S^3$ are uniquely determined by their isotropy data.
\begin{cor}
\label{cor:bifold_spinc_S3}
	If $\check Y = (S^3,K)$, then $\Pic(\check Y)$ is isomorphic to $\bbZ_2^n = \{0,1\}^n$, where $n$ is the number of components of $K$.
\hfill \qedsymbol
\end{cor}

\subsection{Bifold connections}
Let $G$ be a compact connected Lie group and let $W$ be a $G$-representation.
The relevant pairs $(G,W)$ for us are
\begin{citemize}
	\item $G = \SO(3)$ and $W = \reals^3$, 
	\item $G = \U(2) = \Spin^c(3)$ and $W = \C^2$, 
	\item $G = \U(1)$ and $W= \C$.
\end{citemize}
Suppose $\check M = (M,\Sing)$ is a bifold. Following \cite{KMweb}, we define a \emph{$C^{\infty}$ orbifold $G$-connection} to be a pair $(E,A)$ over $\check M$, where $E$ is a $W$-bundle over $\check M \setminus \Sing$ and $A$ is a connection over $E$ such that the pull-back of $(E,A)$ under any orbifold chart $\phi: \tilde U \to U$ extends to a smooth pair $(\tilde E, \tilde A)$.
A \emph{bifold connection} on a bifold $\check M$ is a $C^{\infty}$-connection 
such that every order-$2$ stabilizer groups $H_x$ act non-trivially on the corresponding fibre $\tilde E_x$.

Let $K$ be a web.
The idea of Kronheimer and Mrowka is to study bifold $\SO(3)$-connections on $(Y^3,K)$, with order-$2$ holonomy constraints on meridians of $K$.
When three edges meet at a vertex, the three meridianal holonomies (i.e. the action of $H_x$ on the fibre $\tilde E_x$) generate a subgroup conjugate to $V_4 \subset \SO(3)$.

For abelian bifold connections, as observed by Daemi~\cite{DaemiPrivateComm}, it is necessary to extend $\U(1)$ by $\mathbb Z_2$ (generated by conjugation) so that $V_4$ embeds.
At a vertex of $K$,
the holonomy around an edge is either $(-1)$ or complex conjugation $(\pm \sigma)$.
The only possible configuration is to have exactly two conjugate-linear holonomies and one complex-linear holonomy. In other words, the set of \emph{complex-linear edges} is a $1$-set.

For context, a Lie group $G$ with an involution $\sigma: G \to G$ define a $\mathbb Z_2$-extension $G_{\sigma}$ of $G$.
Gauge theory of \emph{singular $G_{\sigma}$-connections} was studied by Daemi~\cite{DaemiThesis},  where the singular loci are links and surfaces with meridianal holonomy $\sigma$.
A $G_{\sigma}$-bundle over $\check M\setminus \Sing$ comes with a data of a double cover, induced by the component homomorphism $G_{\sigma} \to \mathbb Z_2$.
At the singular locus, two fibres of $G_{\sigma}$ are identified by the holonomy. 
A \emph{singular $G_{\sigma}$-bundle} consists of a choice of double branched cover $\tilde M \to M$, a $G$-bundle $P \to \tilde M$, and a lift $\tau: P \to P$ of the covering involution $\tilde M \to \tilde M$.
This lift $\tau$ induces an action $\tau_*$ over the $G$-connections over $\tilde M$, and a \emph{singular $G_{\sigma}$-connection} is a  $\tau_*$-invariant connection.
\subsection{\emph{Real} covers along webs and foams}
The approach in \cite{ljk2022} is a Seiberg-Witten generalization of the singular $\U(1)_{\sigma}$-instanton theory studied by Daemi~\cite{Daemi20165} under the name of \emph{plane Floer homology}.
In addition to an involution on the space of $\U(1)$-connections, a complex anti-linear involutive lift of the covering involution over the spinor bundle must be chosen.
To set up the Seiberg-Witten theory for webs, we begin by rephrasing
the definition of a $1$-set.
\begin{defn}
A \emph{$1$-set} for a web $K$ is a colouring of $\Edg(K)$ by a 2-element set $\{\sfr, \sfc\}$ such that around each vertex there are exactly two edges coloured by $\sfr$ and one edge coloured by $\sfc$.
\end{defn}
This definition can be seen to be equivalent to the one in Section~\ref{sec:intro_str}, by setting $s$ to be the set of edges coloured by $\sfc$.
Let $K^{\sfr}$ and $K^{\sfc}$ be the closures of the union of $\sfr$-edges and $\sfc$-edges in $K \subset S^3$, respectively.
As a consequence of the definition, $K^{\sfr}$ is a $1$-manifold.
We refer to $K^{\sfr}$ as the \emph{real locus} and $K^{\sfc}$ as the \emph{complex locus}.
\begin{figure}[h!]
	\centering
	\includegraphics[width=0.25\linewidth]{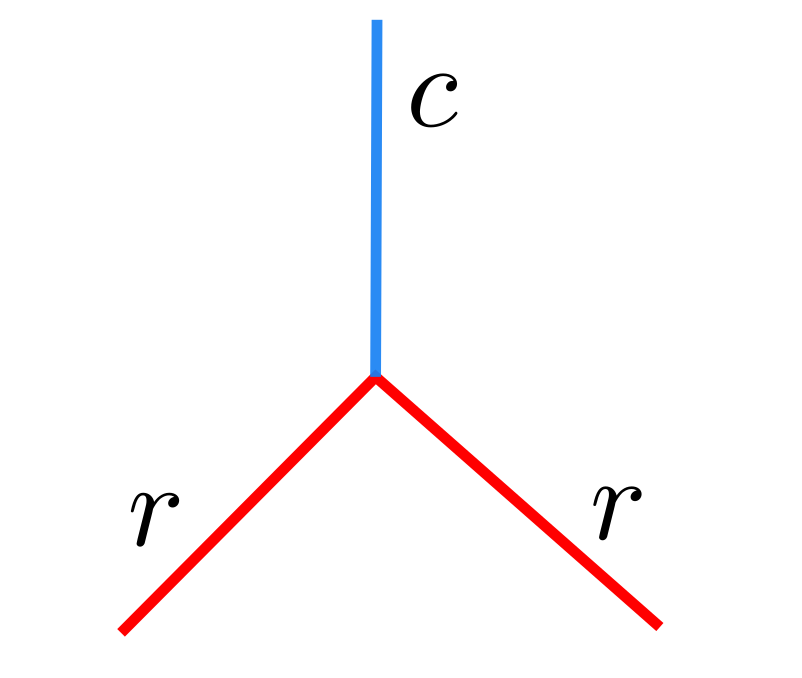}
	\caption{Local picture of a $1$-set. The convention for figures is \textbf{r}ed for \textbf{r}eal singularity and blue for complex singularity.}
\end{figure}
\begin{defn}
	Given a web with a $1$-set $(K,s)$,
	the \emph{(branched) double cover $\check Y$} along $(K,s)$ is a bifold $(Y, \tilde K^{\sfc})$ consisting of the following data:
\begin{itemize}
	\item the underlying manifold $Y$, which is the double cover of $S^3$, branched along the real locus $K^{\sfr}$;
	\item the bifold singular set $\tilde K^{\sfc}$, where $\tilde K^{\sfc}$ is the preimage in $Y$ of the complex locus $K^{\sfc}$, fixed by the covering involution $\upiota: \tilde Y \to \tilde Y$.
\end{itemize}
\end{defn}
Let $\widetilde K^{\sfr}$ be the preimage of $K^{\sfr}$ under $Y \to (S^3,K^{\sfr})$.
The covering involution $\upiota: \check Y \to \check Y$ preserves the bifold structure of $\check Y$, fixing $\widetilde K^{\sfr}$. 
The induced involution on $\tilde K^{\sfc}$, reflects components of $\tilde K^{\sfc}$ along the preimage of $\Ver(K)$, and exchanges pairs of components of $\tilde K^{\sfc}$ that contain no vertices.
We emphasize that the bifold structure of the double cover involves only $\widetilde K^{\sfc}$, while the real part $\widetilde K^{\sfr}$ is recorded in the involution $\upiota$.

To build a category whose objects are webs with $1$-sets, we colour the facets of a foam cobordism $\Sigma$ so that the closure of the union of $\sfr$-facets is a $2$-manifold, possibly with boundaries and nonorientable.
To this end, we make the following definition.
\begin{figure}[th!]
	\centering
	\includegraphics[width=0.6\linewidth]{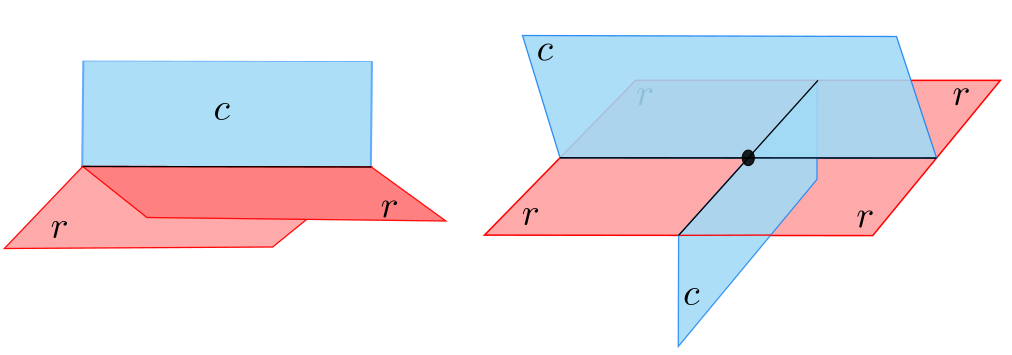}
	\caption{Local pictures around a seam and a tetrahedral point of a $1$-set.}
\end{figure}
\begin{defn}
A \emph{$1$-set of $\Sigma$} is a colouring of the set $\Facet(\Sigma)$ of facets by a $2$-element set $ \{\sfr,\sfc\}$ such that among the three facets around each seam there is precisely one $\sfc$-facet.
\end{defn}
As a consequence, at each tetrahedral point there are exactly two $\sfc$-facet, meeting only at the tetrahedral point.
Let $(\Sigma,s_{\Sigma})$ be a foam with $1$-set $s_{\Sigma}$.
Denote by $\Sigma^{\sfr}, \Sigma^{\sfc}$ the (closures of the) union of $\sfr$-facets and of $\sfc$-facets, respectively.
For practical purposes, we suppose all foams are properly embedded in punctured $S^4$'s.
\begin{defn}
Given a $1$-set $s_{\Sigma}$ of foam $\Sigma$, a \emph{(branched) double cover} $(X, \tilde \Sigma^{\sfc})$ of $(\Sigma,s_{\Sigma})$ has the underlying $4$-manifold as the double cover branched along $\Sigma^{\sfr}$, and singular at the preimage $\tilde\Sigma^{\sfc}$ in $X$ of $\Sigma^{\sfc}$.
If $\Sigma^r = \emptyset$, then the double cover is simply the trivial double cover.
\end{defn}
Due to the presence of $V_4$-singularities and absence of orientability assumptions on the facets, the self-intersection number $\Sigma \cdot \Sigma$, typically defined for smooth oriented embedded $2$-manifolds, needs to be suitably modified.
The following definition is taken from \cite[Definition~2.5]{KMweb}.

A foam $\Sigma \subset X$ can be viewed as an immersion of a surface with corners $i: \Sigma^+ \to X$, injective except on the boundary $\del \Sigma^+$. 
The map $i$ restricted to $\del \Sigma^+$ to the seams as $3$-fold coverings and takes the corners of $\Sigma^+$ to the tetrahedral points.
At each point $q$ on the seams or the tetrahedral points, the tangent spaces of the three branches of $\Sigma$ span a $3$-dimensional subspace $V_q$ of $TX$.

Let $V$ be the resulting $3$-dimensional subbundle of $i^*(TX)|_{\del \Sigma}$.
Since $V$ contains the directions tangent to $i(\Sigma^+)$, there is a $1$-dimensional subbundle $W$ of the normal bundle $N \to \Sigma^+$ to the immersion:
\[
	W \subset N|_{\del \Sigma}.
\]
While $N$ may not be orientable, it has a well-defined square $N^{[2]}$, whose orientation is canonically identified with the orientation bundle of $\Sigma^+$ via the ambient orientation of $X$.
Let $w$ be a section of $N^{[2]}|_{\del \Sigma^+}$ determined by $W$.
\begin{defn}
The \emph{self-intersection number} $\Sigma \cdot \Sigma$ is half of the \emph{relative Euler number}, which in turn is defined to be the following pairing in (co)homology with coefficients in the orientation bundle
\[
e(N^{[2]},w)[\Sigma^+,\del\Sigma^+].
\]
\end{defn}

\subsection{\emph{Real} bifolds}\label{sec:real_bifolds}
From now on the bifolds of interest will arise from double covers of bifolds in Section~\ref{subsec:bifolds_orbi}.
The local models one encounters will be only a subclass of that in Section~\ref{subsec:bifolds_orbi}, but with the addition of global involutions.
In real double covers, the local models in dimension three with covering involutions are:
\vspace{-0.5\baselineskip}
\begin{itemize}
	\setlength{\parskip}{0pt}
	\setlength{\itemindent}{-7pt}
	\item $\reals \times \reals^2/\{\pm 1\}$ where $\upiota$ acts trivially, and
	\item $\reals \times \reals^2/\{\pm 1\}$ where $\upiota$ acts by $(-1) \times \Id$.
\end{itemize}
\vspace{-0.5\baselineskip}
In dimension four, we have the following local models:
\vspace{-0.5\baselineskip}
\begin{itemize}
	\setlength{\parskip}{5pt}
	\setlength{\itemindent}{-7pt}
	\item $\reals^2 \times (\reals^2/\{\pm 1\})$  and $\upiota$ is trivial;
	\item $\reals^2 \times (\reals^2/\{\pm 1\})$ and $\upiota (x_0,x_1,x_2,x_3) = (+x_0,-x_1,-x_2,+x_3)$;
	\item $(\reals^3/V_4) \times \reals$, where $V_4$ is generated by two commuting involutions $\iota_0,\iota_1$ on $\reals^3$:
	\[
	\iota_0(x_0,x_1,x_2) = (-x_0,-x_1,+x_2), \quad  \iota_1(x_0,x_1,x_2)=(+x_0,-x_1,-x_2),
	\]
	 and $\upiota(x_0,x_1,x_2,x_3) = (+x_0,+x_1,-x_2,-x_3)$.
\end{itemize}
(The plus signs are for aesthetic purpose.) 
\begin{notat}
	We stress the difference between upgreek ``$\upiota$'' and ``$\iota$''. 
	The symbol ``$\upiota$'' is reserved for global involutions on bifolds, while ``$\iota$'' is reserved for local involutions in a bifold chart.
\end{notat}
\begin{defn}
	A \emph{real bifold} $(\check M,\upiota)$ is a pair of 3- or 4-orbifold $\check M$ and an involution $\upiota: \check M \to \check M$, with the local models above.
	An involution $\upiota: M \to M$ preserving the bifold structure will be referred to as a \emph{real involution}.
	Given a complex bifold vector bundle $W \to (\check M,\upiota)$,
	a \emph{real structure} on $W$ is a conjugate-linear involutive lift $\uptau:W \to W$ of $\upiota:\check M \to \check M$.
\end{defn}
If $L$ admits a real structure, then $\upiota^*L \cong \bar L$ and so $\upiota$ preserves the isotropy representation of $L$.
Certain converses are true when a $3$-bifold arises as a real double cover of $S^3$.

Recall that given a $3$-bifold $\check Y$ having no $V_4$-points, we defined $\Pic(\check Y)$ to be the group of bifold line bundles with tensor product operation.
If in addition $\check Y$ is a real bifold, then $\upiota$ acts on $\Pic(\check Y)$ by pullback.
Consider the following subgroup
\[
	\Pic(\check Y,-\upiota)=\{L \in \Pic(\check Y): \upiota^*L \cong \overline{L}\}.
\]
In particular, a line bundle in $\Pic(\check Y,-\upiota)$ is required to have isotropy data preserved by $\upiota^*$ and have first Chern class invariant under $(-\upiota^*)$.
As a consequence of the following proposition, for a line bundle over a real cover to be an element $\Pic(\check Y,-\upiota)$ it is enough to ensure $\upiota$ fix the isotropy data.
In this case, the proposition also asserts the existence of a real structure.
\begin{prop}
\label{prop:line_admit_real}
	Let $(\check Y,\upiota)$ be the real cover of $(K,s)$. Let $\tilde K_i^{\sfc}$ be a component of $\tilde K^{\sfc}$.
	A bifold line bundle $L$ on $\check Y$ admits a real structure if and only if $\upiota$ fixes the isotropy data of $L$ on $\tilde K_i^{\sfc}$'s.
	Moreover, a real structure on $L$, if exists, is unique.
\end{prop}
\begin{proof}
	The necessity of the first statement follows from the fact that conjugation preserves isotropy structures.
	Suppose isotropies of $\tilde K_i^{\sfc}$'s are fixed by $\upiota$. 
	Suppressing the tilde and superscript $\sfc$, we will write instead just ``$K_i$''.
	We construct real structures on the orbifold line bundle $E_i$'s defined locally around $ K_i$ in Section~\ref{subsec:cohom_line_bundles} as follows.
	Since $s$ is a $1$-set, either $K_i$ is preserved by $\upiota$ setwise while fixing two points, or $\upiota(K_i) = K_j$ for some $j \ne i$.
	In the first case, we suppose the tubular neighbourhood $N_k$ of $K_k$ was chosen to be $\upiota$-invariant, and define a real structure $\uptau_k$ by
	\[
	(\lambda,\theta, z) \mapsto (\bar{\lambda},-\theta, \bar z),
	\]
	and conjugation on the complement of $N_k$.
	In the second case, for the tensor product $E_i \otimes E_j$, we choose a neighbourhood $N_i$ of $K_i$ and set $N_j$ to be the $\upiota$-image of $N_i$. 
	Gluing the trivial bundle on $\check Y \setminus (K_i \cup K_j)$ as before, we ensure there is an identification of the fibres of $E_i \otimes E_j$ over $N_i \cup N_j$.
	We then define a real structure $\uptau_{ij}$ by declaring $\uptau_{ij}$ be the conjugation on the fibres.
	Any orbifold line bundle $L$ whose isotropy type is fixed by $\upiota$ can be written as a tensor product
	\[
		L_{smooth} \otimes E_{k_1} \otimes \dots \otimes E_{k_m} \otimes (E_{i_1} \otimes E_{j_1}) \otimes \dots \otimes  (E_{i_{\ell}} \otimes E_{j_{\ell}})
	\]
	of a smooth line bundle $L_{smooth}$ over $Y$, some $E_k$'s supported on components fixed by $\upiota$, and product of pairs of $E_i$ and $E_j$.
	By \cite[Lemma~3.5]{ljk2022}, the bundle $L_{smooth}$ admits a unique real structure. 
	Since the rest of the factors also have real structures, the product $L$ also supports a real structure.
	
	Moreover, since $E_i$'s are trivial outside a neighbourhood of $K^{\sfc}$,  the arguments \cite[Section~3]{ljk2022} apply to $N_k$ and $N_{i} \cup N_{j}$ to show the uniqueness of $\uptau_k$ and $\uptau_{ij}$.
	Since the real structure on $L_{smooth}$ is also unique, $L$ must also have a unique real structure.
\end{proof}

\section{Real Seiberg-Witten theory with bifold singularities}
\label{sec:bifoldspinc}
Suppose $\check M = (M,\Sing)$ is a 3- or 4-orbifold having the singularity types of a real bifold as in Section~\ref{sec:real_bifolds}.
In other words, the 3-bifolds will have no $V_4$-points and the cobordisms between them have no $V_8$-points.
We will define (real) bifold \spinc structures using spinor bundle and Clifford multiplication as in \cite{KMbook2007} via explicit local models.
Bifold \spinc structures can be alternatively defined as lifts of the tangent bundle $T^*\check M$ to orbifold $\Spin^{c}(n)$-bundle, as in e.g. \cite{Baldridge2001}.

Let $S \to \check M$ be a bifold hermitian rank-$2$ bundle.
Suppose $U \subset M$ is a bifold chart around a $\mathbb Z_2$-point and  $\iota: \tilde U \to \tilde U$ be the involution induced by the quotient map $\phi: \tilde U \to U$.
There is a smooth hermitian bundle $\tilde S \to \tilde U$, and an involutive complex-linear bundle isomorphism $\tau:\tilde S \to \tilde S$.
We require the eigenvalues of $\tau$ over $\phi^{-1}(\Sing \cap U)$ to be distinct.
\begin{defn}
A \emph{bifold \spinc structure} $\fraks$ over $\check Y$ is a pair $(S,\rho)$ of a bifold rank-$2$ hermitian  bundle $S$ as above and a Clifford multiplication $\rho: T\check Y \to \End(S)$ with the following properties.
Over a bifold chart $\tilde U \to U$ containing a $\bbZ_2$-point, the Clifford multiplication is given by an orientation-preserving bundle isometry $\tilde \rho: T\tilde U \to \su(\tilde S) \subset \End(\tilde S)$.
Here, $\su(\tilde S)$ is the subbundle of traceless, skew-adjoint endomorphism equipped with the inner product $\frac12\Tr(a^*b)$. 
Also, for any section $\Psi$ of $\tilde S$ and any vector field $\xi$ over $\tilde U$,
\[
	\tau(\tilde\rho(\xi) \Psi) = \tilde\rho(\iota_* \xi)\tau(\Psi).
\]
If a chart $U$ contains no $\bbZ_2$-points, then we require $\rho: T U \to \su(S)$ to be an orientation-preserving bundle isometry only. 
\hfill $\diamond$
\end{defn}
\begin{rem}
Choose a bifold chart  $\tilde U \to U$ and an orthonormal frame $(e_1,e_2,e_3)$ of $T\tilde U$ at a $\bbZ_2$ point $y \in \check Y$, such that $\iota_*(e_1,e_2,e_3) = (+e_1,-e_2,-e_3)$.
There is a basis of $\tilde S_{\tilde y}$ such that the Clifford multiplication can be represented by the Pauli matrices $\sigma_i$
\begin{equation*}
	\rho(e_1) =  \sigma_1 = \begin{bmatrix}
		i & 0\\
		0 &-i
	\end{bmatrix},\quad 
	\rho(e_2) = \sigma_2 = \begin{bmatrix}
		0 & -1\\
		1 & 0
	\end{bmatrix}, \quad 
	\rho(e_3) = \sigma_3 = \begin{bmatrix}
		0 & i\\
		i & 0
	\end{bmatrix}.
\end{equation*}
Over $\iota$-fixed points, there are two possibilities for $\tau$ as a matrix in $\U(2)$:
\[
	\begin{bmatrix}
		\pm 1 & 0\\
		0 & \mp 1
	\end{bmatrix}.
\]

\end{rem}
In dimension four, a spinor bundle $S_X$ is a sum $S^+ \oplus S^-$ of two rank-$2$ hermitian bundles.
We make the following assumption about a $V_4$-point.
Let $\tilde U \to U$ be the $V_4$-quotient map of a bifold chart $U$ and let $\tilde S_{\tilde U} = \tilde S^+ \oplus \tilde S^- \to \tilde U$ be the corresponding spinor bundle.
Then $V_4$ acts on $
\tilde U$ by commuting involutions $\iota_0,\iota_1:\tilde U \to \tilde U$. 
The bundle representation is given by two commuting, involutive lifts $\tau_0,\tau_1: \tilde S^{\pm} \to \tilde S^{\pm}$ of $\iota_0,\iota_1$, respectively.
\begin{defn}
	\label{defn:bifoldspinc4}
A \emph{bifold \spinc structure} over a bifold $\check X$ is a pair $(S_X,\rho)$ of a rank-$4$ bifold hermitian bundle $S_X = S^+ \oplus S^-$ and a Clifford multiplication $\rho: T\check X \to \Hom(S_X, S_X)$ with the following properties.
Let $x$ be a point in $\check X$. Then there exists a bifold chart $\tilde U \to U$, where the Clifford multiplication is represented by $\tilde{\rho} : T\tilde{U} \to \Hom(\tilde S_X,\tilde S_X)$.
Moreover, there is an orthonormal frame $(e_0,\dots,e_3)$ of $T\tilde U_{\tilde x}$ and a basis of $\tilde S_{X,\tilde x}$ such that
\begin{equation*}
	\tilde\rho(e_0) = 
	\begin{bmatrix} 0  & -I_2 \\ I_2 & 0\end{bmatrix},\quad
	\tilde \rho(e_i) = 
	\begin{bmatrix} 0  & -\sigma_i^* \\ \sigma_i & 0\end{bmatrix},
\end{equation*}
where $i = 1,2,3$ and $I_2$ is the identity matrix.
The action of the stabilizer $H_x$ on $\tilde S$ obeys the following compatibility conditions:
\begin{enumerate}[(i)]
	\item If $x$ is a smooth point, then there is further requirement.
	\item If $x$ is a $\mathbb Z_2$-point, then the $\bbZ_2$-action is given by a complex-linear involution $\tau: \tilde S_{\tilde U}\to \tilde S_{\tilde U}$ such that
\[
\tau(\tilde\rho(\xi) \Phi) = \tilde\rho(\iota_* \xi)\tau(\Phi),
\]
for any vector $\xi$ on $\tilde U$ and any spinor $\Phi \in \Gamma(\tilde S^{\pm})$.
Over the fibre $\tilde S_{\tilde U,\tilde x}$, the lift $\tau$ has distinct eigenvalues.
	\item Assume $x$ is a $V_4$-point. 
	There are two commuting involutions $\iota_0,\iota_1: \tilde U \to \tilde U$, which lift to commuting two complex-linear unitary involutions 
	$\tau_0,\tau_1: \tilde S_{\tilde U} \to \tilde S_{\tilde U}$. 
	Both $\tau_i$ are compatible with $\tilde \rho$, in the sense that
\[
	\tau_i(\tilde\rho(\xi) \Phi) = \tilde\rho(\iota_* \xi)\tau_i(\Phi),
\]
for any vector field $\xi$ on $\tilde U$ and any spinor $\Phi \in \Gamma(\tilde S^{\pm})$.
Over the fibre $\tilde S_{\tilde U,\tilde x}$, both $\tau_0$ and $\tau_1$ have distinct eigenvalues.
\end{enumerate}
\end{defn}
When a $3$-bifold $\check Y$ has no $V_4$-point, \cite[Theorem~6 and 7]{Baldridge2001} states:
\begin{thm}
\label{thm:Bal_homogeneous}
	For any 3-bifold $(Y,\tilde K)$ such that $\tilde K$ is a link, there exists a bifold \spinc structure.
	Moreover, the difference of two spin\textsuperscript{c} structure is a bifold line bundle.
	In other words, the set of \spinc structures is a principal homogeneous space (i.e. a torsor) over the Picard group $\Pic(\check Y)$ of bifold line bundles.
	\hfill \qedsymbol
\end{thm}

\subsection{\emph{Real} bifold \spinc structures}
\label{subsec:realbifoldstr}
Let $\check M = (M,\Sing^{\sfc})$ be a 3- or 4-bifold with the same assumption as in the beginning of Section~\ref{sec:bifoldspinc}.
Let $\upiota: (M,\Sing^{\sfc}) \to (M,\Sing^{\sfc}) $ be a (global) involution that preserves the bifold structure.
Denote the fixed point set in $M$ of $\upiota$ by $\Sing^{\sfr}$.
Note that $\Sing^r$ can be nonempty,
and $M$ can be disconnected.
We are particularly interested in the case when $M$ comprises two identical copies of a connected bifold.

\begin{defn}
	A \emph{real bifold \spinc structure} over a real 3-bifold $(\check Y,\upiota)$ is a pair $(\fraks, \uptau)$ of a bifold \spinc structure $\mathfrak s = (S,\rho)$ and a complex anti-linear involutive lift $\uptau:S \to S$ of $\upiota$ satisfying the following conditions.
	For any bifold chart $\tilde U \to U$, where the $\mathbb Z_2$-action $\tau: \tilde S \to \tilde S$ and the Clifford multiplication $\tilde \rho$ are defined as in Section~\ref{sec:bifoldspinc}, we have
	\[
	\tau \uptau = \uptau \tau.
	\] 
Moreover,
	for any vector field $\xi$ over $\tilde U$, and any spinor $\Psi$, we have
	\[
	\uptau(\tilde\rho(\xi) \Psi) = \tilde\rho(\upiota_* \xi)\uptau(\Psi).
	\]
A \spinc isomorphism $u:(S,\rho) \to (S',\rho')$ is an \emph{isomorphism of real bifold \spinc structures} if  $u\uptau = \uptau'u$.
An anti-linear lift $\uptau$ satisfying the above assumptions is \emph{compatible} with $\fraks$.
\end{defn}
\begin{notat}
The ``$\uptau$'' denotes an anti-linear global lift,
while  ``$\tau$'' always denotes a locally defined lift.
\end{notat}

\begin{defn}
	A \emph{real bifold \spinc structure} over a real $4$-bifold $(\check X,\upiota_X)$ is a pair $(\fraks_X, \uptau_X)$ of a bifold \spinc structure $\mathfrak s_X = (S_X,\rho)$ and a complex anti-linear involutive lift $\uptau_X:S_X \to S_X$ of $\upiota_X$ that respects the decomposition $S_X = S^+ \oplus S^-$, and satisfies the following compatibility conditions.
	\begin{itemize}
		\itemindent=-13pt
	\item In a $\bbZ_2$-chart $\tilde U$, with associated $(\tilde \rho, \tau)$, we have $\tau \uptau_X = \uptau_X \tau$, and for any vector field $\xi$ and any spinor $\Phi$,
	\[
	\uptau_X(\tilde\rho(\xi) \Phi) = \tilde\rho((\upiota_X)_* \xi)\uptau_X(\Phi).
	\]
	\item In a $V_4$-chart $\tilde U$, with associated ($\tilde \rho, \tau_0$, $\tau_1$), we have 
	$
	\uptau_X(\tilde\rho(\xi) \Phi) = \tilde\rho((\upiota_X)_* \xi)\uptau_X(\Phi)
	$ as in Definition~\ref{defn:bifoldspinc4},
	and
	\[\tau_i \uptau_X = \uptau_X \tau_i.\]
	\end{itemize}
A \spinc isomorphism $u:(S_X,\rho) \to (S_X',\rho')$ is an \emph{isomorphism of real bifold \spinc structures} if  $u\uptau_X = \uptau_X'u$.
An anti-linear lift $\uptau_X$ satisfying the above assumptions is said to be \emph{compatible} with $\fraks$.
\end{defn}
Given a real $3$- or $4$-bifold $(\check M, \upiota)$,  denote by $\SpincR(\check M, \upiota)$ the equivalence classes of real bifold \spinc structures.
A line bundle $L$ with a real structure $\tau^L$ acts on $\SpincR(\check M, \upiota)$ via tensor product $(\fraks, \uptau) \mapsto (\fraks \otimes L, \uptau \otimes \uptau^L)$.
\begin{prop}
	Let $(\check Y,\upiota)$ be the real double cover of a web with $1$-set $(K,s)$.
	Then there exists at least one real \spinc structure on $(\check Y,\upiota)$.
	Furthermore, $\SpincR(\check Y, \upiota)$ is a principal homogeneous space over the group $\Pic(\check Y,-\upiota)$.
\end{prop}
\begin{proof}
To see $\SpincR(\check Y, \upiota)$ is a principal homogeneous space over $\Pic(\check Y,-\upiota)$, fix a pair of real \spinc structures $(\fraks,\uptau)$ and $(\fraks',\uptau')$.
By Theorem~\ref{thm:Bal_homogeneous}, the \spinc structures $\fraks$ and $\fraks'$ must differ by a bifold line bundle $L$.
In addition, $\uptau$ and $\uptau'$ defines difference real structure $\uptau^L$ on $L$.
By Proposition~\ref{prop:line_admit_real}, the real structure $\uptau^L$ is uniquely determined by $L$, which lies in $\Pic(\check Y,-\upiota)$.

To show there exists a real \spinc structure, cut $\check Y$ along an $\upiota$-invariant tubular neighbourhood $\textsf{Nbhd}(\tilde K^{\sfc})$ of the complex locus $\tilde K^{\sfc}$ to obtain 
\[
	\check Y = Y_+ \cup_{T} Y_-,
\]
where $Y_+ = \check Y \setminus \textsf{Nbhd}(\tilde K^{\sfc})$ and $Y_- \cong \check D \times \tilde K^{\sfc}$, where $\check D$ is the bifold disc singular at the origin.
Let $\fraks_0$ be the unique torsion \spinc structure on $T$.
As in \cite[Lemma~2.2]{ChenWM2012SW3orbifolds}, a bifold \spinc structure $\fraks$ on $\check Y$ can be described by \spinc structures $\fraks_{\pm}$ on $Y_{\pm}$, glued using isomorphisms $h_{\pm}:\fraks_{\pm}|_T \to \fraks_{0}$.

On $Y_+$ we choose a real \spinc structure $(\fraks_+,\uptau_+)$, induced by restriction of an $\upiota$-invariant spin structure on $\check Y$, which exists because $Y$ is a double cover of $S^3$.
Over each component of $Y_-$, there are two bifold \spinc structures corresponding to spinor bundles $\underline{\C} \oplus K_{\check D}$ and $K_{\check D} \oplus \underline{\C}$, where $K_{\check D} \to \check D$ is the bifold canonical bundle.
Choose a \spinc structure $\fraks_-$ on $Y_-$ for which, by arranging the \spinc structures on $\upiota$-pairs of components to agree, there is a compatible real structure $\uptau_-$.
Each real structure $\uptau_{\pm}$ induces a real structure on $
\fraks_0$ as $\sigma_{\pm} = h_{\pm} \circ \uptau_{\pm} \circ h_{\pm}^{-1}$.
The difference $\sigma_-^{-1} \sigma_+ $ is an $\upiota$-invariant map from $T$ to $S^1$.
But since $T/\upiota$ is a union of tori (when $\upiota$ swaps components) and 2-spheres (when $\upiota$ fixes components), the symmetrizing homomorphism
\[
(1+\upiota^*):H^1(T,\bbZ) \to H^1(T,\bbZ)^{\upiota^*} 
\]
is surjective.
In particular, $\sigma_+ \sigma_-^{-1} = \upiota^*(u)u$
 for some $u:T \to S^1$. 
Setting $\tilde h_- = uh_-: \fraks_-|_T \to \fraks_0$ to be the new gluing isomorphism on $Y_-$, we obtain a global real structure on $\fraks$
from $\uptau_-$ and $\uptau_+$.
This argument is reminiscent of \cite[Section~3.1]{ljk2022}.
\end{proof}
\subsection{Seiberg-Witten configuration spaces}
\label{subsec:SW_config}
\begin{defn}
Let $\check Y$ be a 3-bifold equipped with a metric $\check g$ and a bifold \spinc structure $(S,\rho)$.
A \emph{(bifold) \spinc connection} $B$ is a unitary bifold connection on $S$ such that, for any vector field $\xi$ and any section $\Psi$ of $S$, we have
\[
	\nabla_B(\rho(\xi)\Psi) = \rho(\nabla \xi)\Psi + \rho(\xi)\nabla_B \Psi.
\]
Similarly, assume $\check X$ is a $4$-bifold equipped with a metric $\check g_X$ and a \spinc structure $(S_X,\rho)$. 
Then a \emph{(bifold) \spinc connection} $A$ is a unitary bifold connection on $S_X$ that preserves the decomposition $S_X = S^+ \oplus S^-$. Furthermore, for any vector field $\xi$ and any section $\Phi$ of $S^{\pm}$, we have
\[
	\nabla_A(\rho(\xi)\Phi) = \rho(\nabla \xi)\Phi + \rho(\xi)\nabla_A \Phi.
\]
\end{defn}

Let $\check M$ be a $3$- or $4$-bifold (possibly with boundary) as above. 
Suppose $\fraks$ is a \spinc structure on $\check M$
 and $W$ is $S$ or $S^+$.
Denote by $\underline{\calA}(\check M, \mathfrak s)$ the space of bifold connections.
Then $\underline{\calA}(\check M, \mathfrak s)$ is, by Schur's lemma, an affine space modelled on  $\Omega^1(\check M; i\reals)$, consisting of imaginary (bifold) $1$-forms on $\check M$.
\begin{defn}
The \emph{bifold Seiberg-Witten configuration space} $\underline{\calC}(\check M,\mathfrak s)$ is the product of  $\underline{\calA}(\check M, \mathfrak s)$ and the space $\Gamma(W)$ of (bifold) sections of the spinor bundle $W$.	
\end{defn}
Assume $\upiota:\check M \to \check M$ is an involutive isometry.
Let $\uptau$ be a real structure on $W$ compatible with $\mathfrak s$.
In other words, $(\fraks, \uptau)$ is a real \spinc structure.
Then $\uptau$ induces an involution on $\underline{\calA}(\check M, \mathfrak s)$ by pullback $\nabla_A \mapsto \uptau\nabla_A(\uptau \cdot )$. 
Once a $\uptau$-invariant reference \spinc connection $A_0$ is chosen,  $\uptau$ acts on $\underline{\calA}(\check M, \mathfrak s)$ by
\[
	A_0 + a \otimes 1 \mapsto A_0 - \upiota^*a \otimes 1.
\]
This action induces an involution on the real vector space of $1$-forms $\Omega^1(\check M; i\reals)$.
The set of $\uptau$-invariant \spinc connections will be denoted by $\calA(\check M,\fraks,\uptau)$.
Over the complex vector space $\Gamma(\check M;W)$ of bifold sections, there is an anti-linear involution $\uptau$:
\[\uptau(\Phi)_x = \uptau(\Phi_x)_{\upiota(x)}.
\]
\begin{defn}
Given a real bifold $(\check M,\upiota)$ and a real \spinc structure $(\fraks,\uptau)$,
the \emph{real bifold Seiberg-Witten configuration space} $\calC(\check M,\fraks,\uptau)$ is the product of $\uptau$-invariant \spinc connections and spinors:
\[
\calC(\check M,\fraks,\uptau) = \calA(\check M, \fraks, \uptau) \times \Gamma(\check M;W)^{\uptau}.
\]
\end{defn}
\begin{notat}
	In this paper, the symbol ``$\fraks$'' will often be dropped from $(\fraks,\uptau)$, since ``$\uptau$'' implicitly depends on ``$\fraks$''.
	For example, instead of writing $\calA(\check M, \fraks,\uptau)$, we will simply write $\calA(\check M,\uptau)$.
\end{notat}
\begin{rem}
	The class of real \spinc connections on determinant line bundles corresponds to the bifold $\O(2)$-connections in Daemi's plane Floer homology for webs \cite{DaemiPrivateComm}.
\end{rem}
The \emph{ordinary bifold gauge group} $\underline{\mathcal G}(\check M)$ is the group of $S^1$-valued \emph{bifold} maps over $\check M$.
In particular, the smoothness of a gauge transformation is defined using its smoothness at local uniformizing chart.
The \emph{continuous} bifold gauge transformations on $\check M$ can be identified with the \emph{continuous} gauge transformations on the underlying manifold $M$.

The \emph{real gauge group} $\mathcal G(\check M,\upiota)$ is the subgroup of $\underline{\mathcal G}(\check M)$ consisting of $\upiota$-(anti)-invariant gauge transformations:
\[
	\mathcal G(\check M,\upiota) = \{u:\check M \to S^1 \bigg\vert u(\upiota(x)) = u^{-1}(x)\}.
\]
An element $u \in \mathcal G(\check M,\upiota)$ 
acts on $\calC(\check M,\uptau)$ by $u(A,\Phi) = (A - u^{-1}du, u \Phi).$
We define $\mathcal B(\check M, \uptau)$ to be the space of $\mathcal G(\check M,\upiota)$-equivalence classes of configurations $\mathcal C(\check M,\uptau)/\mathcal G(\check M,\upiota)$.

A configuration $(A,\Phi) \in \calC(\check M, \uptau)$ is \emph{reducible} if $\Phi = 0$. 
If $\check M$ is connected, then $(B,0)$ has a $\{\pm 1\}$-stabilizer by the gauge group action.
If $\Sing^{\sfr} = \emptyset$ and $\check M$ contains two connected components, then the stabilizer is isomorphic to $S^1$, consisting of locally constant functions on $\check M$ whose values are anti-invariant under $\upiota$.

Over a \spinc $3$-bifold $(\check Y,\fraks)$
the \emph{Chern-Simons-Dirac (CSD) functional} is a $\reals$-valued function on $\underline{\mathcal C}(\check Y, \fraks)$, defined as
\[
	\calL(B,\Psi) = 
	\frac18 \int_{\check Y} (B^t-B_0^t) \wedge (F_{B^t} + F_{B^t_0}) + 
	\frac12 \int_{\check Y} \la D_B\Psi,\Psi\ra d\vol_{\check Y},
\]
where $B_0$ is a $\uptau$-invariant reference \spinc connection, and $B^t$ is, for example, the induced connection on the determinant line bundle $\det(S)$.
The CSD functional is invariant under the identity component of $\underline{\calG}$, but only circle-valued in general:
\[
	\calL(B,\Psi) - \calL(u(B,\Psi))
	= 
	2\pi^2([u] \cup c_1(S))[\check Y].
\]
The $L^2$-gradient of $\calL$ is the 3-dimensional Seiberg-Witten operator, taking values in $\Gamma(\check Y;iT^*\check Y \oplus S)$, is given by
\[
	\grad \calL
	= 
	\left( (\frac12 \ast F_{B^t} + \rho^{-1}(\Psi\Psi^*)_0) \otimes 1_S,
	D_B \Psi
	\right),
\]
where ``$*$'' is the $3$-dimensional Hodge star operator. 
By the compatibility of $\uptau$ with $\fraks$, the gradient $\grad \calL$ commutes with $\uptau$, so over real configuration space $\calC(\check Y,\uptau)$, it can be viewed as a section of a trivial vector bundle $\calT(\check Y,\uptau) \to \calC(\check Y,\uptau)$, having fibre
\[
\Gamma(\check Y;iT^*\check Y \oplus S)^{\uptau}.
\]

Suppose $(\check X,\fraks_X)$ is a \spinc 4-bifold.
Let $A$ be a \spinc connection and $\Phi$ be a section of $S^+$.
The $4$-dimensional Seiberg-Witten operator is 
\[
\frakF(A,\Phi) = \left(\frac12 \rho(F_{A^t}^+) - (\Phi\Phi^*)_0, D^+_A\Phi\right).
\]
Here, $\rho: \Lambda^+ \to \su(S^+) \subset \End(S^+)$ is the Clifford multiplication, and $(\Phi\Phi^*)_0$ is the endomorphism $\Phi\Phi^* - \frac12|\Phi|^2$.
Suppose $\upiota_X: \check X \to \check X$ is a real involution and $\uptau_X: S_X \to S_X$ is a real structure compatible with $\fraks_X$.
The Seiberg-Witten operator is
is equivariant with respect to $\uptau_X$ and therefore descends to a map
\[
	\frakF: \calC(\check X,\uptau_X) \to \Gamma(\check X;i\su(S^+) \oplus S^-)^{\uptau_X}.
\]
The Seiberg-Witten operators can be viewed as sections of vector bundle $\calV(\check X,\uptau_X) \to \calC(\check X,\uptau_X)$, where $\calV(\check X,\uptau_X)$ is a trivial bundle having fibre
\[
	\Gamma(\check X;i\su(S^+)\oplus S^-)^{\uptau_X}
\]

\subsection{Disconnected real bifolds}
\label{subsec:case_real_loci}
Over a real bifold $(\check M,\upiota)$ where $\upiota$ exchanges two connected components of $\check M$, the $\uptau$-invariant objects can be identified with objects on a (non-real) bifold.

Let $(\check M,\upiota)$ be a real bifold such that $\check M = \check M_1 \sqcup \check M_2$ is the disjoint union of two copies of the same bifold $M_o$. 
We fix an identification of $\check M_{i} \cong \check M_o$, for $i = 1,2$.
Suppose the involution $\upiota$ exchanges the two component in a way that it acts as identity on $M_o \to M_o$.
This involution will referred to as the \emph{swapping involution} for $\check M_{i} \cong \check M_o$.
The prototype is the real double cover of a link $K$ with a $1$-set such that every component is coloured by $\sfc$. 

The set of real bifold \spinc structures over $\check M$ can be naturally identified with the set of bifold \spinc structures on $\check M_o$ as follows.
Suppose $(\fraks,\uptau)$ is a real \spinc structure on $\check M$.
Let 
$\fraks_i = (W_i,\rho_i)$ be the restriction of $\fraks_i$ onto $\check M_i$.
We view each $\fraks_i$ over $\check M_i$ as a \spinc structure over $\check M_o$ via $\check M_{i} \cong \check M_o$.
Since $\fraks$ admits a real structure, $\fraks_1 = \bar{\fraks}_2$ on $\check M_o$.
Then $W_1$ and $W_2$ can be set to be $W_o \to \check M_o$ as having the same underlying real vector bundle, but are equipped with opposite complex structures.
Furthermore, the lift $\uptau_0$ of $\upiota$ on $\fraks$, induced by the identity map on $W_o$, is a compatible real structure.
By the uniqueness arguments in \cite[Section~3.1]{ljk2022},
there is at most one compatible real structure on $\fraks$, so a posteriori, $\uptau$ is isomorphic to $\uptau_0$.
Consequently, there is a one-to-one correspondence between real bifold \spinc structures on $\check M$, and bifold \spinc structures on $\check M_o$.

The reasoning also gives an identification of the bifold Seiberg-Witten configuration space on $(\check M_o, \fraks_1)$ with the real bifold configuration space on $(\check M, \fraks, \uptau)$.
Indeed, a $\uptau$-invariant configuration $(A,\Phi)$ on $\check M$ is precisely determined by its restriction 
\[(A|_{M_1}, \Phi|_{M_1}).\] 
Similarly, there is a canonical isomorphism of the real gauge group on $\check M$ and bifold gauge group on $\check M_o$:
\[
\underline{\calG}(\check M_o) \cong \calG(\check M, \upiota).
\]

\section{Analysis of Seiberg-Witten trajectories}
\label{sec:analysis_of_sw_traj}
The setup of real Seiberg-Witten moduli spaces (without bifold singularities) in terms of invariant configurations was written down in \cite{ljk2022}, which are built on of their counterparts in \cite{KMbook2007}.
The results in this section are proved by inserting the adjective ``bifold'' everywhere to the proofs in \cite{KMbook2007} or \cite{ljk2022}.

Let $(\check M,\upiota)$ be a real $3$- or $4$-bifold possibly with boundary, equipped with an $\upiota$-invariant bifold metric $\check g$.
We set the Sobolev spaces $L^p_k(\check M)$ to be the completion of the $C^{\infty}$-bifold functions on $\check M$, with respect to $\|\cdot\|_{L^p_k}$ defined by the $\check g$-Levi-Civita connection.
Assume $E \to \check M$ is a bifold vector bundle with an inner product.
We introduce Sobolev norms $L^p_k$ on $\Gamma(E)$ in the same manner as \cite[Section~5.1]{KMbook2007}, using the Levi-Civita connection and a $C^{\infty}$-bifold connection on $E$. 
The Sobolev space $L^p_k(\check M;E)$
is the resulting completion of the $C^{\infty}$-bifold sections of $E$.

Fix a real \spinc structure $(\mathfrak s, \uptau)$ on $(\check M,\upiota)$.
Suppose $k$ is an integer or a half integer, such that $2(k+1) > \dim M$.
We define the spaces of real connections $\calA_k(\check M,\mathfrak s,\uptau)$ as modelled on $L^2_k(\check M; iT^*\check M)$.
Similarly, $\calC_k(\check M,\mathfrak s, \uptau)$ and $\mathcal G_k(\check M, \upiota)$ are the respective Sobolev completions of $\calC(\check M,\mathfrak s, \uptau)$ and $\mathcal G(\check M, \upiota)$.
The assumption on $k$ ensures $\mathcal G_{k+1}$ is a Hilbert Lie group.

\subsection{Blowing up}
In~\cite{KMbook2007}, Kronheimer and Mrowka applied blow-ups to deal with reducibles, converting the problem of computing equivariant homology to Morse theory on manifolds with boundary.
We review the basic notions below.

In this section, $W$ will stand for the spinor bundle $S$ if $\dim M = 3$ and $S^+$ if $\dim M= 4$.
The $\sigma$-blowup $\calC^{\sigma}_k$  is defined as
\[
	\calC^{\sigma}_k(\check M, \uptau)
	=
	\{(A,s,\phi) \in  \mathcal A(\check M,\uptau) \times \reals_{\ge 0} \times L^2_k(\check M;W)^{\uptau}:\|\phi\|_{L^2(M)} = 1\}.
\]
In other words, $\calC^{\sigma}_k$ replaces the reducible locus of $\calC_k$ with the $L^2$-unit sphere of sections of $W$.
The natural projection map
\[
	\boldsymbol{\pi}:\calC^{\sigma}_k(\check M,  \uptau) \to \calC_k(\check M,  \uptau), \quad 
	\boldsymbol{\pi}(A,s,\phi) =(A,s\phi),
\]
has fibre a $L^2$-unit sphere in $L^2_k(\check M;W)^{\uptau}$ over an element of the reducible locus $ \mathcal A(\check M,\uptau) \times \{0\}$, and is injective otherwise.
Since $2(k+1) > \dim \check M$, the $\sigma$-blowup $\calC^{\sigma}_k(\check M, \uptau)$ is a Hilbert manifold on which the gauge group $\mathcal G_{k+1}$ acts smoothly and freely.
Denote by $\calB^{\sigma}(\check M,\uptau)$ the quotient $\mathcal C^{\sigma}_k(\check M,\uptau)/\mathcal G_{k+1}(\check M,\upiota)$.

For $j \le k$, we define the \emph{$j$-th tangent bundle} $\mathcal T^{\sigma}_j \to \calC^{\sigma}_k(\check M,\uptau)$ by declaring its fibre at $(A_0,s_0,\phi_0)$ to be
\[
	\left\{
	(a,s,\phi) \in L^2_j(\check M; iT^*\check M)^{-\upiota^*} \times \reals \times L^2_j(\check M;W)^{\uptau} \bigg| \Re \la \phi_0, \phi \ra = 0
	\right\}.
\]
When $j = k$, this recovers tangent bundle of $\calC^{\sigma}_k$ in the usual sense.
The tangent bundle admits a smooth decomposition~\cite[Proposition~9.3.5]{KMbook2007}:  
\[\mathcal T^{\sigma}_{j,\upgamma} = \calJ^{\sigma}_{j,\upgamma} \oplus \calK^{\sigma}_{j,\upgamma}\]
as follows.
Given a configuration $\upgamma \in \calC^{\sigma}_k(\check M,\uptau)$,
we linearize the gauge group action to obtain 
a map from $T_e\calG_{k+1}(\check M,\upiota) = L^2_{k+1}(\check M;i\reals)^{-\upiota^*}$ to $\mathcal T^{\sigma}_{k,\upgamma}$, given by
\begin{equation*}
	\mathbf{d}^{\sigma}_{\upgamma}(\xi) = (-d\xi, 0, \xi\phi_0).
\end{equation*}
Extending $\mathbf{d}^{\sigma}_{\upgamma}$ to lower regularities, we define $\mathcal T^{\sigma}_{j,\upgamma}$ as the image of $\mathbf{d}^{\sigma}_{\upgamma}$ in $\mathcal T^{\sigma}_{j,\upgamma}$.
The complementary subspace $\mathcal K_{j,\upgamma}^{\sigma}$ to the gauge action is
\begin{equation*}
	\mathcal K_{j,\upgamma}^{\sigma}
	=
	\{(a,s,\phi) \big| -d^*a + i\Re\langle i\Phi_0, \phi \rangle = 0, \ \langle a, \mathbf{n} \rangle = 0 \text{ at } \del \check M,\text{ and } \Re\langle i\phi_0,\phi \rangle_{L^2(\check M)} = 0\},
\end{equation*} 
where $\mathbf{n}$ is the outward normal at the boundary $\del \check M$, if $\del \check M \neq \emptyset$.
The subspace $\mathcal K_{j,\upgamma}^{\sigma}$ is not an orthogonal complement of $\calJ^{\sigma}_{j,\upgamma}$ with respect to any natural metric.

At each configuration $\upgamma$, a slice $\calS^{\sigma}_{k,\upgamma}$ to the $\calG_{k+1}$-action can be constructed as in \cite[Section~5.6]{ljk2022}, \cite[Defintion~9.3.6]{KMbook2007}. 
Such slices are closed Hilbert submanifolds of $\mathcal C^{\sigma}_k(\check M,\uptau)$, whose tangent spaces are $\calK^{\sigma}_{j,\upgamma}$.
In particular, they provide local charts of $\calB^{\sigma}(\check M,\uptau)$, and thus endow $\calB^{\sigma}(\check M,\uptau)$ with the structure of a Hilbert manifold with boundary.
See~\cite[Coroallary~9.3.8]{KMbook2007}.

There is a decomposition $\mathcal T_{j,\upgamma} = \calJ_{j,\upgamma} \oplus \calK_{j,\upgamma}$ downstairs in $\calC_k(\check M, \uptau)$ where $\calJ_{j,\upgamma}$ is defined as the image of
\[
	\textbf{d}_{\upgamma}:L^2_{j+1}(\check M;i\bbR)^{-\upiota^*} \to \calT_{j,\upgamma}, \quad
	\textbf{d}_{\upgamma}(\xi) = (-d\xi,\xi\Phi_0),
\]
while the orthogonal complement to $\calJ_{j,\upgamma}$ is 
\[
	\mathcal K_{j,\upgamma}
	=
	\{(a,s,\phi) \big| -d^*a + i\Re\langle i\Phi_0, \phi \rangle = 0, \text{ and } \ \langle a, \mathbf{n} \rangle = 0 \text{ at } \del \check M\}.	
\]

Consider the cylinder $(\check Z,\upiota) = (I \times \check Y, \Id_I \times \upiota)$ for a compact interval $I$.
Pulling back the real \spinc structure $(\fraks,\uptau)$ from $(\check Y,\upiota)$, we obtain a real \spinc structure $(\fraks_Z,\uptau)$ on $\check Z$.
The abuse of notations should not cause confuse the readers as the structures are pulled back from the $3$-bifold.
Every configuration on $(\check Z,\uptau)$ can be viewed as a path of configurations in $(\check Y,\upiota)$.
Blowing up $\calC(\check Y_t, \uptau)$ on each time slice $t \in I$,
we introduce the \emph{$\tau$-blowup} of configuration spaces on $\check Z$ as 
\[
	\calC^{\tau}_k(\check Z,\uptau) = \left\{
	(A,s,\phi) \in \mathcal A_k(\check Z,\uptau) \times L^2(I;\reals) \times 
		L^2_k(\check Z;S^+)^{\uptau}\bigg|
		s(t) \ge 0, \|\phi(t)\|_{L^2(\check Y)} = 1
	\right\}.
\]
This is not a Hilbert manifold, but a closed subspace of a Hilbert manifold $\tilde C^{\tau}_k(\check Z,\uptau)$, where $\tilde C^{\tau}_k(\check Z,\uptau)$ is defined by dropping the $s \ge 0$ condition:
\[
\tilde C^{\tau}_k(\check Z,\uptau) = \left\{
	(A,s,\phi) \in \mathcal A_k(\check Z,\uptau) \times L^2(I;\reals) \times 
		L^2_k(\check Z;S^+)^{\uptau}\bigg|
		\|\phi(t)\|_{L^2(\check Y)} = 1
	\right\}.
\]
Indeed, $\calC^{\tau}(\check Z,\uptau)$ is the fixed point set of the involution $\textbf{i}:(A,s,\phi) \mapsto (A,-s,\phi)$.
We denote the spaces of gauge equivalence classes by $\calB^{\tau}(\check Z,\uptau)$ and $\tilde{\calB}^{\tau}(\check Z,\uptau)$ analogously.
Moreover, $\tilde{\calB}^{\tau}(\check Z,\uptau)$ admits the structure of a Hilbert manifold~\cite[Corollary~9.4.4]{KMbook2007}.
The proof goes by introducing slices $\tilde{\calS}^{\tau}_{k,\upgamma}$, as in~\cite[Section~9.4]{KMbook2007}.

Suppose $(\check X,\upiota_X)$ is a real $4$-bifold and $(\fraks_X,\uptau_X)$ is a real \spinc structure.
We define $\calV^{\sigma}_j(\check X,\uptau_X) \to \calC_k^{\sigma}(\check X,\uptau_X)$ to be the pullback by $\boldsymbol{\pi}$ of the bundle $\calV_j(\check X,\uptau_X) \to \calC_k(\check X,\uptau_X)$ in Section~\ref{subsec:SW_config}.
The Seiberg-Witten operator $\frakF^{\sigma}$ on the $\sigma$-blowup is a section of $\calV_j^{\sigma}(\check X,\uptau_X)$, given by
\[
	\frakF^{\sigma}(A,s,\phi) = \left( \frac{1}{2}\rho(F^+_{A^t}) - s^2(\phi\phi^*)_0, D^+_A\phi\right).
\]
Given a cylinder $(\check Z,\upiota)$ as above,
the $\tau$-blowup version $\calV^{\tau}_j \to \calC_k^{\tau}(\check Z,\uptau)$ has fibres defined by
\[
	\calV^{\tau}_{j,\upgamma} =
	\left\{
		(\eta,r,\psi) \in L^2_k(\check Z;i\su(S^+))^{\uptau} \times L^2_j(I;\reals)\times L^2_j(\check Z;S^-)^{\uptau}\bigg| \Re\langle \check\phi(t),\check\psi(t)\rangle_{L^2(\check Y)} = 0 \text{ for all $t$}
	\right\}.
\]
The Seiberg-Witten operator $\frakF^{\tau}$ is a section of $\calV^{\tau}_j \to \calC_k^{\tau}(\check Z,\uptau)$, having the shape of a gradient flow:
\[
	\frakF^{\tau}(A,s,\phi) = \left(
	\frac{1}{2}\rho_Z(F^+_{A^t}) - s^2(\phi\phi^*)_0, \frac{d}{ds} + \Re\langle D^+_A\phi,\phi\rangle_{L^2(\check Y)}s,
	D^+_A\phi - \Re\langle D^+_A\phi,\phi\rangle_{L^2(\check Y)}\phi\right).
\]

\subsection{Homotopy types of configuration spaces}
Let $(\check M, \upiota)$ be a real bifold and $(\fraks,\uptau)$ be a real \spinc structure.
First, assume $\check M$ is connected.
If $\upgamma_0 = (B_0,0) \in \calC_k(\check M, \upiota)$ is a reducible configuration, then there is a diffeomorphism 
\[
\calG^{\perp}_{k+1} \times \calK_{k,\upgamma_0} \to \calC_k, \text{ where }
\calG^{\perp}_{k+1} =
\{e^{\xi} \in \calG_{k+1}(\check M,\upiota)| \int_{\check M} \xi = 0\}.
\]
This follows from the
uniqueness and existence of the (bifold) Coulomb-Neumann problem over $\check M$:
\begin{align*}
	\Delta \xi &= d^*a,\\
	 \la d\xi, n\ra &= \la a, n\ra \text{ at }\del \check M.
\end{align*}
The subgroup of harmonic gauge transformations $\mathcal G^h$ fits into a short exact sequence of groups 
\[
	 0 \to \{\pm 1\} \to \calG^h \to H^1(\check M;\bbZ)^{-\upiota^*} \to 0.
\]
The (non-blown-up) space $\calB_k(\check M,\uptau)$ of gauge equivalence classes satisfies the following homotopy equivalences
\[
\calB_k(\check M,\uptau) \cong \frac{\iota(\calK)}{\calG^h}
\cong
\frac{H^1(\check M;\reals)^{-\upiota^*}}{H^1(\check M;\bbZ)^{-\upiota^*}}
\cong 
\frac{H^1(M;\reals)^{-\upiota^*}}{H^1(M;\bbZ)^{-\upiota^*}}
\]
Hence 
\begin{equation}
\label{eqn:nonempty_config_hE}
	\calB^{\sigma}(\check M,\uptau) \cong \mathbb{RP}^{\infty}\times  H^1(\check M;\reals)^{-\upiota^*}/H^1(\check M;\bbZ)^{-\upiota^*},
\end{equation}
by Kuiper's theorem \cite{KUIPER196519} of contractibility of unitary groups of Hilbert spaces and the following fibration
\[
	\frac{L^2_k(\check M;S)^{\uptau} \setminus 0}{\bbZ_2} \to \calB^{\sigma}(\check M,\uptau) \to \frac{H^1(\check M;\reals)^{-\upiota^*}}{H^1(\check M;\bbZ)^{-\upiota^*}}.
\]
The homotopy equivalence~\eqref{eqn:nonempty_config_hE} is non-canonical. 
Note also that the topology of $\calB^{\sigma}(\check M,\uptau)$ does not depend on the complex locus.

Choose a splitting of the exact sequence $v:H^1(\check M;\bbZ)^{-\upiota^*} \to \calG^h$ and define $\calG^{h,o}$ to be the image of $v$.
For example, $\calG^{h,o}$ can be chosen to be harmonic gauge transformations that evaluates to $1$ at a point $y_0$.
We have $\calG(\check M,\upiota) = S^1 \times \calG^{h,o} \times \calG^{\perp}(\check M,\upiota)$.
The \emph{framed configuration space} $\calB^o_k(\check M,\uptau)$ is the quotient
\[
	\calB^o_k(\check M,\uptau) = \calC_k(\check M,\uptau)/\calG^o_{k+1}(\check M,\upiota).
\]
$\calB^o_k(\check M,\uptau)$ is a branched double cover of $\calB_k(\check M,\uptau)$ along the reducible locus where the projection is quotient by the $\pm 1$-action on spinors.

Assume now $\check M = \check M_1 \sqcup \check M_2$ and $\upiota: \check M \to \check M$ is the swapping involution for fixed identifications $\check M_i \cong \check M_o$, $i = 1,2$. 
The group of harmonic gauge transformations $\mathcal G^h$ fits into a new short exact sequence 
\[
0 \to S^1 \to \calG^h \to H^1(\check M;\bbZ)^{-\upiota^*} \to 0.
\]
Furthermore, there is a new fibration
\[
\frac{L^2_k(\check M;S)^{\uptau} \setminus 0}{S^1} \to \calB^{\sigma}(\check M,\uptau) \to \frac{H^1(\check M;\reals)^{-\upiota^*}}{H^1(\check M;\bbZ)^{-\upiota^*}}
\cong 
\frac{H^1(\check M_o;\reals)}{H^1(\check M_o;\bbZ)}
.
\]
The arguments as before give rise to a homotopy equivalence 
\[
\calB^{\sigma}(\check M,\uptau) \cong \mathbb{CP}^{\infty}\times  H^1(\check M_o;\reals)/H^1(\check M_o;\bbZ).
\]
Similar to the connected case, we define a splitting $v:H^1(\check M_o;\bbZ) \to \calG^{h}$, a subgroup $\calG^{h,o}$, and the framed configuration space $\calB^o_k(\check M,\uptau)$. 
This framed space is an $S^1$-bundle over the non-framed version, which in turn can be identified with a configuration space on $\check M_o$ by Section~\ref{subsec:case_real_loci}:
\[\calB_k(\check M,\uptau) = \calB_k(\check M_o).\] 
By Leray-Hirsch theorem, the cohomology ring of $\calB_k^{\sigma}(\check M,\uptau)$ can be described as follows.
\begin{prop}
	If the real locus of $(\check M,\upiota)$ is nonempty, then there is an isomorphism of cohomology rings
	\[
		H^*(\calB^{\sigma}_k(\check M,\uptau)) \cong
		\Lambda^*(H_1(\check M;\bbZ)^{-\upiota_*}/\text{torsion}) \otimes \bbZ_2[\upsilon],
	\]
	where $\Lambda$ stands for exterior algebra and $\deg \upsilon = 1$. 
	On the other hand, suppose the real locus of $(\check M,\upiota)$ is empty, and $\check M = \check M_1 \sqcup \check M_2$ where $\upiota$ is the swapping involution for an identification $M_i \cong M_o$, $i=1,2$.
	Then we have an isomorphisms of rings
	\[
		H^*(\calB^{\sigma}_k(\check M,\uptau)) \cong
		\Lambda^*(H_1(\check M_o;\bbZ)^{-\upiota_*}/\text{torsion}) \otimes \bbZ_2[U],
	\]
	where $U$ has degree $2$. \qed 
\end{prop}

\subsection{Perturbations}
We set up our perturbations in a way that all definitions are essentially identical to \cite[Section~10.1 and 11]{KMbook2007} or their $\uptau$-invariant counterparts in \cite[Section~6]{ljk2022}, with the addition of the adjective ``bifold''.
This allows us to appeal to the results from \cite[Section~10.6-9 and 11.3-6]{KMbook2007}.

The \emph{perturbed gradient} of the Chern-Simons-Dirac functional is of the form
\[
	\grad \pertL  = \grad \calL + \frakq,
\]
 where $\frakq$ commutes with $\uptau$ and is a \emph{formal gradient}  of a $\calG(\check Y,\upiota)$-invariant continuous function $f:\calC(\check Y,\uptau) \to \reals$, in the sense that
\[
	f \circ \upgamma(1) - f \circ \upgamma(0) =
	\int_0^1 \la \dot{\upgamma},\frakq\ra ,
\] 
for any path $\upgamma:[0,1] \to \mathcal C(\check Y,\uptau)$.

Over a cylinder $(\check Z,\upiota_Z) = ([t_1,t_2] \times \check Y, \Id \times \upiota)$, a continuous path of slice-wise perturbations $\frakq(t)$ can be pullback to obtain a map
\[
\hat{\frakq}:\calC(\check Z,\uptau) \to L^2(Z;i\su(S^+)\oplus S^-)^{\uptau},
\]
by identifying the bundle $iT^*\check Y \oplus S$ with $i\su(S^+)\otimes S^-$ via the Clifford multiplication (see \cite[Page~153]{KMbook2007}).

To define perturbed gradients on the $\sigma$-blowup $(\grad \pertL)^{\sigma}$,
we write a formal gradient $\frakq$ as $(\frakq^0,\frakq^1)$ where $\frakq^0 \in L^2(\check Y;iT^*\check Y)^{-\upiota^*}$, $\frakq^1 \in L^2(\check Y;S)^{\uptau}$, and set
\[
	\tilde{\frakq}^1(B,r,\psi) = \int^1_0 \calD_{(B,sr\psi)}\frakq^1(0,\psi)ds,
	\quad \text{and}\ \ 
	\Lambda_{\frakq}(B,r,\psi) = \Re \la 
	\psi, D_B\psi + \tilde{\frakq}^1(B,r,\psi)
	\ra_{L^2}.
\]
In coordinate form, the perturbation $\frakq^{\sigma}$ is
\[
	\frakq^{\sigma} = \left(\frakq^0(B,r,\psi), \quad
	\la \tilde{\frakq}^1(B,r,\psi), 
	\psi\ra_{L^2(\check Y)}r, \quad 
	\tilde{\frakq}^1(B,r,\psi)^{\perp}\right).
\]

\begin{prop}
	Let $(B,r,\psi)$ be an element of $\calC_k^{\sigma}(\check Y,\uptau)$ and $(B,r\psi)$ be its projection in $\calC_k(\check Y,\uptau)$.
	Then $(B,r,\psi)$ is a critical point of $(\grad \pertL)^{\sigma}$ if and only if either:
	\begin{itemize}
		\item $r \ne 0$ and $(B,r\psi)$ is a critical point of $\grad \pertL$; or
		\item $r = 0$, the point $(B,0)$ is a reducible critical point of $\grad \pertL$, and $\psi$ is an eigenvector of $\phi \mapsto D_B\phi + \calD_{(B,0)}\frakq^1(0,\phi)$.
	\end{itemize}
\end{prop}
The family of perturbations will be built from two classes of functions on $\calC(\check Y,\uptau)$ as follows.
Given a coclosed $1$-form $c \in \Omega^1(\check Y;i\reals)^{-\upiota^*}$, we set
\[
	r_c(B_0 + b\otimes 1, \Psi) = \int_{\check Y} b \wedge *\bar c = \la b,c \ra_{\check Y}.
\] 
This function is invariant under the identity component of $\calG$ and fully gauge invariant if $c$ is also coexact, as
\[
r_c(u(B_0 + b \otimes 1, \Psi)) = r_c(B_0 + b\otimes 1,\Psi) + (h \cup [*\bar c])[\check Y].
\]
Set 
\[\mathbb{T} =
\bbT(\upiota)= H^1(\check Y;i\reals)^{-\upiota^*}/(2\pi H^1(\check Y;i\bbZ)^{-\upiota^*})\]
which can be interpreted as the space of $(-\upiota^*)$-invariant imaginary harmonic $1$-forms, modulo periods in $2\pi i\bbZ$.
Choose an integral basis $(\omega_1,\dots,\omega_t)$ of $H^1(\check Y;i\reals)^{-\upiota^*}$ to identify $\bbT$ with $\bbR^t/2\pi\bbZ^t$.
The $\calG(\check Y,\upiota)$-invariant map $(B_0 + b\otimes 1,\Psi) \mapsto b_{harm}$ that projects $b$ onto its harmonic part $b_{harm}$ can be identified as, modulo periods, 
\[
	(B,\Psi) \mapsto (r_{\omega_1}(B,\Psi), \dots, r_{\omega_t}(B,\Psi)).
\]
Let $\bbS \to \bbT \times \check Y$ be the quotient of $H^1(\check Y;i\reals)^{-\upiota^*} \times S$ by $\calG^{h,o}$.
A smooth $\uptau$-invariant section $\Upsilon$ of $\bbS$ lifts to a $\uptau$-invariant section $\tilde\Upsilon$ of $H^1(\check Y;i\reals)^{-\upiota^*} \times S \to H^1(\check Y;i\reals)^{-\upiota^*} \times \check Y$ with  the following quasi-periodicity:
for any $\kappa \in H^1(\check Y;i\bbZ)^{-\upiota^*}$, there exists a unique $u = v(\kappa) \in \calG^{h,o}$, such that
\[
	\tilde\Upsilon_{\alpha + \kappa} (y) = u(y)\tilde\Upsilon_{\alpha}(y).
\]
Given a $\uptau$-invariant section of $\bbS$, we define a $\calG^o(\check Y,\upiota)$-equivariant map
\[
	\Upsilon^{\dag}:\calC(\check Y,\uptau)\to C^{\infty}(S)^{\uptau}
\]
by
\[
	\Upsilon^{\dag}(B_0+b\otimes 1,\Psi) = 
	e^{-Gd^*b}\tilde\Upsilon_{b_{harm}},
\]
where $b_{harm}$ is the harmonic part of the $1$-form $b$.
Pairing sections of $S$ with $\Upsilon^{\dag}$ gives rise to a $\calG^o(\check Y,\upiota)$-invariant map $q_{\Upsilon}:\calC(\check Y,\uptau) \to \C$,
\[
	q_{\Upsilon}(B,\Psi) = \int_{\check Y} \la \Psi,\Upsilon^{\dag}(B,\Psi)\ra.
\]

Finally, choose a finite collection of $(-\upiota^*)$-invariant coclosed $1$-forms $c_1,\dots,c_{n+t}$ where the first $n$ are coexact, and the last $t$ terms $\omega_1,\dots,\omega_t$ constitute a basis for the invariant harmonic forms.
Choose also a collection of $\uptau$-invariant smooth sections $\Upsilon_1,\dots,\Upsilon_m$ of $\bbS$.

We define $p:\calC(\check Y,\uptau) \to \reals^n \times \bbT(\upiota) \times \C^m$ as
\[
	p(B,\Psi)=
	\left(r_{c_1}(B,\Psi),\dots, r_{c_n}(B,\Psi), 
	[r_{c_{n+1}}], \dots,[r_{c_{n+t}}],
	q_{\Upsilon_1}(B,\Psi),\dots,q_{\Upsilon_m}(B,\Psi)\right).
\]
\begin{defn}
	A function $f:\calC(\check Y,\uptau) \to \reals$ is a \emph{real cylinder function} if it is of the form $g \circ p$ where the map $p:\calC(\check Y,\uptau) \to \reals^n \times \mathbb{T} \times \C^m$ is defined as above, using $n$ coexact forms $c_1,\dots, c_n$ and $m$ sections $\Upsilon_1,\dots,\Upsilon_m$; and the function $g$ on $\reals^n \times \mathbb{T} \times \C^n$ is $S^1$-invariant, smooth, and compactly-supported.
\end{defn}
\begin{rem}
	The maps $p$'s in the definition of \emph{real} cylinder functions can be viewed as functions 
	on the ordinary configuration space $\underline{\calC}(\check Y,\fraks)$ as follows.
	By pairing with all $1$-forms rather than $(-\upiota^*)$-invariant forms, the domain of the $r_i$'s can be extended from the subspace $\calC(\check Y,\uptau)$ to $\underline{\calC}(\check Y,\fraks)$.
	Let the \emph{ordinary Picard torus} be denoted as
	\[\underline{\bbT} = H^1(\check Y;i\reals)/2\pi H^1(\check Y;i\bbZ),\] 
	which contains $\bbT(\upiota)$ as a subtorus. 
	Hence the image of $([r_{c_{n+1}}], \dots,[r_{c_{n+t}}])$ lies in $\underline{\bbT}$.
	Furthermore, extending by quasi-periodicity, a section $\Upsilon$ of $\bbS \to \bbT(\upiota) \times \check Y$ defines a section of the bundle $\underline{\bbS}$ over $\underline{\bbT} \times \check Y$, as a quotient of the ordinary harmonic framed gauge group $\underline{\calG}^{h,o}$: 
	\[
	\underline{\bbS} = \left(H^1(\check Y;i\reals) \times S \to H^1(\check Y;i\reals) \times \check Y\right)/\underline{\calG}^{h,o}.
	\]
	Consequently, a real cylinder function can be thought of as a cylinder function in~\cite[Definition~11.1.1]{KMbook2007} for bifold configurations, with the exception that $g$ is not a function on the entire $\underline{\bbT}$-factor in $\reals^n \times \underline{\bbT} \times \C^n$.
\end{rem}
The analytic results in \cite[Section~11.3-4]{KMbook2007} about perturbations from cylinder functions therefore carry over to our setting.
In particular, the real cylinder functions are \emph{$k$-tame}, where
$k$-tameness comprises a number of constraints to retain analytic properties of solutions such as compactness.
See \cite[Definition~10.5.1]{KMbook2007} and \cite[Definition~6.1]{ljk2022}.
\begin{rem}
	In the case when $\check Y = \check Y_1 \sqcup \check Y_2$ with  a given identification $\check Y_i \cong \check Y_o$ such that $\upiota: \check Y \to \check Y$ is the swapping involution, the real cylinder functions on $\calC(\check Y,\uptau)$ are precisely the cylinder functions on $\underline{\calC}(\check Y_o,\fraks)$.
\end{rem}
The key proposition for establishing transversality is the following.
\begin{prop}
\label{prop:embedding}
Given a compact subset $K$ of a finite-dimensional $C^1$-submanifold $M \subset \calB^o_k(\check Y,\uptau)$, and both $K,M$ are invariant under $\bbZ_2$ if $\Fix(\upiota) = \emptyset$ or $S^1$ if $\Fix(\upiota) = \emptyset$.
Then there exists a collection of coclosed forms $c_{\nu}$, sections $\Upsilon_{\mu}$ of $\bbS$, and a neighbourhood $U$ of $K$ in $M$, such that the corresponding map
\[
	p:\calB^o_k(\check Y,\uptau) \to \reals^n \times \bbT(\upiota) \times \C^m
\]	
induces an embedding of $U$. 
\end{prop}
\begin{proof}
	The proof of \cite[Propostion~12.2.1]{KMbook2007} can be adapted to the $\uptau$-invariant bifold version the same way as in \cite[Proposition~5.5]{ljk2022}.
	Roughly speaking, to separate real connections or spinors it suffices to use cylinder functions defined by real configurations.
\end{proof}
\begin{cor}
\label{cor:embed_tangent}
Given any $[B,\Psi]$ in $\calB^*_k(\check Y,\uptau)$ and any nonzero tangent vector $v$ to $\calB^*_k(\check Y,\uptau)$, there exists a cylinder function $f$ whose differential $\calD_{[B,\Psi]}f(v)$ is nonzero.	\hfill \qedsymbol
\end{cor}
\begin{proof}
	This is a consequence of proof of Proposition~\ref{prop:embedding}.
	See references therein.
\end{proof}
To ensure a large supply of perturbations for the Sard-Smale theorem, we choose a countable collection of real cylinder functions $(c_{\mu}, \Upsilon_{\nu}, K, g)$ satisfying suitable density conditions as follows (cf. \cite[Page~192-193]{KMbook2007}).

Given a pair $(n,m)$, pick a countable collection of $(c_i,\Upsilon_{\nu})$ that are dense in the subspace of such $(n+m)$-tuples, in the $C^{\infty}$ topology. 
Choose a countable set of compact subsets $K$ of $\reals^n \times \bbT \times \C^m$ that are dense in the Hausdorff topology. 
Once $K$ is chosen, let $g_{\alpha} = g(n,m,K)_{\alpha}$ be a countable collection of $S^1$-valued functions with the properties that $g_{\alpha}$ is supported in $K$,
and $\{g_{\alpha}\}$ is dense in the $C^{\infty}$ topology of smooth, $S^1$-invariant functions supported in $K$.
In order to deal with reducibles, we require also the functions $g_{\alpha}$'s that vanish on the subset
\[
K_0 = K \cap (\reals^n \times \bbT\times \{0\})
\]	
are dense among the smooth circle-valued functions that are supported in $K$ and vanish on $K_0$, in the $C^{\infty}$ topology.

With such a collection of $(c_{i}, \Upsilon_{\nu}, K, g_{\alpha})$ specified, one can construct a separable Banach space $\calP$ together with 
a linear map 
\[
\mathfrak{O}: \calP \to C^0(\calC(\check Y,\uptau),\calT_0).
\] 
A pair $(\mathcal{P},\mathfrak{O})$ as above will be referred to as
a \emph{large Banach space of tame perturbations}.
See \cite[Theorem~11.6.1, Definition~11.6.3]{KMbook2007} and also \cite[Section~6]{ljk2022} for the precise definitions.

\subsection{Transversality}
In dimension three, 
a critical point $\fraka \in \calC^{\sigma}_k$ of $(\grad \pertL)^{\sigma}$ is \emph{non-degenerate} if the smooth section $(\grad \pertL)^{\sigma}$ is transverse to the subbundle $\calJ^{\sigma}_{k-1}$.
\begin{thm}
\label{thm:transv_crit}
Let $\calP$ be a large Banach space of tame perturbations.
Then there is a residual subset of $\calP$ such that for any $\frakq$ in  this subset, all zeros of the section $(\grad \pertL)^{\sigma}$ of $\calT^{\sigma}_{k-1} \to \calC^{\sigma}_k(\check Y,\uptau)$ are non-degenerate.
For such a perturbation, the image of the zeros in $\calC_k(\check Y,\uptau)$ consists of a finite set of gauge orbits. 
\end{thm}
\begin{proof}
The proof of for case $\Fix(\upiota) = \emptyset$ without bifold singularities is given in \cite[Section~12]{KMbook2007}.
The case $\Fix(\upiota) \ne \emptyset$ without bifold singularities is done in \cite[Section~7]{ljk2022}, which is the $\uptau$-equivariant version of \cite[Section~12]{KMbook2007}.
The proofs with bifold singularity undergo only cosmetic changes.
It is worth mentioning that the transversality proof at irreducible critical points uses crucially Proposition~\ref{prop:embedding}, and in the reducible case, Corollary~\ref{cor:embed_tangent}.
\end{proof}
To state the transversality results over (infinite) cylinders, we introduce some more definitions.
Let $I$ be a subset of $\reals$, possibly unbounded.
Define $\tilde{\calC}^{\tau}_{k,loc}(I \times \check Y,\uptau)$ as the subspace of $\calA_{k,loc}(I \times \check Y,\uptau) \times L^2_{k,loc}(I;\reals) \times L^2_{k,loc}(I \times \check Y;S^+)^{\uptau}$ consisting of triples $(A,s,\phi)$, where $\check \phi(t)$ has unit $L^2$-norm on $\{t\} \times Y$ for any $t \in I$. 
Also, define $\calC_{k,loc}(I \times \check Y, \uptau)\subset \tilde{\calC}_{k,loc}(I \times \check Y, \uptau)$ to be the closed subset subject to $s(t) \ge 0$.
The gauge group $\calG_{k+1,loc}(I \times \check Y,\uptau)$ comprises $L^2_{k+1,loc}$-maps $\check Y \to S^1$.
Denote the quotient spaces as $\calB_{k,loc}$ and
$\tilde{\calB}_{k,loc}$, respectively.

Choose $\frakq \in \calP$ so that are critical points of $(\grad \pertL)^{\sigma}$ are non-degenerate, by Theorem~\ref{thm:transv_crit}.
The \emph{perturbed Seiberg-Witten equations} on the $\tau$-blow up is a section
\[
	\frakF^{\tau}_{\frakq}:
	\tilde{\calC}^{\tau}_{k,loc}(I \times \check Y,\uptau) \to
	\calV^{\tau}_{k-1,loc}(I \times \check Y,\uptau),
\]
where the fibre of $\calV^{\tau}_{k-1,loc}$ at $\upgamma = (A_0,s_0,\phi_0)$ is the subspace
of $L^2_{k,loc}(I\times \check Y; i\mathfrak{su}(S^+) \oplus L^2_{k,loc}(I;\reals) \oplus L^2_{k,loc}(I \times \check Y)$ consisting of triples $(a,s,\phi)$ with $\Re\la \check\phi_0(t),\check\phi(t)\ra_{L^2(Y)} = 0$ for all $t$.

Let 
\[
	(\check Z,\upiota) = (I \times \check Y, \Id_I \times \upiota),
\]
and $\uptau$ denote also the pullback of $\uptau$ on $Z$.
Given a critical point $\frakb$, there is a translation-invariant element $\upgamma_{\frakb}$ in $\calC^{\tau}_{k,loc}(\check Z,\uptau)$ that solves $\frakF^{\tau}_{\frakq}(\upgamma_{\frakb}) = 0$.
A configuration $[\upgamma] \in \tilde{\calB}^{\tau}_{k,loc}(\check Z,\uptau)$ is \emph{asymptotic} to $[\frakb]$ as $t\to \pm\infty$ if
\[
	[\tau_t^*\upgamma] \to [\upgamma_{\frakb}] \quad
	\text{ in } \tilde{\calB}_{k,loc}(\check Z,\uptau_Z)
\]
as $t \to \pm\infty$,
where $\tau_t:\check Z \to \check Z$ is the translation map $(s,y)\mapsto (s+t,y)$.
We write $\lim_{\rightarrow} [\upgamma]= [\frakb]$ and $\lim_{\leftarrow} [\upgamma] = [\frakb]$ if $[\upgamma]$ is asymptotic to $[\upgamma]$ as $t \to +\infty$ or $t \to -\infty$, respectively.
\begin{defn}
	The \emph{moduli space of trajectories} $M([\fraka],[\frakb])$ on the cylinder $(\check Z,\upiota)$  is the space of configurations $[\upgamma]$ in $\calB^{\tau}_{k,loc}(\check Z,\uptau)$ which are asymptotic to $[\fraka]$ as $t \to -\infty$ and to $[\frakb]$ as $t \to +\infty$, and solves the perturbed Seiberg-Witten equations:
	\[
M([\fraka],[\frakb]) =
\left\{
[\upgamma] \in \calB^{\tau}_{k,loc}(\check Z,\uptau):
\frakF^{\tau}_{\frakq}(\upgamma) = 0,
\lim_{\leftarrow}[\upgamma] = [\fraka],
\lim_{\rightarrow}[\upgamma] = [\frakb]
\right\}.
	\]
\end{defn}
Alternatively (as shown to be equivalent by \cite[Theorem~13.3.5]{KMbook2007}), the moduli space can be defined as a quotient of a subset of Hilbert space by a Hilbert group, modelled on $L^2_k$ instead of $L^2_{k,loc}$.

Given two critical points $[\fraka]$ and $[\frakb]$ in $\calB^{\sigma}_k(\check Y,\uptau)$, represented by $\fraka, \frakb \in \calC^{\sigma}_k(\check Y,\uptau)$, respectively.
Choose a smooth configuration $\upgamma_0 = (A_0,s_0,\phi_0) \in \calC^{\tau}_{k,loc}(\check Z,\uptau)$ which agrees near $\pm\infty$ with the corresponding translation-invariant configurations $\upgamma_{\fraka}$ and $\upgamma_{\frakb}$.
Note that one can arrange $[\check\upgamma_0]$ to belong to any homotopy class $z$.
Set
\[
\calC^{\tau}_k(\fraka,\frakb)=
\left\{
	\upgamma \in \calC^{\tau}_{k,loc}(\check Z,\uptau)\big|
	\upgamma - \upgamma_0 \in
	L^2_k(\check Z;iT^*\check Z) \times 
	L^2_k(\reals;\reals) \times L^2_{k,A_0}(\check Z;S^+)
\right\}.
\]
The Hilbert group $\calG_{k+1}(\check Z,\upiota) = \{u:\check Z \to S^1| \bar u(z) = u(\upiota(z)), 1 - u \in L^2_{k+1}(\check Z;\C)\}$ acts on $\calC^{\tau}_k(\fraka,\frakb)$, where the resulting quotient space will be denoted as  $\calB^{\tau}_k(\fraka,\frakb)$.

Let $\calT_{j,\upgamma}$ be the subset of $L^2_j(\check Z;iT^*\check Z) \oplus L^2_j(\reals;\reals) \oplus L^2_{j,A_0}(\check Z;S^+)^{\uptau}$ comprising triples $(a,s,\phi)$ for which $\Re\la \phi_0|_t, \phi|_t\ra_{L^2(\check Y)} = 0$ for all $t$.
The $j$-tangent bundle decomposes smoothly as $\calJ^{\tau}_j \oplus \calK^{\tau}_j$  \cite[Proposition~14.3.2]{KMbook2007}.
The adjoint of the derivative of the gauge action, coupled with the Coulomb gauge condition, gives an operator
\[
	\mathbf{d}^{\tau,\dag}_{\upgamma}(a,s,\phi)
	= -d^*a + is_0^2 \Re\la i\phi_0,\phi\ra + i|\phi_0|^2\Re \mu_{\check Y}\la i\phi_0,\phi\ra.
\]
Let $\calQ_{\upgamma_0}$ be sum of linearization of the Seiberg-Witten operator on the $\tau$-blowup $\frakF^{\tau}$ and $\mathbf{d}^{\tau,\dag}$:
\begin{equation}
\label{eqn:Q_operator}
	\calQ_{\upgamma_0} = \calD\frakF^{\tau}_{\frakq} \oplus \mathbf{d}^{\tau,\dag}_{\upgamma} 
	:\calK^{\tau}_{j,\upgamma} \to \calV^{\tau}_{j-1,\upgamma} \oplus L^2_{j-1}(\check Z;i\reals)^{-\upiota^*}.
\end{equation}
By the tameness of the perturbation $\frakq$, it can be shown that $\calQ_{\upgamma_0}$ is Fredholm and satisfies a G\r{a}rding inequality~\cite[Theorem~14.4.2]{KMbook2007}.
The index of $\calQ_{\upgamma_0}$, equal to the restriction $\calD\frakF^{\tau}_{\frakq} :\calK^{\tau}_{j,\upgamma} \to \calV^{\tau}_{j-1,\upgamma}$, depends on only on $\fraka$ and $\frakb$, by \cite[Proposition~14.4.3]{KMbook2007}.

In fact by~\cite[Theorem~14.4.2]{KMbook2007}, $\ind(\calQ_{\upgamma_0})$ is equal to the spectral flow of the corresponding extended Hessian operators along $\check\upgamma(t)$:
\[
	\widehat{\text{Hess}}_{\frakq,\fraka}^{\sigma}:
	\calT^{\sigma}_{k,\alpha} \oplus L^2(\check Y;i\reals)^{-\upiota^*} \to \calT^{\sigma}_{k-1,\fraka}\oplus L^2_{k-1}(\check Y;i\reals)^{-\upiota^*},
\]
given by
\[
	\begin{bmatrix}
		\calD_{\fraka}(\grad \pertL)^{\sigma} & \mathbf{d}_{\fraka}^{\sigma}\\
		\mathbf{d}_{\fraka}^{\sigma,\dag} & 0
	\end{bmatrix}.
\]
\begin{defn}
For critical points $\fraka,\frakb \in \calC^{\sigma}_k(\check Y,\uptau)$, the \emph{relative grading} $\gr(\fraka,\frakb)$ of $\fraka$ and $\frakb$ is the index of $\calQ_{\upgamma}$, where $\upgamma$ is any element of $\calC^{\sigma}_k(\frak a,\frakb)$.
For a relative homotopy class $z$ of the path $\boldsymbol{\pi}\circ \check \upgamma$, we write $\gr_z([\fraka],[\frakb])$.
Here, $[\fraka]$ and $[\frakb]$ are the gauge orbits of the $\fraka$ and $\frakb$, respectively.
\end{defn}
It follows from the spectral flow interpretation that relative grading is additive: $\gr(\fraka,\frakc) = \gr(\fraka,\frakb)+\gr(\frakb,\frakc)$.
For a closed loop $z_u$ based at $[\fraka]$ in $\calB^{\sigma}_k(\check Y,\uptau)$, where $u: \check Y \to S^1$ is a real gauge transformation, the proof of \cite[Lemma~8.16]{ljk2022} shows that
\[
	\gr_{z_u}([\fraka],[\fraka]) =
	\frac{1}{2}([u]\cup c_1(S))[\check Y].
\]
The formula differs from \cite[Lemma~14.4.6]{KMbook2007} by a factor of $(1/2)$.
In particular, the grading $\gr_z([\fraka],[\frakb])$ does not depend on the homotopy class $z$ if $(\fraks,\uptau)$ is torsion.

The moduli spaces of trajectories $M([\fraka],[\frakb])$, for some $([\fraka],[\frakb])$, are never transversely cut out in the usual sense.
The failure is due to the presence of \emph{boundary-obstructed cases}, where the operator $\calD\frakF^{\tau}_{\frakq}$ can not be made surjective.
To this end, let us modify the definition of regular moduli space as follows.

A reducible critical point $\fraka \in \calC^{\sigma}_k(Y,\uptau)$ is \emph{boundary-stable} if $\Lambda_{\frakq}(\fraka)>0$ and \emph{boundary-unstable} if $\Lambda_{\frakq} < 0$.
The operator $\calQ_{\upgamma}$ has ``boundary'' and ``normal'' parts
\[
\calQ_{\upgamma} = \calQ_{\upgamma}^{\del}\oplus \calQ_{\upgamma}^{\nu},
\]
where the former $\calQ^{\del}_{\upgamma} = (\calD_{\upgamma}\frakF_{\frakq}^{\tau})^{\del} \oplus \mathbf{d}_{\upgamma}^{\tau,\dag}$ acts on the domain and codomain where the involution $\mathbf{i}$ is trivial. 
The latter is defined as
\[
	\calQ_{\upgamma}^{\nu}:L^2_k(\reals;i\reals) \to L^2_{k-1}(\reals;i\reals), 
	\quad
	\calQ^{\nu}_{\upgamma}(s) = ds/dt + \Lambda_{\frakq}(\check\upgamma)s.
\]
A moduli space $M([\fraka],[\frakb])$ is \emph{boundary-obstructed} if $[\fraka]$ is boundary-stable and $[\frakb]$ is boundary-unstable.
In the boundary-obstructed case, the cokernel of $\calQ_{\upgamma}^{\nu}$ is $1$-dimensional therefore the cokernel of $\calD_{\upgamma}\frakF^{\tau}$ is at least $1$-dimensional.
\begin{defn}
	Let $[\upgamma]$ be a solution in $M_z([\fraka],[\frakb])$. 
	If the moduli space is not boundary-obstructed, then $\upgamma$ is \emph{regular} if $\calQ_{\upgamma}$ is surjective.
	If the moduli space is boundary-obstructed, then $\upgamma$ is \emph{regular} if $\calQ^{\del}_{\upgamma}$ is surjective.
	The space $M_z([\fraka],[\frakb])$ is \emph{regular} if all of its elements are regular.
\end{defn}
There are four possibilities of regular $M_z([\mathfrak a],[\mathfrak b])$ according to the type of endpoints.
\begin{prop}
\label{prop:class_regular_mod_space}
	Suppose $M_z([\mathfrak a],[\mathfrak b])$ is regular, and let $d = \gr_z([\mathfrak a],[\mathfrak b])$. 
	Then $M_z([\mathfrak a],[\mathfrak b])$ is
	\begin{itemize} 
		\item a smooth $d$-manifold containing only  irreducible solutions if either $\mathfrak a$ or $\mathfrak b$ is irreducible;
		\item a smooth $d$-manifold with boundary if $\mathfrak a$, $\mathfrak b$ are respectively boundary-unstable and -stable (in this case, the boundary comprises reducible elements);
		\item a smooth $d$-manifold containing only  reducibles if $\mathfrak a,\mathfrak b$ are either both boundary-stable or both boundary-unstable;
		\item a smooth $(d+1)$-manifold containing only reducibles, in the boundary-obstructed case. 
	\end{itemize}
\end{prop}
\begin{proof}
	The proof is the same as the proof of \cite[Proposition~14.5.7]{KMbook2007}.
\end{proof}
\begin{thm}
\label{thm:transv_cylinder}
	Let $\calP$ be a large Banach space of tame perturbations.
	Then there is a $\frakq \in \calP$ such that all critical points $\frakq \in \calC^{\sigma}_k(\check Y,\uptau)$ are non-degenerate.
	Moreover,
		for each pair of critical points $\fraka,\frakb$, and each relative homotopy class $z$, the moduli space $M_z([\fraka],[\frakb])$ is regular.
\end{thm}
\begin{proof}
The theorem can be proved by the same method in \cite[Section~15]{KMbook2007}, appealing to the same kind of unique continuation results in  \cite[Section~11]{KMbook2007}.
The ingredients include Proposition~\ref{prop:embedding}, Corollary~\ref{cor:embed_tangent}, and additional assumptions on the choices of cylinder function as in \cite[Page~193]{KMbook2007}.
\end{proof}

\subsection{Compactification}
\label{subsec:compactification}
Let $\fraka,\frakb$ be two critical points and $z$ be the corresponding homotopy class of paths joining $[\fraka]$ to $[\frakb]$.
A trajectory $[\upgamma]$ in $M_z([\fraka],[\frakb])$ is \emph{non-trivial} if it is not invariant under translation on $\check Z = \reals \times \check Y$.
An \emph{unparametrized} trajectory is an equivalence class of nontrivial trajectories in $M_z([\fraka],[\frakb])$.
We denote by $\check M_z([\fraka],[\frakb])$ the space of unparametrized trajectories.
To compactify, we add broken trajectories.
\begin{defn}
	An \emph{unparametrized broken trajectory} joining $[\mathfrak a]$ to $[\mathfrak b]$ is given by the following data:
	\begin{itemize}
		\item (\emph{number of components}) an integer $n \ge 0$ of components;
		\item (\emph{the restpoints}) an $(n+1)$-tuple of critical points $[\mathfrak a_0],\dots,[\mathfrak a_n]$ with $[\mathfrak a_0] = \mathfrak a$ and $[\mathfrak a_n] = [\mathfrak b]$ ;
		\item (\emph{the $i$th component of the broken trajectory}) an unparametrized trajectory $[\check{\upgamma}_i]$ in $\check M_{z_i}([\mathfrak a_{i-1}],[\mathfrak a_i])$, for each $i$ with $1 \le i \le n$.
	\end{itemize}
	The \emph{homotopy class} of the broken trajectory is the class of the path obtained by adjoining representatives of the classes $z_i$, or the constant path at $[\mathfrak a]$ if $n=0$.
	We denote by $\check M^+_z([\mathfrak a],[\mathfrak b])$ the space of unparametrized broken trajectories in the homotopy class $z$, and denote a typical element by $[\boldsymbol{\check{\upgamma}}] = ([\check{\upgamma}_1],\dots,[\check{\upgamma}_n])$. 
	\hfill $\diamond$
\end{defn}
Topologizing the space of unparametrized broken trajectories as in \cite[Page~276]{KMbook2007}, we obtain the following counterpart of \cite[Theorem~16.1.3]{KMbook2007}.
\begin{thm}
\label{thm:compactness_cylinder}
	The space $\check M_z^+([\mathfrak a],[\mathfrak b])$ of unparametrized broken trajectories is compact.
\end{thm}
\begin{proof}
	The analytic inputs needed for the compactness theorems come from \cite[Section~8,10.7,16]{KMbook2007}. 
	By our choice of the large Banach space of tame perturbations, the $\uptau$-invariant trajectories are naturally trajectories in ordinary monopole Floer homology except with bifold singularities.
	The arguments and results in the book apply equally well to bifold Seiberg-Witten configurations.
\end{proof}
\begin{defn}
	A space $N^d$ is a \emph{$d$-dimensional space stratified by manifolds} if there is a nested sequence of closed subsets
	\begin{equation*}
		N^d \supset N^{d-1} \supset \cdots \supset N^0 \supset N^{-1} = \emptyset
	\end{equation*}	
	such that $N^d \ne N^{d-1}$ and each $N^e \setminus N^{e-1}$ (for $0 \le e \le d)$ is either empty or homeomorphic to a manifold of dimension $e$.
	The difference $N^e \setminus N^{e-1}$ is the \emph{$e$-dimensional stratum}.
	The term \emph{stratum} will also stand for any union of path components of $N^e \setminus N^{e-1}$.
\end{defn}
\begin{prop}
	Suppose that $M_z([\mathfrak a],[\mathfrak b])$ is nonempty and of dimension $d$.
	Then the compactification $\check M^+_z([\mathfrak a],[\mathfrak b])$ is a $(d-1)$-dimensional space stratified by manifolds.
	If $M_z([\mathfrak a],[\mathfrak b])$ contains irreducible trajectories, then the $(d-1)$-dimensional stratum of $\check M^+_z([\mathfrak a],[\mathfrak b])$ consists of the irreducible part of $\check M^+_z([\mathfrak a],[\mathfrak b])$.
\end{prop}
\begin{proof}
	The proof is identical to \cite[Proposition~16.5.2]{KMbook2007}, replacing the use of \cite[Proposition~14.5.7]{KMbook2007} by Proposition~\ref{prop:class_regular_mod_space}
\end{proof}
Consider a stratum of the form
\begin{equation}
\label{eqn:productmodspace}
	\check M_{z_1}([\mathfrak a_0],[\mathfrak a_1]) \times \cdots \times
	\check M_{z_{\ell}}([\mathfrak a_{\ell-1}],[\mathfrak a_{\ell}]).
\end{equation}
Denote $d_i$ the dimension of $M_{z_i}([\mathfrak a_{i-1}],[\mathfrak a_i])$, and $\epsilon_i$ such that $d_i - \epsilon_i = \gr_{z_i}([\mathfrak a_{i-1}],[\mathfrak a_i])$.
The vector $(d_i-\epsilon_i)$ will be called the \emph{grading vector}, and $(\epsilon_i)$ will be called the \emph{obstruction vector}.
\begin{prop}
	Let $M_z([\mathfrak a],[\mathfrak b])$ be a $d$-dimensional moduli space that contains irreducibles.
In other words, $\check M_z^+([\mathfrak a],[\mathfrak b])$ is a compact $(d-1)$-dimensional space stratified by manifolds, whose top stratum is the irreducible part of $\check M_z([\mathfrak a],[\mathfrak b])$.
Then the $(d-2)$-dimensional stratum of $\check M_z^+$ comprise three types of strata:
	\begin{itemize}
	\itemindent=-13pt
		\item the top stratum is of the product form~\eqref{eqn:productmodspace} with grading vector $(d_1,d_2)$ and obstruction vector $(0,0)$;
		\item the top stratum is of the form~\eqref{eqn:productmodspace} with grading vector $(d_1,d_2-1,d_3)$ and obstruction vector $(0,1,0)$;
		\item the intersection of $\check M_z([\mathfrak a],[\mathfrak b])$ with the reducibles, if there are both reducibles and irreducibles.
	\end{itemize}
	The third case occurs only when $[\mathfrak a]$ is boundary-unstable and $[\mathfrak b]$ is boundary-stable.
 \end{prop}
 \begin{proof}
 	This corresponds to \cite[Proposition~16.5.5]{KMbook2007}, whose proof is essentially a corollary of the proof of \cite[Proposition~16.5.2]{KMbook2007}.
 \end{proof}
\begin{prop}
	Suppose $M^{\red}_z([\mathfrak a],[\mathfrak b])$ is non-empty and of dimension $d$.
	Then the space of unparametrized, broken reducible trajectories $\check M^{\red +}_z([\mathfrak a],[\mathfrak b])$ is a compact $(d-1)$-dimensional space stratified by manifolds.
	The top stratum comprise precisely $\check M^{\red}_z([\mathfrak a],[\mathfrak b])$.
	The $(d-\ell)$-dimensional stratum are formed the spaces of unparametrized broken trajectories with $\ell$-factors:
	\begin{equation*}
		\check M_z^{\red}([\mathfrak a_0],[\mathfrak a_1]) \times
		\cdots \times 
		\check M_z^{\red}([\mathfrak a_{\ell-1}],[\mathfrak a_{\ell}]). 
	\end{equation*}
\end{prop}
\begin{proof}
	The proof can be adapted from the proofs of the non-reducible cases above.
\end{proof}
\subsection{Gluing and boundary stratification}
\label{subsec:gluing}
Section~19 of \cite{KMbook2007} sets up a general framework for gluing trajectories.
This framework was adapted in \cite[Section~10]{ljk2022} to real Seiberg-Witten trajectories. 
The case when the configurations have bifold singularities undergoes no changes.
\begin{defn}
	Let $(Q,q_0)$ be a (pointed) topological space, let $\pi: S \to Q$ be continuous, and let $S_0 \subset \pi^{-1}(q_0)$.
	Then $\pi$ is a \emph{topological submersion} along $S_0$, if for all $s_0 \in S_0$ there is a neighbourhood $U$ of $s_0 \in S$, a neighbourhood $Q'$ of $q_0$ in $Q$, and a homeomorphism $(U \cap S_0) \times Q' \to U$ such that the diagram
\[\begin{tikzcd}
	{(U \cap S_0) \times Q'} && U \\
	{Q'} && {Q'}
	\arrow[Rightarrow, no head, from=2-1, to=2-3]
	\arrow[from=1-1, to=2-1]
	\arrow[from=1-3, to=2-3]
	\arrow[from=1-1, to=1-3]
\end{tikzcd}\]
	commutes.
\end{defn}
\begin{defn}
\label{defn:delta2str}
	Let $N^d$ be a $d$-dimensional space stratified by manifolds.
	Let $M^{d-1} \subset N$ be the union of components of the $(d-1)$-dimensional stratum.
	Then $N$ has a \emph{codimension-$c$ $\delta$-structure} along $M^{d-1}$ if $M^{d-1}$ is smooth and there are additional data:
	\begin{itemize}
		\item an open set $W \subset N$ containing $M^{d-1}$,
		\item an embedding $j:W \to EW$ of $W$ into a topological space $EW$, and
		\item a map
		\begin{equation*}
			\mathbf{S} = (S_1,S_2): EW \to (0,\infty]^{c+1}
		\end{equation*}
	\end{itemize}
	such that:
	\begin{enumerate}
	\itemindent=-7pt
		\item[(i)] the map $\mathbf{S}$ is a topological submersion along $\mathbf S^{-1}(\infty,\infty)$;
		\item[(ii)] the fibre $\mathbf S^{-1}(\infty,\infty)$ is $j(M^{d-1})$;
		\item[(iii)] the subset $j(W) \subset EW$ is the zero set of a map $\delta:EW \to \Pi^c$, where $\Pi^c \subset \reals^{c+1}$ is the hyperplane $\Pi^c = \{\delta \in \reals^{c+1}| \sum \delta_i = 0\}$;
		\item[(iv)] if $e \in EW$ has $S_{i_0}(e) = \infty$ for some $i_0$, then $\delta_{i_0}(e) \le 0$, with equality only if $S_{i}(e) = \infty$ for all $i$;
		\item[(v)] $\delta$ is smooth and transverse to zero on the subset of $EW$ where all $S_i$ are finite. 
	\end{enumerate}
\end{defn}

We state the counterpart of \cite[Theorem~19.5.4]{KMbook2007} about structure of codmension-1 strata in terms of $\delta$-structures.
\begin{thm}
	Suppose the moduli space $M_z([\mathfrak a],[\mathfrak b])$ is $d$-dimensional and contains irreducible trajectories, so that the moduli space $\check M^+_z([\mathfrak a],[\mathfrak b])$ is a $(d-1)$-dimensional space stratified by manifolds.
	Let $M' \subset \check M^+([\mathfrak a],[\mathfrak b])$ be any component of the codimension-1 stratum.
	Then along $M'$, the moduli space $\check M^+([\mathfrak a],[\mathfrak b])$ either is a $C^0$-manifold with boundary, or has a codimension-1 $\delta$-structure.
	The latter only occurs when $M'$ comprises 3-component broken trajectories, with the middle component boundary-obstructed. \qed 
\end{thm}

\begin{thm}
\label{thm:str_cod1_cylinder}
	Suppose that the moduli space $M^{\red}_z([\mathfrak a],[\mathfrak b])$ is $d$-dimensional and non-empty, so that the compactified moduli space of broken reducible trajectories $\check M^{\red +}([\mathfrak a], [\mathfrak b])$ is a $(d-1)$-dimensional space stratified by manifolds.
	Let $M' \subset \check M^{\red +}([\mathfrak a], [\mathfrak b])$ be a component of the codimension-1 stratum.
	Then along $M'$ the moduli space $\check M^{\red +}([\mathfrak a],[\mathfrak b])$ is a $C^0$-manifold with boundary.
	\qed 
\end{thm}
\begin{lem}
\label{lem:bound_multiplicities_0}
	Let $N^1$ be a compact 1-dimensional space stratified by manifolds, so that $N^0$ is a finite number of points.
	If $N^1$ has a codimension-1 $\delta$-structure along $N^0$,
	then the number of points of $N^0$ is zero modulo two. 
\end{lem}
Lemma~\ref{lem:bound_multiplicities_0} follows from a version of Stokes' theorem in \cite[Section~21]{KMbook2007}, proved using \v{C}ech models.
Since their setup will be used to define pairing of (co)homology classes over compactified moduli spaces, we recall the necessary definitions, all taken from \cite[Section~21]{KMbook2007}.
\subsection{\v{C}ech cohomology and Stokes' theorem}
\label{subsec:stokes}
A metric space has \emph{covering dimension $\le d$} if every open cover \emph{has covering order $\le d+1$}, that is, there is a refinement $\calU$ such that every $(d+2)$-fold intersection of the form
\[
	U_0 \cap U_1 \cap \dots \cap U_{d+1}, \quad U_i \in \calU \text{ distinct}
\]
is empty.
Each open covering $\calU$ can be assigned a simplicial complex $K(\calU)$ called \emph{nerve}.
The \v{C}ech cohomology of $B$ is the limit
\[
	\check H(B;\bbF_2) =
	\lim_{\rightarrow} H^n_{Simp}(K(\calU);\bbF_2),
\]
over all open coverings $\calU$.
For a subcover $\calU' \subset \calU$ and a subset $B' \subset B$ such that $\{U \cap B': U \in \calU'\}$ covers $B'$, there is a corresponding subcomplex $K(\calU'|B')$.
The relative cohomology of $(B,B')$ is defined as the limit
\[
	\check H(B, B';\bbF_2) =
	\lim_{\rightarrow} H^n_{Simp}(K(\calU),K(\calU'|B');\bbF_2).
\]
Suppose $N^d$ is a compact $d$-dimensional space stratified by manifolds.
We have
\[
	\check H^d(N^d,N^{d-1};\bbF_2) = H^d_c(M^d;\bbF_2),
\]
which is a free abelian group generated by $\mu_{\alpha}^d$, corresponding to a component $M^d_{\alpha}$ of $M^d = N^{d} \setminus N^{d-1}$.
Let 
\[
	I_{\alpha}: \check H^d(N^d,N^{d-1};\bbF_2) \to \bbF_2
\]
take value $1$ on $\mu_{\alpha}^d$ and zero otherwise.

Let $\delta_*$ be the coboundary map
\[
	\delta_*: 
	H^{d-1}_c(M^{d-1};\bbF_2) =
	\bigoplus_{\beta} H^{d-1}_c(M^{d-1}_{\beta};\bbF_2) \to
	\bigoplus_{\alpha} H^d_c(M^d_{\alpha};\bbF_2).
	=  H^d_c(M^d;\bbF_2)
\]
associated to the long exact sequence of the triple $(N^d,N^{d-1},N^{d-2})$.
For each pair $(M^d_{\alpha}, M^{d-1}_{\beta})$, the entry $\delta_{\alpha\beta}$ is the multiplicity of $M^{d-1}_{\beta}$ appearing in the boundary of $M^d_{\alpha}$.
In other words,
\[
	\delta_{\alpha\beta} = I_{\alpha} \delta_* \mu_{\beta}^{d-1}.
\]
The \emph{boundary multiplicity} of the component $M_{\beta}^{d-1}$ in the stratified space $N^d$ is the sum
\[
	\delta_{\beta} = \sum_{\alpha} \delta_{\alpha\beta}.
\]
Suppose $N^d$ is embedded in a metric space $B$.
An open cover $\calU$ of $B$ is \emph{transverse} to the strata if $\calU|N^e$ has covering order $\le e+1$ for all $e \le d $.
Given a collection of stratified spaces, a covering can be always be refined to become transverse to the strata, by the following lemma.
(The prototype of such stratified spaces for us are moduli spaces of trajectories.)
\begin{lem}[Lemma~21.2.1 of \cite{KMbook2007}]
\label{lem:exist_refine}
Let $N^{d_k}_k$ be a countable, locally finite collection of spaces stratified 	by manifolds.
Then every open cover $\calU$ of $B$ has a refinement $\calU'$ that is transverse the strata in every $N^{d_k}_k$. 
\qed 
\end{lem}
Using open coverings that are transverse the stratification of a stratified space $N^d$, one
can compute the \v{C}ech cohomology of $B$ as the limit
\[
	\check H(B;\bbF_2) = \lim_{\rightarrow}H^n(K(\calU);\bbF_2).
\]
Given an open cover $\calU$ satisfying the transversality condition and a class $u \in \check C^d(\calU|N^d;\bbF_2)$, there is a well-defined class $[u]$ in $\check H(N^d,N^{d-1};\bbF_2)$.
Define a pairing with $[M_{\alpha}]$ by
\[
	\langle -, [M_{\alpha}] \rangle :
	\check C^d(\calU;\bbF_2) \to \bbF_2,\quad
	\langle u, [M_{\alpha}] \rangle = I_{\alpha}[u|_{N^d}].
\]
Let $\delta: \check C^{d-1}(\calU;\bbF_2) \to \check C^d(\calU;\bbF_2)$ be the \v{C}ech coboundary map.
The Stokes' theorem follows by definition: 
for any $v \in \check C^{d-1}(\calU;\bbF_2)$,
\[
	\sum_{\beta} \delta_{\alpha\beta} \langle v, [M_{\beta}^{d-1}] \rangle = \langle \delta v, [M^d_{\alpha}]\rangle.
\]
The following is the unoriented version of \cite[Lemma~21.3.1]{KMbook2007}.
\begin{lem}
	Let $N^d$ be a compact $d$-dimensional space stratified by manifolds, and let $M^{d-1}_{\beta}$ be a component of $M^{d-1}$.
	If $N^{d}$ has a codimension-$1$ $\delta$-structure along $M^{d-1}_{\beta}$, then the boundary multiplicity $\delta_{\beta}$ of $M^{d-1}_{\beta}$ in $N^d$ is $1$. \qed
\end{lem}

\section{Floer homologies for web and foams}
\label{sec:floer_homology}
\subsection{Floer homologies for \emph{real} bifolds}
Let $(\check Y,\upiota)$ be a real bifold, $\check g$ be an $\upiota$-invariant bifold metric, and $(\mathfrak s, \uptau)$ be a real bifold spin\textsuperscript{c} structure.
Suppose $\mathcal P$ is a large Banach space of tame perturbations.
By Theorem~\ref{thm:transv_cylinder}, choose
$\mathfrak q \in \mathcal P$ so that all critical points of $(\grad \pertL)^{\sigma}$ in $\mathcal B^{\sigma}_k(\check Y,\uptau)$ are nondegenerate, and all moduli spaces $M([\mathfrak a],[\mathfrak b])$ are regular.
In addition, if $c_1(S)$ is not torsion, then assume $\frakq$ is chosen so that there are no reducible critical points. 
\begin{defn}
A tame perturbation $\mathfrak q$ is \emph{admissible} if all critical points of the associated $(\grad \pertL)^{\sigma}$ are non-degenerate, all moduli spaces $M([\fraka],[\frakb])$ of trajectories are regular, and there are no reducible critical points unless $c_1(S)$ is torsion.
\end{defn}

Let $\mathfrak C \subset \mathcal B^{\sigma}_k(\check Y,\uptau)$ be the set of critical points.
In terms of irreducible $(o)$, boundary-stable $(s)$, and boundary-unstable critical points $(u)$:
\begin{equation*}
	\mathfrak C = \mathfrak C^o \cup \mathfrak C^s \cup \mathfrak C^u.
\end{equation*}
Define the chain groups over $\mathbb F_2$ as
\[
	\check C = C^o \oplus C^s, \quad
	\hat C = C^o \oplus C^u, \quad
	\bar C = C^s \oplus C^u.
\]
where we set
\[
	C^o = \bigoplus_{[\mathfrak a] \in \mathfrak C^o} \mathbb F_2[\mathfrak a], \quad 
	C^s = \bigoplus_{[\mathfrak a] \in \mathfrak C^s} \mathbb F_2[\mathfrak a], \quad 
	C^u = \bigoplus_{[\mathfrak a] \in \mathfrak C^u} \mathbb F_2[\mathfrak a].
\]
The differential $\bar\del:\bar C \to \bar C$ is given by
\begin{equation*}
	\bar\del [\mathfrak a] = \sum_{[\mathfrak b],z} \# \check M_z^{\text{red}}([\mathfrak a],[\mathfrak b]) \cdot \mathfrak [\mathfrak b],
	\quad 
	\bar\del = \begin{pmatrix}
		\bar\del^s_s && \bar\del^u_s\\
		\bar\del^s_u && \bar\del^u_u
	\end{pmatrix} \text{ over } C^o \oplus C^s
\end{equation*}
where $\check M_z^{\text{red}}([\mathfrak a],[\mathfrak b])$'s are unparametrized moduli spaces of \emph{reducible} trajectories with dimension zero.
Similarly, we define operators
\begin{align*}
	\del^o_o: C^o \to C^o \quad [\mathfrak a] \mapsto \sum_{[\mathfrak b] \in \mathfrak C^o} \#\check M_z([\mathfrak a],[\mathfrak b]) \cdot [\mathfrak b],\\
	\del^o_s: C^o \to C^s \quad [\mathfrak a] \mapsto \sum_{[\mathfrak b] \in \mathfrak C^s} \#\check M_z([\mathfrak a],[\mathfrak b]) \cdot [\mathfrak b],\\
	\del^u_o: C^u \to C^o \quad [\mathfrak a] \mapsto \sum_{[\mathfrak b] \in \mathfrak C^o} \#\check M_z([\mathfrak a],[\mathfrak b]) \cdot [\mathfrak b],\\
	\del^u_s: C^u \to C^s \quad [\mathfrak a] \mapsto \sum_{[\mathfrak b] \in \mathfrak C^s} \#\check M_z([\mathfrak a],[\mathfrak b]) \cdot [\mathfrak b],
\end{align*}
by counting points in zero-dimensional moduli spaces.
Let the boundary operators 
 $\check\del : \check C \to \check C$ and 
  $\hat\del : \hat C \to \hat C$ be	\begin{equation*}
		\check\del = 
		\begin{pmatrix}
			\del^o_o && -\del^u_o\bar\del^s_u\\
			\del^o_s && \bar\del^s_s - \del^u_s\bar\del^s_u
		\end{pmatrix},
\quad 
		\hat\del = \begin{pmatrix}
			\del^o_o && \del^u_o\\
			-\bar\del^s_u\del^o_s &&
			-\bar\del^u_u\del^u_s
		\end{pmatrix}
\end{equation*}
respectively.
\begin{prop}
	$\check \del^2 = 0$, $\hat \del^2 =0$, and $\bar\del^2 = 0$.
\end{prop}
\begin{proof}
	In terms of the matrix entries, the statement is equivalent to a series of identities:
	\begin{align*}
		-\del^o_o\del^o_o+\del^u_o\bar\del^s_u\del^o_s &= 0, 
		&\bar\del^s_s\bar\del^s_s  + \bar\del^u_s\bar\del^s_u = 0,\\
		-\del^o_s\del^o_o-\bar\del_s^s\del^o_s + \del^u_s\bar\del^s_u\del^o_s &= 0, 
		&\bar\del^s_s\bar\del^u_s + \bar\del^u_s\bar\del^u_u = 0,\\
		-\del^o_o\del^u_o + \del^u_o\bar\del^s_u\del^u_s &= 0 ,
		&\bar\del^s_u\bar\del^s_s + \bar\del^u_u\bar\del^s_u = 0,\\
		-\bar\del^u_s - \del^o_s \del^u_o - \bar\del^s_s\del^u_s + \del^u_s\bar\del^u_u + \del^u_s\bar\del^s_u\del^u_s &= 0,
		&\bar\del^s_u\bar\del^u_s + \bar\del^u_u\bar\del^u_u = 0.
	\end{align*}
Each of the identity can be derived by enumerating (with multiplicity) boundary points of some $1$-dimensional moduli spaces and then appealing to Lemma~\ref{lem:bound_multiplicities_0}.
The proof depends only on the descriptions of the boundary stratifications of moduli spaces in Section~\ref{subsec:compactification} and \ref{subsec:gluing}.
Hence the arguments in \cite[Proposition~22.1.4]{KMbook2007} applies to our situation verbatim.

To get a taste of the arguments, let us look at the first identity $-\del^o_o\del^o_o+\del^u_o\bar\del^s_u\del^o_s = 0 $.
Suppose $[\fraka]$ and $[\frakb]$ are two irreducible critical points such that $\dim \check M_z([\fraka],[\frakb]) = 1$.
By Theorem~\ref{thm:str_cod1_cylinder} , the boundary degeneration can either be of the form
\[
	\check M_{z_1}([\fraka],[\fraka_1]) \times
	\check M_{z_2}([\fraka_1],[\frakb]),
\]
whose neighbourhood has a structure of a $C^0$-manifold with boundary, or of the form
\[
	\check M_{z_1}([\fraka],[\fraka_1]) \times
	\check M_{z_2}([\fraka_1],[\fraka_2]) \times 
	\check M_{z_3}([\fraka_2],[\frakb])
\]
where the middle is boundary-obstructed. 
This triple product has a codimension-$1$ $\delta$-structure in $M_z([\fraka],[\frakb])$ and contributes $\del^u_o\bar\del^s_u\del^o_s$.
\end{proof}
\begin{defn}
	The \emph{real bifold monopole Floer homology} groups for the real bifold $(\check Y, \upiota)$ and real \spinc structure $(\fraks,\uptau)$ are the the homologies of the respective chain complexes:
\begin{equation*}
	\widecheck{\HMR}_*(\check Y,\upiota; \fraks, \uptau) = H_*(\check C,\check \del), \quad
	\widehat{\HMR}_*(\check Y,\upiota;\fraks, \uptau) = H_*(\hat C,\hat \del), \quad
	\overline{\HMR}_*(\check Y,\upiota;\fraks, \uptau) = H_*(\overline C,\bar \del).
\end{equation*}
We will write $\HMR^{\circ}_*$, where $\circ \in \{\vee,\wedge,-\}$.
Whenever we do not specify $(\fraks,\uptau)$, the group will be the sum over all real \spinc structures:
\[
	\HMR^{\circ}_*(\check Y,\upiota) = 
	\bigoplus_{(\fraks,\uptau) \in \SpincR(\check Y,\upiota)}\HMR^{\circ}_*(\check Y,\upiota; \fraks, \uptau).
\]
\end{defn}
\begin{notat}
	We will continues to suppress ``$\upiota$'' and ``$\fraks$'' if $\uptau$ is given.
	For example, $\HMR^{\circ}_*(\check Y,\uptau)$ stands for $\HMR^{\circ}_*(\check Y,\upiota;\fraks,\uptau)$.
\end{notat}
The chain maps 
\begin{equation*}
	i: \bar C \to \check C,\quad
	j: \check C \to \hat C, \quad
	p: \check C \to \bar C,
\end{equation*}
given by
\begin{equation*}
	i = \begin{pmatrix}
		0 & \del^u_o \\
		1 & -\del^u_s
	\end{pmatrix},
	\quad 
	j = \begin{pmatrix}
		1 & 0 \\
		0 & -\delbar^u_s
	\end{pmatrix}, 
	\quad 
	p = \begin{pmatrix}
		\del^o_s & \del^u_s \\
		0 & 1
	\end{pmatrix}.
\end{equation*}
induce a long exact sequence
	\begin{equation*}
	\begin{tikzcd}
		\cdots \ar[r,"i_*"] 
		& \widecheck{\HMR}_*(\check Y,\uptau)	 \ar[r,"j_*"] 
		& \widehat{\HMR}_*(\check Y,\uptau) \ar[r,"p_*"]
		& \overline{\HMR}_*(\check Y,\uptau) \ar[r,"i_*"]
		& \widehat{\HMR}_*(\check Y,\uptau) \ar[r]
		& \cdots 
	\end{tikzcd}
\end{equation*}
\begin{defn}
	We define the \emph{reduced} real bifold monopole Floer homology $\HMR_{*}(\check Y,\upiota,\uptau)$ as the image of the map
	\[
		j_*:\widecheck{\HMR}_*(\check Y,\upiota,\uptau) \to \widehat{\HMR}_*(\check Y,\upiota,\uptau).
	\]
\end{defn}
The grading $* \in \mathbb J(\uptau)$ on the Floer homology groups will be defined as follows.
Given two elements $([\mathfrak a],\mathfrak q_1,n)$ and $([\mathfrak b],\mathfrak q_2,m)$ in 
\[
\mathcal B^{\sigma}_k(\check Y,\uptau) \times \mathcal P \times \mathbb Z.
\]
We declare $([\mathfrak a],\mathfrak q_1,n) \sim ([\mathfrak b],\mathfrak q_2,m)$ if there exists a path $\zeta$ and a $1$-parameter family of perturbations $\mathfrak p$ such that the index of operator $P_{\upgamma,\mathfrak p}$, defined in \cite[Equation~(20.6)]{KMbook2007} is equal to $(n-m)$.
(Roughly speaking, $P_{\upgamma,\mathfrak p}$ generalizes operator the $\calQ_{\upgamma}$ in \eqref{eqn:Q_operator} where the perturbations can vary on a cylinder.)
In particular, if $[\mathfrak a]$ and $[\mathfrak b]$ are critical points of the same perturbation, the index of $P_{\upgamma}$ is equal to $\gr_z([\mathfrak a],[\mathfrak b])$.

The set $\bbJ(\uptau)$ consists of $\sim$-equivalence classes:
	\begin{equation*}
		\mathbb J(\uptau) = (\mathcal B^{\sigma}_k(\check Y,\uptau)\times \mathcal P \times \mathbb Z)/\sim.
	\end{equation*}
The \emph{grading} of a critical point $[\mathfrak a]$ is the equivalence class
	\begin{equation*}
		\gr[\mathfrak a] = ([\mathfrak a],\mathfrak q,0)/\sim.
	\end{equation*}
This grading is additive in the sense that $
		\gr[\mathfrak a]= \gr[\mathfrak b] + \gr_z([\mathfrak a],[\mathfrak b])$.
For the reducibles, we also define:
	\begin{equation*}
		\bar\gr[\mathfrak a] = \begin{cases}
			\gr[\mathfrak a], & [\mathfrak a]\in \mathfrak C^s\\
			\gr[\mathfrak a]-1, & [\mathfrak a] \in \mathfrak C^u
		\end{cases}.
	\end{equation*}
The chain complexes are graded by $\mathbb J(\uptau)$, where the differentials lower gradings by $1$.
For $j \in \bbJ(\uptau)$, we write $C^o_j, C^u_j, C^s_j$ and
\[		\check{C}_j = C^o_j \oplus C_j^s, \quad 
		\hat{C}_j = C^o_j \oplus C^u_j, \quad
		\bar{C}_j = C^s_j \oplus C^u_j.
\]

There is a natural $\bbZ$-action on $\bbJ(\uptau)$, where $n \in \bbZ$ sends $j=([\mathfrak a],\mathfrak q,m)$ to the element $([\mathfrak a],\mathfrak q,m+n)$, also denoted as $j+n$. 
The action of $\mathbb Z$ on is transitive, where
the stabilizer is given by the image of the map 
\[
H_2(\check Y;\mathbb Z)^{-\upiota^*} 
		\to \frac{1}{2}\mathbb Z, \quad 
[\sigma]\mapsto
		\frac{1}{2}\langle c_1(\mathfrak s),[\sigma]\rangle.
\]
In particular, the action is free if and only if $c_1(\mathfrak s)$ is torsion.
\begin{rem}
The extra factor of $\frac12$ is significant, in contrast with the ordinary monopole formula \cite[Lemma~14.4.6]{KMbook2007}.
This factor prevents $\HMR$ from having an absolute $\bbZ/2$-grading.
\end{rem}

Cobordism maps in $\HMR_*^{\circ}$ can have infinitely many entries along the negative direction.
We complete the grading of $\HMR_*^{\circ}$ as follows.
In general, let $G_*$ be an abelian group graded by a set $\mathbb J$.
Let $O_{\alpha}$, $\alpha \in A$ be the free $\mathbb Z$-orbits in $\mathbb J$, and choose an element $j_{\alpha} \in O_{\alpha}$ for each $\alpha$.
Let $G_*[n] \subset G_*$ be the subgroup
\begin{equation*}
	G_*[n] = \bigoplus_{\alpha}\bigoplus_{m \ge n} G_{j_{\alpha} - m},
\end{equation*}
which defines a decreasing filtration for $G_*$.
Let $G_{\bullet} \supset G_*$ be the topological group obtained by completing with respect to this filtration.
We apply this negative completion procedure to $\mathbb J = \mathbb J(\check Y,\uptau)$ to obtain
\[\HMR^{\circ}_{\bullet}(\check Y,\uptau).\]

\subsection{Floer homologies for webs}
Suppose $(K,s)$ is a web with a $1$-set.
Let $(\check Y,\upiota) = (Y,\tilde K^{\sfc},\upiota)$ be the real double of $(K,s)$.
We start with a bifold metric on $(S^3,K)$ as in Section~\ref{subsec:bifolds_orbi} and pull it back to $(\check Y,\upiota)$ to obtain an $\upiota$-invariant bifold metric.
We can apply $\HMR^{\circ}_*$  to the real double covers of $K$ for all $1$-sets.
\begin{defn}
The \emph{monopole web Floer homology of} $(K,s)$ is the monopole Floer homology of $(\check Y,\upiota)$, summed over all real \spinc structures:
\[
	\HMR^{\circ}_{\bullet}(K,s) = \bigoplus_{(\mathfrak s, \uptau) \in \RSpin^{\sfc}(\check Y,\upiota)} \HMR^{\circ}_{\bullet}(\check Y,\upiota;\fraks,\uptau).
\] 
If we do not specify the $1$-set $s$, then 
\[
	\HMR^{\circ}_{\bullet}(K) = \bigoplus_{s \in \{\text{1-sets of } K\}} \HMR^{\circ}_{\bullet}(K,s).
\]
\end{defn}
\section{Funtoriality and foam evaluation}
\label{sec:functoriality}
In this section, we will define cobordism maps and module structures in $\HMR^{\circ}_*$.
Also, we define an invariant for closed real bifolds and an evaluation for dotted closed foams.
The new feature, compared to \cite{ljk2022}, is that a cobordism can have boundary components with empty real loci.
\subsection{Moduli spaces over 4-bifolds with boundaries}
\label{subsec:mod_space_over_bifold_boundary}
Let $(\check X,\upiota)$ be a compact, connected, oriented real $4$-bifold, and $\check g_X$ be an $\upiota_X$-invariant bifold metric.
We write $\upiota: \check Y \to \check Y$ for the restriction of the involution on the boundary $\del \check X = \check Y$.
Assume $\check X$ contains an isometric copy of $(I \times \check Y, \Id_I \times \upiota)$ for some interval $I = (-C,0]$ with $\del \check X$ identified with $\check Y \times \{0\}$.
Let $(\fraks_X,\uptau_X)$ be a real bifold \spinc structure, and $(\fraks,\uptau)$ on $(\check Y,\upiota)$ be its restriction \spinc structure on the boundary.

For each component $\check Y^{\alpha}$ of the boundary,
choose a large Banach space of tame perturbations $\calP_k(\check Y,\fraks^{\alpha},\uptau^{\alpha})$.
Let
\begin{equation*}
	\mathcal B^{\sigma}_k(Y,\mathfrak s,\uptau) =
	\prod \mathcal B^{\sigma}_k(Y^{\alpha},\mathfrak s^{\alpha},\uptau^{\alpha}),
	\quad
	\mathcal P^{\sigma}_k(Y,\mathfrak s,\uptau) 
	=\prod 
	\mathcal P^{\sigma}_k(Y^{\alpha},\mathfrak s^{\alpha},\uptau^{\alpha}).
\end{equation*}
Let us again suppress ``$\fraks$'' whenever ``$\uptau$'' is present.

The perturbations of the Seiberg-Witten map $\mathfrak F^{\sigma}$ occur on the cylindrical regions.
Let $\beta$ be a cut-off function, equal to $1$ near $t = 0$ and to $0$ near $t = -C$.
Let $\beta_0$ be a bump function with compact support in $(-C,0)$.
Let $\mathfrak q$ and $\mathfrak p_0$ be two elements in $\mathcal P(\check Y,\uptau)$.
Define a section $\hat{\mathfrak p}$ of $\mathcal V_k \to \calC_k(\check X,\uptau_X)$ over the non-blown-up space as
\begin{equation*}
	\hat{\mathfrak p} = \beta \hat{\mathfrak q} + \beta_0 \hat{\mathfrak p}_0.
\end{equation*} 
Over the blown-up configuration space $\mathcal C_k^{\sigma}(\check X,\uptau_X)$, define $\hat{\mathfrak p}^{\sigma}$ by
\[\hat{\mathfrak p}^{0,\sigma}(A,s,\phi) = \hat{\mathfrak p}^0(A,s\phi), \quad 
 	\hat{\mathfrak p}^{1,\sigma}(A,s,\phi) =
 		(1/s)\hat{\mathfrak p}^1(A,s\phi).\]
The \emph{perturbed Seiberg-Witten operator} $\mathfrak F^{\sigma}_{\mathfrak p}$ is a section:
\[
\mathfrak F^{\sigma}_{\mathfrak p}
= \frakF^{\sigma} + \hat{\mathfrak p}^{\sigma}:
\calC^{\sigma}(\check X, \uptau_X) \to \calV^{\sigma}_{k-1}(\check X, \upiota_X)
\] 
with corresponding moduli space of solutions
\begin{equation*}
	M(\check X,\uptau_X) = \{(A,s,\phi)|\mathfrak F^{\sigma}_{\mathfrak p}(A,s,\phi) = 0\}/\mathcal G_{k+1}(\check X,\upiota_X).
\end{equation*}
There is a larger space $\tilde M(\check X, \uptau_X) \subset \tilde B_k^{\sigma}(\check X,\uptau_X)$ defined by dropping the condition $s \ge 0$.

Adjoining an infinite cylinder $(\check Z,\upiota) = ([0,\infty) \times \check Y, \Id \times \upiota)$, we define the configuration space on the end-cylindrical manifold $\check X^* = \check X \cup_{\check Y} \check Z$:
\begin{equation*}
	\mathcal C^{\sigma}_{k,\loc}(\check X^*,\uptau_X) =
	\left\{(A,\reals^{\ge}\phi,\Phi) 
	\in \mathcal A_{k,\loc} \times \mathbb S \times
	L^2_{k,\loc}(\check X^*;S^+)^{\uptau_X} 
	\bigg| \Phi \in \reals^{\ge}\phi \right\}.
\end{equation*}
Instead of taking the unit $L^2$-sphere, we will model the sphere $\bbS$ as the equivalence classes of $\reals^{\ge}\phi$ under positive scaling, as a spinor on $\check X^*$ can have infinite $L^2$-norm.

The perturbed Seiberg-Witten operator $\mathfrak F^{\sigma}_{\mathfrak p} = \frakF^{\sigma} + \hat{\mathfrak p}^{\sigma}$ can be defined over $\check X^*$. 
In particular, the perturbation is set to be $\hat{\mathfrak q}^{\sigma}$ on the cylindrical end $\check Z$, extending the perturbation previously defined on the compact part $\check X$.
Let $
		M(\check X^*,\uptau_X;[\mathfrak b]) \subset
		\mathcal B^{\sigma}_{k,\loc}(\check X^*,\uptau_X)$
be the set of all $[\upgamma]$ such that $\mathfrak F^{\sigma}_{\mathfrak p}(\upgamma) = 0$ and such that the restriction of $[\upgamma]$ is asymptotic to $[\mathfrak b]$ on the cylindrical end $\check Z$. 

The above definitions extend to families.
Let $P$ be a smooth manifold, possibly with boundary.
Suppose $\{g^p:p \in P\}$ is a family of $\upiota_X$-invariant bifold metrics that all contain isometric copies of a collar $(I \times \check Y, \upiota)$.
Assume $\mathfrak p_0^P \in \mathcal P(\check Y,\uptau)$ is a smooth family of perturbations. 
Define a family of perturbations on $(\check X^*,\uptau_X)$:
\begin{equation*}
	\mathfrak p^p = \beta(t)\mathfrak q + \beta_0(t)\mathfrak p^p_0.
\end{equation*}
An individual moduli space is denoted as $M(\check X^*,\uptau_X; [\mathfrak b])_p$ and the total space is a subset of $P \times \mathcal B^{\sigma}_{k,loc}(\check X^*, \uptau_X)$, as the union
\[
	M(\check X^*,\uptau_X; [\mathfrak b])_P
	= 
	\bigcup_p \{p\}
 \times M(\check X^*, \uptau_X; [\mathfrak b])_p	
\]	
It is convenient to describe $M(\check X^*,\uptau_X;[\mathfrak b])$ as the fibre product $\textsf{Fib}(R_+,R_-)$
of the restriction maps $R_{\pm}$ to $\check Y \times \{0\}$:
\begin{align*}
	R_+: M(\check X,\uptau_X) \to \calB^{\sigma}_{k-1/2}(\check Y,\uptau),\quad
	R_-:M(\check Z,\uptau_Z;[\frakb]) \to \calB^{\sigma}_{k-1/2}(\check Y,\uptau)
\end{align*}
where $M(\check X,\uptau_X)$ and $M(\check Z,\uptau_Z;[\frakb])$ 
are infinite-dimensional Hilbert manifolds with boundaries by~\cite[Proposition~24.3.1]{KMbook2007}, regardless of perturbations.
The restriction map 
\[\rho:M(\check X^*,\uptau_X)
	\to M(\check X,\uptau_X) \times M(\check Z,\uptau_Z;[\frakb])
	\] 
is a homeomorphism onto $\textsf{Fib}(R_+,R_-)$, by~\cite[Lemma~24.4.2]{KMbook2007}.
Moreover, the linearization 
\begin{equation}
\label{eqn:linearized_res}
	\calD_{[\upgamma_1]} R_+ + \calD_{\upgamma_2}R_-:
T_{[\upgamma_1]}M(\check X,\uptau_X) \oplus
T_{[\upgamma_2]}M(\check Z,\uptau_Z;[\frakb]) \to T_{[\frakb]}\calB^{\sigma}_{k-1/2}(\check Y,\uptau)
\end{equation} 
is Fredholm~\cite[Lemma~24.4.1]{KMbook2007}. 

Given an element $[\upgamma]$ of the moduli space $M(\check X^*,\uptau_X;[\mathfrak b])$. 
If $[\upgamma]$ is irreducible, then the space $M(\check X^*, \uptau_X;[\mathfrak b])$ is \emph{regular} at $[\upgamma]$, if $R_+$ and $R_-$ are transverse at $\rho[\upgamma]$. 
Suppose $[\upgamma]$ is reducible. 
Then the moduli space is \emph{regular} at $[\upgamma]$, if the restrictions
\[
		R_+: M^{\red}(\check X,\uptau_X) 
		\to \prod_{\alpha} \mathcal B^{\sigma}_{k-1/2} (\check Y^{\alpha}, \uptau^{\alpha}), \quad
		R_-: M^{\red}(\check Z,\uptau_Z; [\mathfrak b]) 
		\to \prod_{\alpha} \mathcal B^{\sigma}_{k-1/2} (\check Y^{\alpha},\uptau^{\alpha})
\]
are transverse at $\rho[\upgamma]$. 
The moduli space $M(\check X^*,\uptau;[\mathfrak b])$ is \emph{regular} if it is regular at all points.
The following regularity theorem can be proved the same way as \cite[Proposition~24.4.7]{KMbook2007}.
\begin{prop}
	Let $\frakq$ be an admissible perturbation for $(\check Y,\uptau)$ and $\frakb$ be a critical point. 
	Let $\hat{\frakp} = \beta(t)\hat{\frakq} + \beta_0(t)\hat{\frakp}_0$ on the $I \times \check Y$ collar of $\check X$. 
	Then there exists a residual subset of $\calP(\check Y,\uptau)$ such that for all $\frakp_0$ in this subset, the moduli space $M(\check X^*,\uptau_X;[\frakb])$ is regular.
	\hfill \qedsymbol
\end{prop}
A regular moduli space is a manifold possibly with boundary. More precisely, we have the following counterpart~\cite[Proposition~24.4.3]{KMbook2007}, with identical proofs.
\begin{prop}
	Let $[\mathfrak b]$ be a critical point that restricts to $[\mathfrak b^{\alpha}]$ on the $\alpha$-component of $\check Y$.
Suppose moduli space $M(\check X^*,\uptau_X;[\mathfrak b])$ is regular and nonempty.
Then $M(\check X^*,\uptau_X;[\mathfrak b])$ is one of the following:
\begin{enumerate}
		\item[(i)] a smooth manifold containing only irreducibles, if any $[\mathfrak b^{\alpha}]$ is irreducible;
		\item[(ii)] a smooth manifold containing only reducibles, if any $[\mathfrak b^{\alpha}]$ is reducible and boundary-unstable;
		\item[(iii)] a smooth manifold with possibly nonempty boundary, if all $[\mathfrak b^{\alpha}]$ are reducible and boundary-stable.
		The boundary is given by the reducible  elements.
	\end{enumerate}
If case~(ii) occurs and more than one of the $[\mathfrak b^{\alpha}]$ is boundary-unstable, then the solution $[\upgamma]$ is \emph{boundary-obstructed}.
If $c+1$ of the $[\mathfrak b^{\alpha}]$ are boundary-unstable, then $[\upgamma]$ is boundary-obstructed with \emph{corank} c. \hfill \qed
\end{prop}

Suppose we are in the setup of parametrized moduli spaces above.
Given $(p,[\upgamma])$ in $M(\check X^*,\uptau_X;[\frakb])_P$, set $\rho[\upgamma] = ([\upgamma_0],[\upgamma_1])$. 
If $[\upgamma]$ is irreducible, then the moduli space $M(\check X^*,\uptau_X;[\frakb])_P$ is \emph{regular} at $(p,[\upgamma])$ if the maps
\[
	R_+:M(\check X^*,\uptau_X;[\frakb])_P \to \calB^{\sigma}_{k-1/2}(\check Y,\uptau), \quad
	R_-:M(\check Z,\uptau_Z;[\frak b]) \to \calB^{\sigma}_{k-1/2}(\check Y,\uptau)
\]
are transverse at $(p,[\upgamma]),[\upgamma_1])$.
On the other hand, assume $[\upgamma]$ is reducible.
Then the moduli space is \emph{regular} at $[\upgamma]$ if the restrictions
\[
	R_+:M^{\red}(\check X,\uptau_X)_P
	\to
	\prod_{\alpha} \del \calB^{\sigma}_{k-1/2}(\check Y^{\alpha},\uptau^{\alpha}),
	\quad
	R_-:M^{\red}(\check Z,\uptau_Z;[\frakb]) \to \prod_{\alpha} \del \calB^{\sigma}_{k-1/2}(\check Y^{\alpha},\uptau^{\alpha})
\]
are transverse at $(p,[\upgamma],[\upgamma_1])$.
The arguments in \cite[Propositon~24.4.10]{KMbook2007} applies verbatim to the family transversality theorem below.
\begin{prop}
	Let $\frakq$ be an admissible perturbation for $(\check Y,\upiota)$ and $\hat{\frakp}^p$ be a family of perturbations, so that $\frakp^{p} = \beta\frakq + \beta_0\frakp^p_0$ as above.
	Let $P_0 \subset P$ be a closed subset such that $M(\check X^*,\uptau_X;[\frakb])_P$ is regular at points $(p_0,[\upgamma])$ for $p_0 \in P_0$.
	Then there is a new family of perturbations $\tilde \frakp^p$ satisfying
	$\tilde\frakp^p = \frakp^p$, for all $p\in P_0$,
	such that the moduli space $M(\check X^*,\uptau_X;[\frakb])_P$ for $\tilde \frakp^p$ is regular everywhere. 
	\qed 
\end{prop}
Let $\boldsymbol{\calB}^{\sigma}_k(\check X, \upiota_X)$ and $\boldsymbol{\calB}^{\sigma}_{k,loc}(\check X, \upiota_X)$ be the union of $\mathcal B^{\sigma}_k(\check X^*, \fraks_X, \uptau_X)$, $\mathcal B^{\sigma}_{k,loc}(\check X^*, \fraks_X,\uptau_X)$ over all real bifold \spinc structures $(\fraks_X,\uptau_X)$, respectively.
Moreover, let $M_k(\check X, \upiota_X;[\frakb])$ be the union over all $M(\check X^*,\fraks_X,\uptau_X;[\frakb])$ for which the restriction of $(\fraks_X,\uptau_X)$ to $\check Y$ is $(\fraks,\uptau)$, as a subset of $\boldsymbol{\calB}^{\sigma}_{k,loc}(\check X, \upiota_X)$.
Define $\boldB^{\sigma}(\check X, \upiota_X;[\frakb])$ to be the fibre over $[\frakb]$ of the restriction map
\[
	\boldB^{\sigma}(\check X,\upiota_X) 
	\dashrightarrow 
	\boldB^{\sigma}(\check Y,\uptau).
\]
We can partition $M(\check X, \upiota_X;[\frakb])$
according to homotopy classes $z \in \pi_0(\boldB^{\sigma}(\check X, \upiota_X;[\mathfrak b]))$:
\[
M(\check X, \upiota_X;[\frakb])=
\bigcup_z M_z(\check X, \upiota_X;[\frakb]).
\]
Fix such a $z$.
Let $[\upgamma] \in \mathcal B^{\sigma}(\check X,\upiota_X;[\mathfrak b])$ be represented by $\upgamma$, and
let $[\upgamma_{\mathfrak b}]$ be the constant trajectory corresponding to $[\frakb]$ in $\mathcal B^{\tau}(\check Z,\uptau_Z)$.
Then there is a Fredholm operator $\calQ^{\sigma}_{\upgamma} = \calD_{\upgamma}\frakF^{\sigma}_{\frakp} \oplus  \mathbf{d}^{\sigma,\dag}_{\upgamma}$ on $\check X$, analogous to \eqref{eqn:Q_operator} (see also~\cite[Equation~(24.7)]{KMbook2007}).
Also, there is a translation invariant operator $\mathcal Q_{\upgamma_{\mathfrak b}}$ on $\check Z$.
Consider the restriction maps on the respective kernels of $\calQ^{\sigma}_{\upgamma}$ and $\mathcal Q_{\upgamma_{\mathfrak b}}$:
\[r_+ : \Ker(\mathcal Q^{\sigma}_{\upgamma}) 
	\to L^2_{k-1/2}(Y; iT^*\check Y \oplus S \oplus i\reals)^{-\uptau} 
\text{ and }
r_- : \Ker(\mathcal Q_{\upgamma_{\mathfrak b}}) 
	\to L^2_{k-1/2}(Y; iT^*\check Y \oplus S \oplus i\reals)^{-\uptau}.
\]
	
We define $\gr_z(\check X,\upiota_X;[\mathfrak b])$ to be the index of 
\begin{equation*}
	r_+ - r_-:\ker(\mathcal Q^{\sigma}_{\upgamma}) \oplus \ker(\mathcal Q_{\upgamma_{\mathfrak b}})
	\to L^2_{k-1/2}(Y; iT^*Y \oplus S \oplus i\reals)^{-\uptau}.
\end{equation*}
If a nonempty regular moduli space boundary-unobstructed, then its dimension is equal to $\gr_z(\check X,\upiota_X;[\mathfrak b])$.
If $M_z(\check X,\upiota_X;[\mathfrak b])$ is boundary-obstructed of corank $c$, then its dimension is $\gr_z(\check X,\upiota_X;[\mathfrak b]) + c$.
\begin{defn}
	Given a critical point $[\mathfrak b]$,
	a \emph{broken $\check X$-trajectory asymptotic to $[\mathfrak b]$} comprises the data of
	\begin{itemize}
		\item an element $[\upgamma_0]$ in a moduli space $M_{z_0}(\check X^*,\upiota_X;[\mathfrak b_0])$; and
		\item for each component $\check Y^{\alpha}$, an unparametrized broken trajectory $[\check{\upgamma}^{\alpha}]$ in a moduli space $\check M_{z^{\alpha}}^+([\mathfrak b_0^{\alpha}],[\mathfrak b^{\alpha}])$, where $[\mathfrak b_0^{\alpha}]$ is the restriction of $[\mathfrak b_0]$ to $\check Y^{\alpha}$.	
	\end{itemize}
In practice, we denote a broken $\check X$-trajectory by $([\upgamma_0],[\boldsymbol{\check{\upgamma}}])$, where $[\boldsymbol{\check{\upgamma}}]$ represent the possibly empty collection $[\check\upgamma^{\alpha}_i]$ of unparametrized trajectories on the components $\check Y^{\alpha}$, $1 \le i \le n^{\alpha}$.
Moreover, if $z_1$ is the homotopy class of paths from $[\mathfrak b_0]$ to $[\mathfrak b]$ whose $\alpha$-th component is $z^{\alpha}$, then the \emph{homotopy class} of the broken $\check X$-trajectory is the element
	\begin{equation*}
		z = z_1 \circ z_0 \in \pi_0(\boldB^{\sigma}(\check X,\upiota_X;[\mathfrak b])).
	\end{equation*} 
\end{defn}
\begin{defn}
Let $M^+_z(\check X^*,\upiota_X;[\mathfrak b])$ be the space of $X$-trajectories in the homotopy class $z$, topologized in the same way as~\cite[Page~485-486]{KMbook2007}.
Furthermore, let $\bar M_z(\check X^*,\upiota_X;[\mathfrak b])$ be the image of  $M^+_z(\check X^*,\upiota_X;[\mathfrak b])$ under the map
\[
r: M^+_z(\check X^*,\upiota_X;[\mathfrak b]) \to
\boldB^{\sigma}_{k,loc}(\check X^*,\upiota_X), \quad
([\upgamma_0],[\check{\boldsymbol{\upgamma}}])
\mapsto
[\upgamma_0].
\]
In particular, $\bar M_z(\check X^*,\upiota_X;[\mathfrak b])$ is a smaller (coarser) compactification.
For a family over $P$, the compactification is defined fibrewise:
\begin{equation*}
	M^+(\check X^*,\upiota_X;[\mathfrak b])_P
	= 
	\bigcup_{p} \{p\} \times M^+(\check X^*,\upiota_X;[\mathfrak b])_p.
\end{equation*}
Similarly, we define $\bar M_z(X^*,\upiota_X;[\mathfrak b_0])_P$ as the image of the map
\[
M^+_z(\check X^*,\upiota_X;[\mathfrak b])_P \to P \times
\boldsymbol{\calB}^{\sigma}_{k,loc}(\check X^*,\upiota_X), \quad
(p,[\upgamma_0],[\check{\boldsymbol{\upgamma}}])
\mapsto
(p,[\upgamma_0]).
\]
\end{defn}
We state the compactness theorem for families, which contains the case when $P = \text{point}$.
\begin{thm}
If $M(\check X^*,\upiota_X;[\mathfrak b])_p$ is regular for every $[\mathfrak b]$ and $p$, then for each $[\frakb]$, the map $M_z^+(\check X^*,\upiota_X;[\mathfrak b])_P \to P$ is proper.
\end{thm}
\begin{proof}
The proof for families can be adapted from compactness of a single moduli space, which in turn, from the compactness on a cylinder, i.e. Theorem~\ref{thm:compactness_cylinder}.
The proof in the ordinary monopole setting is in \cite[Section~24.5-6]{KMbook2007}, which requires merely notational adjustments for real bifolds.
\end{proof}
The compactification $M_z^+(\check X^*,\upiota_X;[\mathfrak b])_P$ is a disjoint union of subspaces of the form
\begin{equation}
\label{eqn:elt_in_famil_cpct}
	M_{z_0}(\check X^*,\upiota_X;[\mathfrak b_0])_P
	\times
	\prod_{\alpha}
	\check M^+_{z^{\alpha}}([\mathfrak b^{\alpha}_0],[\mathfrak b^{\alpha}]),
\end{equation}
where each of the factors is stratified by manifolds.
We write a typical element $([\upgamma_0],[\boldsymbol{\upgamma}])$ of $M_z^+(\check X^*,\upiota_X;[\mathfrak b])_P$ as an element in  \eqref{eqn:elt_in_famil_cpct} above, where the $\alpha$-th factor of $[\boldsymbol{\upgamma}]$
consists of $n^{\alpha}$ components, such that each component is of the form
\[
	[\check\upgamma_i^{\alpha}] \in \check M_{z_i^{\alpha}}([\frakb^{\alpha}_{i-1}], [\frakb^{\alpha}_i]),
\] 
for $1 \le i \le n^{\alpha}$.

The structure near the codimension-$1$ strata in compactified moduli spaces is summarized in the following Proposition and Theorem.
They correspond to Proposition~24.6.10 and Theorem~24.7.2 in \cite{KMbook2007}, whose proofs can be translated into out setup.
\begin{prop}
\label{prop:d1strata_X}
Assume $M_z(\check X^*, \upiota_X;[\frakb])_P$ is a $d$-dimensional moduli space that contains irreducible solutions.
	Then both $M_z^+(\check X^*,\upiota_X;[\mathfrak b])_P$ and $\bar M_z(\check X^*,\upiota_X;[\mathfrak b])_P$ are $d$-dimensional spaces stratified by manifolds, where $M_z(\check X^*,\upiota_X;[\frakb])_P$ constitute the top stratum. 
Then the elements of the $(d-1)$-dimensional stratum in $M^+_z(\check X^*, \upiota_X;[\frakb])_P$ satisfies one of the following properties.
\begin{enumerate}[(i)]
\item There is a unique component $\alpha_*$ for which $n^{\alpha_*} = 1$, while all other $n^{\alpha} = 0$. Neither $[\check\upgamma_1^{\alpha_*}]$ nor $[\upgamma_0]$ can be boundary-obstructed.
\item There is a unique component $\alpha_*$ for which $n^{\alpha_*} = 2$, while all other $n^{\alpha} = 0$.
	In this case, $[\check\upgamma_1^{\alpha_*}]$ is boundary-obstructed while $[\check\upgamma_2^{\alpha_*}]$  and $[\upgamma_0]$ are not.
\item The solution $[\upgamma_0]$ is boundary-obstructed of corank-$c$, and $n^{\alpha_*}$ for exactly $c+1$ components, and other $n^{\alpha}$ are zero. All such $\check\upgamma_1^{\alpha_*}$ cannot be boundary-obstructed.
\item Unbroken and reducible. In this case $M_z(\check X^*,\upiota_X;[\frakb])$ contains both irreducibles and reducibles.
\item Unbroken, irreducible, and lies above $\del P$.
\end{enumerate}
In the first three cases, if any of the moduli spaces contains both reducibles and irreducibles, then only the irreducible contribute to the $(d-1)$-dimensional stratum.
The $(d-1)$-dimensional strata in $\bar M_z(\check X^*,\upiota_X;[\mathfrak b_0])_P$ are image under $r$ of the strata described above with the additional constraint that, in the first three cases, each of the $[\upgamma_i^{\alpha_*}]$ belongs to an $1$-dimensional moduli space. \qed
\end{prop}
\begin{thm}
	Suppose the moduli space $M_z(\check X^*,\upiota_X;[\mathfrak b])_P$ is $d$-dimensional and contains irreducibles, so that $M^+(\check X,\upiota_X;[\mathfrak b])$ is a $d$-dimensional space stratified by manifolds having $M_z(\check X^*,\upiota_X;[\mathfrak b])_P$ as its top stratum.
	Let $M' \subset M_z^+(\check X^*,\upiota_X;[\mathfrak b])_P$ be any component of the codimension-1 stratum. Then along $M'$, the moduli space $M^+(\check X^*,\upiota_X;[\mathfrak b])_P$ is either a $C^0$-manifold with boundary, or has a codimension-$c$ $\delta$-structure.
	\qed 
\end{thm}

\subsection{Moduli spaces over cobordisms}
We switch notations.
Suppose $(\check W, \upiota_W)$ is an cobordism with oriented boundaries $\del \check W = - (\check Y_-,\upiota_-) \sqcup (\check Y_+,\upiota_+)$.
Fix perturbations $(\mathfrak p_-,\mathfrak p_+)$ supported over the collars of $\check Y_{\pm}$.
The end-cylindrical manifolds are formed by adding $(-\infty,0] \times \check Y_-$ and $\check Y_+ \times [0,\infty)$.
Denote the disjoint union of configuration spaces, over real bifold spin\textsuperscript{c} structures, as
\[\boldsymbol{\mathcal B}^{\sigma}(\check M, \upiota) 
	= \coprod_{(\mathfrak s,\uptau) \in \SpincR(\check M,\upiota)}
	\mathcal B^{\sigma}(\check M, \upiota;\mathfrak s, \uptau),\] 
where $\check M = \check W,\check Y_-,\check Y_+$.
Let $[\fraka]$ and $[\frakb]$ be critical points on $\check Y_-$ and $\check Y_+$, respectively.
Let $M([\fraka],\check W^*,\upiota_W,[\frakb])$ be the moduli space of Seiberg-Witten solutions.

A \emph{$\check W$-path} $z$ from $[\mathfrak a]$ to $[\mathfrak b]$ is an element $[\upgamma]$ in $\boldsymbol{\mathcal B}^{\sigma}(\check W,\upiota_W)$ which under the  partially defined restriction map
$r: \boldsymbol{\mathcal B}^{\sigma}(\check W,\upiota_W)
	\to 
	\boldsymbol{\mathcal B}^{\sigma}(\check Y_-,\upiota_-)
	\times
	\boldsymbol{\mathcal B}^{\sigma}(\check Y_+,\upiota_+)$
is the given pair: $r([\upgamma]) = ([\mathfrak a],[\mathfrak b])$. Two $\check W$-paths are \emph{homotopic} if they belong to the same path component of the fibre $r^{-1}([\mathfrak a],[\mathfrak b])$. 
We write
$\boldsymbol{\pi}([\mathfrak a], \check W, \upiota_W, [\mathfrak b])$
for the set of homotopic classes of $\check W$-paths.

The moduli space $M_z([\fraka],\check W,\upiota_W, [\frakb])$ is \emph{boundary-obstructed (with corank-1)} if $[\mathfrak a]$ is boundary-stable and $[\mathfrak b]$ is boundary-unstable.
Given $z \in \boldsymbol{\pi}([\mathfrak a], \check W, \upiota_W, [\mathfrak b])$, the 
\emph{grading} $\gr_z([\fraka],\check W, \upiota_W,[\frakb])$ is as the notation $\gr_z(\check W, \upiota_W;[\frakb])$ before.

The compactification $M_z^+([\mathfrak a], \check W^*,\upiota_W , [\mathfrak b])$ is formed by adding \emph{broken trajectory}, consisting of triples of the form $([\boldsymbol{\check\upgamma}_-],[\check\upgamma_0],[\boldsymbol{\check\upgamma}_+])$, where
\[
[\boldsymbol{\check\upgamma}_-] \in \check M^+([\fraka],[\fraka_0]), \quad
[\boldsymbol{\check\upgamma}_+] \in \check M^+([\frakb_0],[\fraka_0]), \quad
[\upgamma_0] \in M([\fraka_0],\check W^*,\upiota_W,[\frakb_0]).
\]
\begin{prop}
\label{prop:W_cod1_strata}
	Assume $M_z([\mathfrak a], \check W^*,\upiota_W , [\mathfrak b])$ has dimension $d$ and contains irreducible solutions.
	The compactifications $M_z^+([\mathfrak a], \check W^*,\upiota_W, [\mathfrak b])$ and $\bar M_z([\mathfrak a], \check W^*,\upiota_W, [\mathfrak b])$ are $d$-dimensional space stratified by manifolds, whose top stratum is the irreducible part of $M_z([\mathfrak a], \check W^*,\upiota_W, [\mathfrak b])$.
	The $(d-1)$-dimensional stratum in $M^+_z([\mathfrak a], \check W^*,\upiota_W, [\mathfrak b])$ comprises elements of the types:
\begin{align*}
	\check M_{-1} &\times M_0	\\
	M_0 &\times \check M_1\\
	\check M_{-2} \times &\check M_{-1} \times M_0\\
	\check M_{-1} \times &M_0 \times \check M_1\\
	M_0 \times &\check M_1 \times \check M_2,
\end{align*}
and finally
\begin{equation*}
	M_z^{\red}([\mathfrak a], \check W^*, \upiota_W, [\mathfrak b])
\end{equation*}
in the case that the moduli space contains both reducibles and irreducibles.
We used $M_0$ to denote a moduli space on $(\check W^*,\upiota_W)$. 
Moreover, $\check M_{-n}$ and $\check M_n$ $(n > 0)$ indicate typical unparametrized moduli spaces on $(\check Y_-,\upiota_-)$ and $(\check Y_+,\upiota_+)$, which changes from line to line.
In the strata with three factors, the middle factor is boundary-obstructed.
Suppose in a stratum the unparametrized moduli spaces on the cylinder $\check M_{-1},\check M_1$ are all $1$-dimensional.
Then this stratum has codimension $1$ in the compactification $\bar M_z([\mathfrak a], \check W^*,\upiota_W, [\mathfrak b])$.

The reducible moduli space $M_z^{\red}([\fraka],\check W^*,\upiota_W,[\frakb])$ can be compactified similarly to a space $M_z^{\red+}([\fraka],\check W^*,\upiota_W,[\frakb])$, stratified by manifolds, whose codimension-$1$ strata are of the forms
\begin{align*}
	M_0^{\red} &\times \check M^{\red}_1, \\
	M_{-1}^{\red} &\times \check M^{\red}_0. 
\end{align*}
\end{prop}
\begin{proof}
	This proposition is special case of Proposition~\ref{prop:d1strata_X}.
	See also \cite[Proposition~25.1.1]{KMbook2007}.
\end{proof}

Let $u \in H^*(\boldB^{\sigma}(\check W, \upiota_W))$ be a cohomology class. 
Following the approach in \cite[Section~25]{KMbook2007}, we aim to define maps of the form
\begin{equation*}
	\HMR^{\circ}(u|\check W, \upiota_W):
	\HMR^{\circ}(\check Y_-,\upiota_-) \to \HMR^{\circ}(\check Y_+,\upiota_+)
\end{equation*}
which evaluates the class $u$ while applying the $\HMR^{\circ}$ cobordism map. 

Let $d_0 > 0$.
Suppose $(z,[\mathfrak a],[\mathfrak b])$ is such that $\bar M_z([\mathfrak a],\check W^*,\upiota_W,[\mathfrak b])$ or $\bar M_z^{\text{red}}([\mathfrak a],\check W^*,\upiota_W,[\mathfrak b])$ has at most dimension $d_0$.
Their compatifications 
\[\bar M_z([\mathfrak a],\check W^*,\upiota_W,[\mathfrak b])\text{ and }\bar M_z^{\text{red}}([\mathfrak a],\check W^*,\upiota_W,[\mathfrak b])
\] 
form a locally finite collection of closed subsets of $\boldB^{\sigma}_{k,\loc}(\check W^*,\upiota_W)$.
By Lemma~\ref{lem:exist_refine}, every open cover of $\boldB^{\sigma}_{k,\loc}$ has a refinement that is transverse to all strata in all compactified moduli spaces $\bar M$ and $\bar M^{\red}$ of dimension at most $d_0$.
Let $\mathcal U$ be such an open cover transverse to these moduli spaces and  $u \in C^d(\mathcal U;\Ftwo)$ be a \v{C}ech cochain with $d \le d_0$.
If $M_z([\mathfrak a],\check W^*,\upiota_W,[\mathfrak b])$ has dimension $d$, there is a well-defined evaluation in the sense of Section~\ref{subsec:stokes}:
\begin{equation*}
	\langle u, [M_z([\mathfrak a], \check W^*, \upiota_W,[\mathfrak b])]\rangle \in \Ftwo,
\end{equation*}
where the evaluation is zero if the dimension of the moduli space is not $d$.

\begin{defn}
	\label{defn:entry_m}
We define a number of chain maps as follows.
\begin{align*}
	m^o_o:C^d(\mathcal U;\Ftwo) \otimes C^o_\bullet(\check Y_-, \upiota_-) &\to C^o_\bullet(\check Y_+, \upiota_+), \quad
u \otimes [\mathfrak a]\mapsto \sum_{[\mathfrak b] \in \mathfrak C^o(\check Y_+, \upiota_+)} \sum_z 
	\langle u, [M_z([\mathfrak a], \check W^*,\upiota_W,[\mathfrak b])]\rangle [\mathfrak b],\\
	m^o_s:C^d(\mathcal U;\Ftwo) \otimes C^o_\bullet(\check Y_-, \upiota_-) &\to C^s_\bullet(\check Y_+, \upiota_+), \quad
u \otimes [\mathfrak a]~\mapsto \sum_{[\mathfrak b] \in \mathfrak C^s(\check Y_+, \upiota_+)} \sum_z 
	\langle u, [M_z([\mathfrak a], \check W^*,\upiota_W,[\mathfrak b])]\rangle [\mathfrak b],\\
	m^u_o:C^d(\mathcal U;\Ftwo) \otimes C^u_\bullet(\check Y_-, \upiota_-) &\to C^o_\bullet(\check Y_+, \upiota_+), \quad
u \otimes [\mathfrak a]\mapsto \sum_{[\mathfrak b] \in \mathfrak C^o(\check Y_+, \upiota_+)} \sum_z 
	\langle u, [M_z([\mathfrak a], \check W^*,\upiota_W,[\mathfrak b])]\rangle [\mathfrak b],\\
	m^u_s:C^d(\mathcal U;\Ftwo) \otimes C^u_\bullet(\check Y_-, \upiota_-) &\to C^s_\bullet(\check Y_+, \upiota_+), \quad
u \otimes [\mathfrak a]\mapsto \sum_{[\mathfrak b] \in \mathfrak C^s(\check Y_+, \upiota_+)} \sum_z 
	\langle u, [M_z([\mathfrak a], \check W^*,\upiota_W,[\mathfrak b])]\rangle [\mathfrak b].
\end{align*}
Also, using reducible trajectories we define:
\begin{align*}
	\bar m^s_s :C^d(\mathcal U;\Ftwo) \otimes \bar C^s_\bullet(\check Y_-, \upiota_-) 
	&\to C^s_\bullet(\check Y_+, \upiota_+), \quad
	u \otimes [\mathfrak a] 
	\mapsto \sum_{[\mathfrak b] \in \mathfrak C^s(\check Y_+, \upiota_+)} \sum_z 
	\langle u, [M_z^{\red}([\mathfrak a], \check W^*,\upiota_W,[\mathfrak b])]\rangle [\mathfrak b],
	\\
	\bar m^u_u : C^d(\mathcal U;\Ftwo) \otimes C^u_\bullet(\check Y_-, \upiota_-) 
	&\to C^u_\bullet(\check Y_+, \upiota_+), \quad
	u \otimes [\mathfrak a] 
	\mapsto \sum_{[\mathfrak b] \in \mathfrak C^u(\check Y_+, \upiota_+)} \sum_z 
	\langle u, [M_z^{\red}([\mathfrak a], \check W^*,\upiota_W,[\mathfrak b])]\rangle [\mathfrak b],
	\\
	\bar m^s_u :C^d(\mathcal U;\Ftwo) \otimes C^s_\bullet(\check Y_-, \upiota_-) 
	&\to C^u_\bullet(\check Y_+, \upiota_+), \quad
	u \otimes [\mathfrak a] 
	\mapsto \sum_{[\mathfrak b] \in \mathfrak C^u(\check Y_+, \upiota_+)} \sum_z 
	\langle u, [M_z([\mathfrak a], \check W^*,\upiota_W,[\mathfrak b])]\rangle [\mathfrak b],
	\\
	\bar m^u_s:C^d(\mathcal U;\Ftwo) \otimes C^u_\bullet(\check Y_-, \upiota_-) 
	&\to C^s_\bullet(\check Y_+, \upiota_+), \quad
	u \otimes [\mathfrak a] 
	\mapsto \sum_{[\mathfrak b] \in \mathfrak C^s(\check Y_+, \upiota_+)} \sum_z 
	\langle u, [M_z([\mathfrak a], \check W^*,\upiota_W,[\mathfrak b])]\rangle [\mathfrak b].
\end{align*}
\end{defn}
\begin{defn}
The map $\bar m: C^d(\mathcal U;\Ftwo) \oplus \bar C_\bullet(\check Y_-, \upiota_-) \to \bar C_\bullet(\check Y_+, \upiota_+)$ with respect to the decomposition $\bar C_\bullet = C^s_{\bullet} \oplus C^u_{\bullet}$ is given by
\begin{equation*}
	\bar m = \begin{pmatrix}
		\bar m^s_s && \bar m^u_s\\
		\bar m^s_u && \bar m^u_u
	\end{pmatrix}.
\end{equation*}
Over $\check C_\bullet = C^o_\bullet \oplus C^s_\bullet$, we define $\check m: C^d(\mathcal U;\Ftwo) \otimes
	\check C_\bullet(\check Y_-, \upiota_-) \to
	\check C_\bullet(\check Y_+, \upiota_+)$
for $d \le d_0$ to be
\begin{equation*}
	\check m = \begin{pmatrix}
		m^o_o && -m^u_o \bar\del^s_u(\check Y_-) - \del^u_o(\check Y_+)\bar m^s_u\\
		m^o_s && \bar m^s_s - m^u_s\bar\del^s_u(\check Y_-,) -\del^u_s(\check Y_+)\bar m^s_u
	\end{pmatrix},
\end{equation*}
where $\del^u_o(\check Y_+)$ denotes an operator on $\check Y_+$.

Finally  with respect to $\hat C_\bullet = C^o_\bullet \oplus C^u_\bullet$, we define $\hat m:C^d(\mathcal U;\Ftwo) \otimes
	\hat C_\bullet(\check Y_-, \upiota_-) \to
	\hat C_\bullet(\check Y_+, \upiota_+)$ by 
\begin{equation*}
	\hat m = \begin{pmatrix}
		m^o_o && m^u_o\\
		\bar m^s_u \del^o_s(\check Y_-)\sigma  - \bar\del^s_u(\check Y_+)m^o_s 
		&& 
			\bar m^u_s \sigma + \bar m^s_u \del^u_s (\check Y_-) \sigma - \bar\del^s_u(\check Y_+)m^u_s
	\end{pmatrix}.
\end{equation*}
\end{defn}
\begin{prop}
\label{prop:identities_m(uW)}
	The operators $\check m, \hat m,$ and $\bar m$ satisfy the following identities:
\begin{align*}
	(-1)^d\check\del(\check Y_+, \upiota_+) \check m(u \otimes \check \xi)
	&= -\check m(\delta u \otimes \check\xi) + \check m(u \otimes \check\del(\check Y_-, \upiota_-)\check\xi),\\
	(-1)^d\hat\del(\check Y_+, \upiota_+) \hat m(u \otimes \hat \xi) 
	&= -\hat m(\delta u \otimes \hat\xi) + \hat m(u \otimes \hat\del(\check Y_-, \upiota_-)\hat\xi),\\
	(-1)^d\bar\del(\check Y_+, \upiota_+) \bar m(u \otimes \bar \xi) 
	&= -\bar m(\delta u \otimes \bar\xi) + \bar m(u \otimes \bar\del(\check Y_-, \upiota_-)\bar\xi),
\end{align*}
where $\delta$ is the \v{C}ech coboundary map,  $u \in C^d(\mathcal U;\Ftwo)$, and $d \le d_0-1$.
Moreover, $\check\xi \in \check C_\bullet(\check Y_-,\upiota_-)$, and similarly for $\hat\xi$ and $\bar\xi$.
(Of course, over $\bbF_2$ the signs are meaningless - 
they are kept to align with \cite{KMbook2007}.)
Consequently, $\check m, \hat m,$ and $\bar m$ descend to maps
\begin{align*}
	\check m : \check H^d(\mathcal U; \Ftwo) \otimes
	\widecheck{\HMR}_j(\check Y_-, \upiota_-) 
	&\to \widecheck{\HMR}_{k-d}(\check Y_+, \upiota_+),\\
	\hat m : \check H^d(\mathcal U; \Ftwo) \otimes
	\widehat{\HMR}_j(\check Y_-, \upiota_-) 
	&\to \widehat{\HMR}_{k-d}(\check Y_+, \upiota_+),\\
	\bar m : \check H^d(\mathcal U; \Ftwo) \otimes
	\overline{\HMR}_j(\check Y_-, \upiota_-) 
	&\to \overline{\HMR}_{k-d}(\check Y_+, \upiota_+),
\end{align*}
for any open cover $\mathcal U$ of $\boldB^{\sigma}_{k,\loc}(\check W^*,\upiota_W)$ transverse to all the moduli spaces of dimension $d_0$ or less.
\end{prop}
\begin{proof}
	The arguments are the same as that of
	\cite[Proposition~25.3.4, Lemma~25.3.7]{KMbook2007}.
	In summary, one applies Stokes' theorem to compactified $(d+1)$-dimensional moduli spaces $\bar M_z([\fraka],\check W^*,\upiota_W,[\frakb])$.
	A sample identity to be proved is (see \cite[Lemma~25.3.6]{KMbook2007} for the full list of identities)
	\begin{equation}
	\label{eqn:sample_identity_m(uW)}
		m^o_o \del^o_o - \del^o_o m^o_o -m^u_o\bar\del^s_u\del^o_s
		-\del^u_o\bar\del^s_um^o_s
		= m^o_o(\delta \otimes 1).
	\end{equation}
These identities can be derived from the description of boundary contributions in Proposition~\ref{prop:W_cod1_strata}.
\end{proof}

By taking the limit over all open covers of $\boldsymbol{\mathcal B}^{\sigma}_{k,\loc}(\check W^*,\upiota_W)$ transverse to the moduli space, and identifying the \v{C}ech cohomology $\check{H}^d(\boldsymbol{\mathcal B}^{\sigma}_{k,\loc}(\check W^*,\upiota_W); \Ftwo)$ with $H(\boldsymbol{\mathcal B}^{\sigma}_{k,\loc}(\check W^*,\upiota_W); \Ftwo)$, we obtain
\begin{align*}
		\widecheck{\HMR}(u|\check W,\upiota_W) : H^d(\boldB^{\sigma}_{k,\loc}(\check W^*,\upiota_W); \Ftwo) \otimes
		\widecheck{\HMR}_j(\check Y_-, \upiota_-) 
		&\to \widecheck{\HMR}_{k-d}(\check Y_+, \upiota_+),\\
		\widehat{\HMR}(u|\check W,\upiota_W): H^d(\boldB^{\sigma}_{k,\loc}(\check W^*,\upiota_W); \Ftwo) \otimes
		\widehat{\HMR}_j(\check Y_-, \upiota_-) 
		&\to \widehat{\HMR}_{k-d}(\check Y_+, \upiota_+),\\
		\overline{\HMR}(u|\check W,\upiota_W) :H^d(\boldB^{\sigma}_{k,\loc}(\check W^*,\upiota_W); \Ftwo) \otimes
		\overline{\HMR}_j(\check Y_-, \upiota_-) 
		&\to \overline{\HMR}_{k-d}(\check Y_+, \upiota_+).
\end{align*}
The follow proposition ensures $\HMR^{\circ}(u|\check W, \upiota_W)$ is independent of the choices of metrics and perturbations.
\begin{prop}
	\label{prop:cob_map_indep_metric_pert}
	Let $g(0)$ and $g(1)$ be two $\upiota_W$-invariant metrics on $\check W$, isometric in a collar of the boundary to the same cylindrical metric.
	Let $\mathfrak p(0)$ and $\mathfrak p(1)$ be two perturbations on $(\check W,\upiota_W)$, constructed using the same perturbations on $(\check Y_{\pm},\upiota_{\pm})$, and chosen so that the corresponding moduli spaces are regular.
	Let $u$ be a \v{C}ech cocycle in $\boldB^{\sigma}_{k,\loc}(\check W^*,\upiota_W)$ as above.
	Suppose $\check m(0)$ and $\check m(1)$ are defined by the formulae above, using $(g(0),\mathfrak p(0))$ and $(g(1),\mathfrak p(1))$, respectively.
	Then for $d \le d_0$, there is a chain homotopy
	\begin{equation*}
		\check K: C^d(\mathcal U;\Ftwo) \otimes 
		\check C_{\bullet}(\check Y_-, \upiota_-) \to 
		\check C_{\bullet}(\check Y_+, \upiota_+),
	\end{equation*}
	satisfying the identity
	\begin{equation*}
		(-1)^d\check\del \check K(u \otimes \check\xi) = 
		-\check K(\delta u \otimes \check\xi) + 
		\check K(u \otimes \check\del \check\xi)
		+(-1)^d\check m(0)(u \otimes \check\xi)
		-(-1)^d\check m(1)(u \otimes \check\xi).
	\end{equation*}
\end{prop}
\begin{proof}
	The proof in \cite[Proposition~25.3.8]{KMbook2007} applies with only notational changes to our setting.
	To sketch a proof, connect $(g(0),\frakp(0))$ to $(g(1),\frakp(1))$ via a path, parametrized by $P = [0,1]$. 
	Using parametrized moduli spaces $M_z([\fraka],\check W^*,\upiota_W,[\frakb])$ over $P$, 
	one defines a map $\check m(P)$, which satisfies the following chain-homotopy relation
	\[
		(-1)^d\check\del \check m(P)(u \otimes \check\xi) = 
		-\check m(P)(\delta u \otimes \check\xi) + 
		\check m(P)(u \otimes \check\del \check\xi)
		+(-1)^d\check m(0)(u \otimes \check\xi)
		-(-1)^d\check m(1)(u \otimes \check\xi).
	\] 
	The proposition follows by setting $\check K = \check m(P)$.
	The above equation can be broken into several identities, each of which can be proved in a similar manner as Proposition~\ref{prop:identities_m(uW)}.
	An example of such identities, analogous to Equation~\eqref{eqn:sample_identity_m(uW)}, is
	\[
	m^o_o(\del P) + m^o_o(P) \del^o_o - \del^o_o m^o_o(P) -m^u_o(P)\bar\del^s_u\del^o_s
	-\del^u_o\bar\del^s_um^o_s(P)
	= m^o_o(P)(\delta \otimes 1),
	\]
	where the brackets indicate the parameter spaces of families.
	This identity differs from~\eqref{eqn:sample_identity_m(uW)}
	by contributions from the boundary of $P$.
\end{proof}
\begin{prop}
\label{prop:trivialcob}
	If $\check W$ is the trivial cylindrical cobordism from $(\check Y,\upiota;g,\mathfrak q)$ to itself.
	Then $\HMR^{\circ}(\check W,\upiota_W)$ induces the identity map.
\qed
\end{prop}
To state the composition law, let $(\check Y_0,\upiota_0), (\check Y_1,\upiota_1)$, and $(\check Y_2,\upiota_2)$ be real 3-bifolds.
Suppose $(\check W_{01},\upiota_{01}): (\check Y_0,\upiota_0) \to (\check Y_1,\upiota_1)$ and $(\check W_{12},\upiota_{12}): (\check Y_1,\upiota_1) \to (\check Y_2,\upiota_2)$ are two real bifold cobordisms.
Consider the composition corbordism
\[
	(\check W,\upiota_W) = 
	(\check W_{01},\upiota_{01}) \cup_{(\check Y_1,\upiota_1)} 
	(\check W_{12},\upiota_{12}).
\]
Assume $u_{01} \in H^{d_{01}}(\boldsymbol{\mathcal B}^{\sigma}(\check W_{01},\upiota_{W_{01}});\Ftwo)$ and  $u_{12} \in H^{d_{12}}(\boldsymbol{\mathcal B}^{\sigma}(\check W_{12},\upiota_{W_{12}});\Ftwo)$ are two cohomology classes.
The product $u = u_{01}u_{12} \in H^*(\boldsymbol{\mathcal B}^{\sigma}(\check W,\upiota_{W});\bbF_2)$ is defined
as follows.
Let $R_i$ be the restriction map, well-defined over a weak homotopy equivalence subset, 
\begin{equation*}
	R_i: \boldsymbol{\mathcal B}^{\sigma}(\check W,\upiota_W) \to \boldsymbol{\mathcal B}^{\sigma}(\check W_i,\upiota_{W_i}),
\end{equation*}
for $i=01, 12$.
Then the pullback $R_i^*: H^*(\boldsymbol{\mathcal B}^{\sigma}(\check W_i,\upiota_{W_i});\bbF_2) \to H^*(\boldsymbol{\mathcal B}^{\sigma}(\check W,\upiota_W);\bbF_2)$ is well-defined, and
we set
\begin{equation*}
	u_{01} u_{12} = R_{01}^*(u_{01}) \cup R_2^*(u_{12}).
\end{equation*}
The composition law in full generality is as follows.
\begin{prop}
\label{prop:compositionlaw}
	Given $(\check Y_0,\upiota_0), (\check Y_1,\upiota_1)$, and $(\check Y_2,\upiota_2)$ be real 3-bifolds with metrics and admissible perturbations.
	Let $(\check W,\upiota_W)$ be the composition of cobordisms $(\check W_{01}, \upiota_{W_{01}}): (\check Y_0,\upiota_0) \to (\check Y_1,\upiota_1)$ and  $(\check W_{12}, \upiota_{W_{12}}): (\check Y_1,\upiota_1) \to (\check Y_2,\upiota_2)$.
	Assume $u = u_{01}u_{12}$ is the product of cohomology classes $u_{01} \in H^{d_{01}}(\boldsymbol{\mathcal B}^{\sigma}(\check W_{01},\upiota_{W_{01}});\Ftwo)$ and  $u_{12} \in H^{d_{12}}(\boldsymbol{\mathcal B}^{\sigma}(\check W_{12},\upiota_{W_{12}});\Ftwo)$.
	Then
	\begin{equation*}
		\HMR^{\circ}(u|\check W, \upiota_W) 
	=
	\HMR^{\circ}(u_2|\check W_2,\upiota_{W_2}) \circ \HMR^{\circ}(u_1|\check W_1,\upiota_{W_1}).
	\end{equation*}
\end{prop}
\begin{proof}
	We outline the proof in  \cite[Proposition~26.1.2]{KMbook2007}.
	Consider a $[0,\infty)$-family of metrics and perturbations as follows.
	For $S > 0$, we metrically insert an $\upiota_1$-invariant cylinder of length $S$ to obtain the cobordism:
	\[
		\check W(S) = \check W_{01} \cup_{Y_1} ([0,S] \times \check Y_1) \cup_{Y_1} \check W_{12}.
	\]
	The Seiberg-Witten equations will be perturbed near the collars of $\check W_{01}$ and $\check W_{12}$, using the given admissible perturbations $\frakq^i$ on $(\check Y_i,\upiota_i)$.
	Furthermore, we insert perturbations $\frakq^1$ across the cylinder $[0,S] \times \check Y_1$, independent of $t \in [0,S]$.
	Fix critical points $[\fraka],[\frakb]$ on $\check  Y_0$ and $\check Y_1$, respectively.
	Consider the parametrized moduli space over $[0,\infty)$
	\[
		\calM_z([\fraka],[\frakb]) = \bigcup_{S \in [0,\infty)} \{S\} \times M_z([\fraka],\check W^*,\upiota_W, [\frakb]).
	\]
	For suitable choices of perturbations $\frakp$ on $\check W_{01}$ and $\check W_{12}$, one can show that $\mathcal M_z([\fraka], [\frakb])$ is regular, and its boundary is the usual $M_z([\fraka], \check W, \upiota_W, [\frakb])$.
	 
	 We compactify $\calM_z([\fraka],[\frakb])$ to $\calM_z^+([\fraka],[\frakb])$ by adding broken trajectories fibrewise over $[0,\infty)$, and a space $M_z^+([\fraka], \check W(\infty)^*, \upiota_W,[\frakb])$ when the metric is broken. 
	 There is a coarser compactification $\bar{\mathcal M_z}([\fraka],[\frakb])$, defined as the image of 
	 \[
	 	r:\calM^+_z([\fraka],[\frakb]) \to
	 	[0,\infty] \times \boldB^{\sigma}_k(\check W_{01},\upiota_{W_{01}}) \times
	 	\boldB^{\sigma}_k(\check W_{12},\upiota_{W_{12}}).
	 \]
	 The map $r$ is defined by restrictions over $\check W_{01}$ and $\check W_{12}$, away from $[0,S] \times \check Y_1$
	 (see \cite[Proposition~26.1.4]{KMbook2007}).
	 Choose open covers $\calU_{01}$ and $\calU_{12}$ for $\boldB^{\sigma}_k(\check W_{01},\upiota_{W_{01}})$ and $\boldB^{\sigma}_k(\check W_{12},\upiota_{W_{12}})$, respectively.
	 By Lemma~\ref{lem:exist_refine}, there is an open cover $\calV$ of
	 \[
	 	[0,\infty] \times \boldB^{\sigma}_k(\check W_{01},\upiota_{W_{01}}) \times
	 	\boldB^{\sigma}_k(\check W_{12},\upiota_{W_{12}}),
	 \] 
	refining the cover $([0,\infty] \times \calU_{01} \times \calU_{12})$
	such that $\calV$ is transverse to all strata of $\mathcal M_z([\fraka],[\frakb])$.
	
	Assume $\dim \calM_z([\fraka],[\frakb]) = d-1$.
	Using pairing with the ``external product'' $u_{01} \times u_{12} \in  C^{d_{01} + d_{12}}(\calV;\bbF_2)$:
	 \[
	 	\langle u_{01} \times u_{12}, \calM_z([\fraka],[\frakb]) \rangle, 
	 \]
	 we define maps $K^o_o,K^o_s,K^u_o,K^u_s,\bar K^s_s, \bar K^u_u, \bar K^s_u, \bar K^u_s$ by varying the types of $(\fraka,\frakb)$, as in the the definition of $m^o_o$, $m^o_s$, etc.
	 There is a chain homotopy
	 \[
	 	\check K:C^{d_{01}} (\calU_{01};\bbF_2)\otimes C^{d_{12}}(\calU_{12};\bbF_2) \otimes
	 	\check C_{\bullet}(\check Y_0,\upiota_0) \to
	 	\check C_{\bullet}(\check Y_2,\upiota_2),
	 \]
	 which satisfies, schematically,
	 \begin{equation}
	 \label{eqn:double_comp_chain_hmtpy}
	 	\check K \check\del - \check \del \check K - \check m c +\check m \check m \tau = \check K\underline{\delta}. 
	 \end{equation}
	 Here, $\underline{\delta}$ is the \v{C}ech coboundary map on the tensor product complex $C^*(\calU_{01}) \times C^*(\calU_{12})$ and $\tau: C^*(\calU_{01}) \times C^*(\calU_{12}) \to C^*(\calU_{01}) \times C^*(\calU_{12})$ is the map that interchanges the factors.
	 The map $c$ is the ``internal product''
	 as in \cite[Equation~(26.9)]{KMbook2007}.
	 
	 The proof of Equation~\eqref{eqn:double_comp_chain_hmtpy} involves enumeration of various strata and Stokes' theorem.
	 In particular, the first term on the right hand side of~\eqref{eqn:double_comp_chain_hmtpy} comes from breaking of trajectories over $(0,\infty)$ and the second to last term comes from the fibre at $S = \infty$.
	 The last term arises from the fibre at $S=\infty$.
\end{proof}

\subsection{Monopole invariants of closed real 4-bifolds}
Let $(\check X,\upiota_X)$ be a closed real 4-bifold.
Consider the moduli space of Seiberg-Witten solutions in the blown-up configuration
\[
	M(\check X,\uptau_X) = \left\{[\upgamma] \in \calB^{\sigma}_k(\check X, \uptau_X) | \frakF^{\sigma}(\upgamma) = 0\right\},
\]
where
\[
	\frakF^{\sigma}(A,s,\phi) = \left(\frac12 \rho(F^+_A) - s^2(\phi\phi^*)_0, D^+_A\phi\right).
\]
As in the classical setup, we perturb the Seiberg-Witten operator by an imaginary-valued $2$-form $\omega$.
Writing $\omega^+$ for the self-dual part of $\omega$, define
\[
	\frakF^{\sigma}_{\omega} = \left(\frac12 \rho(F^+_A - 4\omega^+) - s^2(\phi\phi^*)_0, D^+_A\phi\right).
\]
\begin{lem}
\label{lem:residual_closed_manifold}
	There is a residual set of $(-\upiota^*)$-invariant $2$-forms in $L^2_{k-1}(\check X;i\Lambda^2(\check X))$ such that for each $\omega$,
	the section $\frakF^{\sigma}_{\omega}$ is transverse to zero and the corresponding moduli space $\tilde M(\check X,\uptau_X)$ is a smooth, compact manifold of dimension $d(\check X,\upiota_X)$.
	The space $M(\check X,\uptau_X)$ with $s \ge 0$ imposed is a smooth manifold possibly with boundary, and can be identified with the quotient of $\tilde M(\check X,\uptau_X)$ by $s \mapsto -s$.
	If $b^+_{-\upiota^*}(\check W) > 0$, then moduli space $M(\check X,\uptau_X)$ has empty boundary.
\end{lem}
\begin{proof}
	To sketch a proof, let $\tilde Z \subset \tilde{\calB}^{\sigma}_k(\check X,\uptau_X)$ be the locus of the solutions to the Dirac equation $D_A^+\phi = 0$.
	This is a Hilbert submanifold as the linearization of the Dirac operator is surjective.
	The solutions to the Seiberg-Witten equations $M(\check X,\uptau_X)$ are the fibres over $2\rho(\omega^+)$ of the map
	\[
		\varpi: \tilde{Z} \to L^2_{k-1}(\check X;i\su(S^+))^{\uptau}
	\]
	given by 
	\[
		(A,s,\phi) \mapsto \frac{1}{2}\rho(F^+_{A^t}) - s^2(\phi\phi^*)_0.
	\]
	This is a Fredholm map of index for which we denote as $d(\check X,\upiota_X)$. So there is a residual set of $\omega^+$ for which $M(\check X,\uptau_X)$ is regular.
	The last statement follows because the boundary is nonempty, or equivalent a reducible solution exists only if
	\[
		2\int_{\check X} \omega \wedge \kappa =
		-(2\pi i c_1(S^+) \cup [\kappa])[\check X],
	\]
	which is a codimension-$b^+_{-\upiota^*}(\check W)$ condition.
	This is the analogue of \cite[Lemma~27.1.1]{KMbook2007}.
\end{proof}
\begin{prop}
	Suppose $(\check X,\upiota_X)$ is the real double cover of a foam with $1$-set $(\Sigma,s_{\Sigma})$ in $S^4$ such that $\tilde \Sigma^{\sfc}$ is orientable.
	Given a real \spinc structure $(\fraks_X,\uptau_X)$, the formal dimension $d(\check X,\uptau_X)$ of the Seiberg-Witten moduli space $M(\check X, \uptau_X)$ is given by
	\[
	\frac12b_1(\Sigma^{\sfr}) - \frac14 [\Sigma^{\sfr}]\cdot [\Sigma^{\sfr}] + \frac{c_1(S^+)^2-\sigma(X)}{8}+\frac18\langle c_1(S^+),\tilde \Sigma^{\sfc}\rangle + \frac{1}{32}[\tilde \Sigma^{\sfc}] \cdot [\tilde \Sigma^{\sfc}],
	\]
	where $[\Sigma^{\sfr}]\cdot [\Sigma^{\sfr}]$ is the euler number of $\Sigma^{\sfr}$ in $S^4$ and $[\tilde \Sigma^{\sfc}] \cdot [\tilde \Sigma^{\sfc}]$ is the self-intersection of $\tilde \Sigma^{\sfc}$ in $X$.
\end{prop}
\begin{proof}
	The formal dimension of Seiberg-Witten moduli space is the sum of the indices of the following two operators
	\begin{align*}
		D^+_A:L^2_k(\check X;S^+)^{\uptau_X} &\to L^2_{k-1}(\check X;S^-)^{\uptau_X}\\
		d^* \oplus d^+: L^2_k(\check X;iT^*\check X)^{-\upiota_X^*} &\to
		L^2_{k-1}(\check X;i\reals \oplus i\Lambda^+)^{-\upiota_X^*}.
	\end{align*}
	First, we drop the real involution and derive the index of the bifold Seiberg-Witten operator.
	The bifold index of $d^* \oplus d^+$ is simply $\dim H^1(\check X;\reals) - \dim H^0(\check X;\reals) - \dim H^+(\check X;\reals)$.
	It follows that the index of $d^* \oplus d^+$ is $(b^1 - b^0 - b^+)$ of the underlying $4$-manifold $X$. 
	
	To compute the index of the Dirac operator,
	assume for now there is no seam singularity.
	By the excision principle, it is enough to check the formula for quotient bifolds.
	By the $\bbZ_2$-equivariant Atiyah-Segal-Singer Lefshetz formula~\cite[Theorem~6.16]{BGV1992HeatkernelDirac}, the invariant index of a global quotient bifold has the following shape:
	\[
		\frac{c_1(S^+)^2-\sigma(X)}{4}+m_0+ m_1\langle c_1(S^+),\tilde \Sigma^{\sfc}\rangle + m_2[\tilde \Sigma^{\sfc}] \cdot [\tilde \Sigma^{\sfc}]
		+m_3 \chi(\tilde \Sigma^{\sfc}).
	\]
	The coefficients $m_i$ can be computed from examples.
	
	Let $\Sigma_g$ be the standard genus-$g$ surface in $S^4$.
	The invariant index of the $\spinc$ Dirac operator over the the double cover $\sharp_g (S^2 \times S^2)$ of $(S^4,\Sigma_g)$ is zero because the existence of invariant positive scalar curvature metric.
	So $m_0 = m_3 = 0$.
	Let $X = \mathbb{CP}^2$ and $\tilde \Sigma$ be a degree-$2$ rational curve.
	For this holomorphic pair, there is a canonical \spinc structure, such that $c_1(S^+) = c_1(X) - [\tilde \Sigma]/2 = 2H$ where $H$ is the hyperplane class.
	The bifold index of Dirac operator for holomorphic pairs with canonical \spinc structures is equal to $(b^0+b^+-b^1)$, as observed by LeBrun~\cite[Proposition~2.4]{LeBrun2015_edges}.
	As a result, $2 = 3/4 + 4m_1 + 4m_2$.
	When $\Sigma$ is instead a degree-$4$ curve in the plane, the same argument gives $2 = 0 + 4m_1 + 16m_2$.
	This is enough to determine all $m_i$:
	\[
		\ind(D_A) = \frac{c_1(S^+)^2-\sigma(X)}{4}+\frac14\langle c_1(S^+),\tilde \Sigma^{\sfc}\rangle + \frac{1}{16}[\tilde \Sigma^{\sfc}] \cdot [\tilde \Sigma^{\sfc}].
	\]
	
	To deal with seams, we consider the local model around a seam which is a foam in $S^1 \times B^3$.
	Its double is a foam in $S^1 \times S^3$, possibly with an order-two or three monodromy along the $S^1$-factor. 
	In a finite cover, the foam is the standard product foam $(S^1 \times S^3, S^1 \times \text{Theta graph})$.
	We can equip this bifold with a positive scalar curvature metric, from which it follows that $\ind(D_A) = 0$. 
	So the formula in the seam case checks out.
	
	Back to the case with involution, let $(\check X,\upiota_X)$ be the double branched cover of $S^4$ along $\Sigma^{\sfr}$. 
	Then $b^0_{-\upiota^*}(X) = 0$, while $b^1_{-\upiota^*}(X) = b^1(X)$ and $
		b^+_{-\upiota^*}(X) = b^+(X)$.
	Furthermore, $b^1(X)=0$ and by~\cite{HsiangSzczarbar1971}, the formula for $b^+(X)$ involves the genus and self-intersection of $\Sigma^{\sfr}$:
	\[
		b^+(X) = \frac12b_1(\Sigma^{\sfr}) - \frac14 [\Sigma^{\sfr}]\cdot [\Sigma^{\sfr}].
	\]
	Since $\uptau$ acts anti-linearly, the invariant part of the Dirac index is simply half of the complex linear one:
	\[
		\ind(D_A)^{\uptau_X} = \frac{c_1(S^+)^2-\sigma(X)}{8}+\frac18\langle c_1(S^+),\tilde \Sigma^{\sfc}\rangle + \frac{1}{32}[\tilde \Sigma^{\sfc}] \cdot [\tilde \Sigma^{\sfc}].
	\]
	In summary, the index of the real bifold Seiberg-Witten operator is 
	\[
		\frac12b_1(\Sigma^{\sfr}) - \frac14 [\Sigma^{\sfr}]\cdot [\Sigma^{\sfr}] + \frac{c_1(S^+)^2-\sigma(X)}{8}+\frac18\langle c_1(S^+),\tilde \Sigma^{\sfc}\rangle + \frac{1}{32}[\tilde \Sigma^{\sfc}] \cdot [\tilde \Sigma^{\sfc}]. \qedhere
	\]
\end{proof}

Let $P$ be a manifold (possibly with boundary) and $g^p$ be a family of bifold $\upiota_W$-invariant metrics on $\check W$ parametrized by $P$, containing the same isometric copy of cylindrical regions near $Y_{\pm}$.
Suppose $\omega_p$ is a family of perturbing $2$-forms.
We have the family moduli space
\[
	M(\check X,\uptau_X)_P \subset P \times \calB^{\sigma}_k(\check X,\uptau_X).
\]
A family analogue of Lemma~\ref{lem:residual_closed_manifold} shows that the family moduli space is a manifold for a residual set of choices (cf. \cite[Lemma~27.1.4]{KMbook2007}).
The following is the parametrized version of Lemma~\ref{lem:residual_closed_manifold}
\begin{prop}
\label{prop:family_2form_pert}
	Suppose $b^+_{-\upiota^*}(\check W) > \dim P$.
	There exists a family of compactly supported $2$-forms $\omega_p$ such that there exists a family of  2-forms $\omega^p$ such that
	\[
		4\int_{W^*}\omega_p \wedge \kappa_p \ne
		(-2\pi i)\langle c_1(\fraks_X), [\kappa_p] \rangle
	\]
	for all $p \in P$.
	If $\omega_p$ is previously chosen for $q \in \del P$, then $\omega_p$ can be chosen to agree with the boundary.
	Moreover, there exists $T \ge T_0$ such that for all $T \ge T_0$, there are no reducible solutions in the parametrized moduli space $M_z([\fraka], \check W(T)^*,\uptau_W, [\frakb])_P$. \qed 
\end{prop}
A corollary to the above proposition is that the following monopole invariant is well-defined, i.e. independent of metrics and perturbations.
\begin{defn}
	\label{defn:monopole_invariant}
	Let $(\check X,\upiota_X)$ be a closed, connected, oriented real $4$-bifold with $b^+_{-\upiota}(\check X) > 1$.
	The \emph{monopole invariant} $\frakm(-|\check X, \uptau_X)$ of $(\check X, \upiota_X)$ for the real \spinc structure $(\fraks_X,\uptau_X)$ is the map
	\[
	\frakm(-|\check X, \uptau_X):
	H^*(\calB^{\sigma}(\check X, \uptau_X)) \to \bbZ_2,
	\]
	defined by evaluation on Seiberg-Witten moduli spaces 
	\[\frakm(u|\check X, \uptau_X) = \langle u, [M(\check X, \uptau_X)]\rangle.\]
\end{defn}

\subsection{Monopole invariants as mixed invariants}
The monopole invariant can be reformulated as a map $\overrightarrow{\HMR}_{\bullet}$ involving two flavours of the monopole Floer homology:
\[
\overrightarrow{\HMR}_{\bullet}(u|\check W, \upiota_W):	\widehat{\HMR}_{\bullet}(S^3 \sqcup S^3, \upiota) 
\to
\widecheck{\HMR}_{\bullet}(S^3\sqcup S^3, \upiota),
\]
where $\upiota$ is the swapping involution.
This invariant counts irreducible Seiberg-Witten solutions on the cobordism and generalizes to cobordisms that are not obtained from puncturing closed real $4$-bifolds.

We start by analyzing reducible solutions on cobordisms.
Suppose $(\check W,\upiota_W):(\check Y_-,\upiota_-) \to (\check Y_+, \upiota_+)$ is a cobordism.
Let $(\check W^*,\upiota_W)$ be the corresponding bifold with cylindrical ends.
Given metrics on $W$ and $Y_{\pm}$,
let $I^2(\check W^*) \subset H^2(\check W^*;\reals)$ be the image of compactly supported cohomology.
There is a non-degenerate pairing $\langle -, - \rangle$ on $I^2(\check W^*)$ induced by integration.
Denote the maximal positive-definite subspace by $I^+(\check W^*)$.  
Let $\calH(W^*,g)$ be the closed square-integrable $2$-forms and $\calH^+(W^*,g)$ the subspace consisting of self-dual square integrable $2$-forms.
The following lemma (\cite[Lemma~27.2.1]{KMbook2007}) summarizes the results for the reducible solutions to the curvature equation $F^+_A = \omega^+$.
\begin{lem}
	\label{lem:red_on_cylindend}
	Let $L$ be a complex bifold line bundle on $\check W^*$ whose Chern class restricts to torsion classes on cylindrical ends.
	Let $A$ be a connection on $L$ with square integrable curvature. 
	Let $\kappa$ be a square-integrable closed, self-dual $2$-form on $\check W^*$.
	Then
	\[
	\int_{\check W^*} F_A^+ \wedge \kappa
	= 
	-2\pi \langle c_1(L),\kappa\rangle.
	\]
	If $b^+(\check W) > 0$, then there exists a compactly supported imaginary-valued $2$-form $\omega$ on $\check W^*$ such that there exists no connection $A$ with $L^2$ curvature in any line bundle $L \to \check W^*$ satisfying $F_A^+ = \omega^+$. \hfill \qedsymbol
\end{lem}

Choose admissible perturbations $\frakq_{\pm}$ on $(\check Y_{\pm},\uptau_{\pm})$.
Let $\hat{\frakp}$ be the corresponding perturbation supported on collar as in Section~\ref{subsec:mod_space_over_bifold_boundary}. 
We further perturb the Seiberg-Witten operator by an $(-\upiota^*)$-invariant form $\omega$ supported in the interior of $\check W$.
Denote the new perturbed operator by $\frakF^{\sigma}_{\hat{\frakq},\omega}$.
Furthermore,
suppose there are cylindrical region $I \times \check Y_{\pm}$ disjoint from the support of both $\omega$ and $\hat{\frakp}$.
Let $g_T$ be a real bifold metric on $W$ obtained by increasing the cylinder by $T > 0$.
\begin{prop}
	\label{prop:no_red_cylindend_closed}
	Suppose $b^+_{-\upiota^*}(\check W) > 0$ and $\omega \in \calH^+_{-\upiota^*}(\check W^*,g)$ satisfies
	\[
	4\int_{\check W^*} \omega \wedge \kappa \ne
	(-2\pi i)\langle c_1(\fraks_X),[\kappa]\rangle.
	\]
	for at least one closed, self-dual, square-integrable 2-form $\kappa$ on $\check W^*$.
	Then there exists $T_0 > 0$ such that for all $T \ge T_0$, there are no reducible solutions in the moduli space $M_z([\fraka], \check W(T)^*,\uptau_W, [\frakb])$ defined by $\frakF^{\sigma}_{\hat{\frakp},\omega} = 0$.
\end{prop}
\begin{proof}
	The idea is that as if there were reducibles as $T \to \infty$, then the connection part converges to one on $\check W^*$ satisfying $F^+_A = 4\omega^+$.
	By
	the energy arguments in the compactness proof in \cite[Proposition~26.1.4]{KMbook2007}, the limiting connection has $L^2$ curvature which contradicts Lemma~\ref{lem:red_on_cylindend}.
	See \cite[Proposition~27.2.4]{KMbook2007}.
\end{proof}
Let $\calU$ be an open cover of $ \boldsymbol{\mathcal B}^{\sigma}(\check W,\upiota_W)$  transverse to all moduli spaces of dimension $d \le d_0$. 
\begin{defn} 
	We let
	\[
		\overrightarrow{m}: C^d(\calU;\bbF_2) \otimes \hat C_{\bullet}(\check Y_-,\upiota_-) \to \check C_{\bullet}(\check Y_+,\upiota_+)
	\]
	for $d\le d_0$, be the map
	\[
		\overrightarrow{m} = 
	\begin{bmatrix}
			m^o_o & m^u_o\\
			m^o_s & m^u_s
	\end{bmatrix}.
	\]
	Each entry $m^*_*$ is defined by counting solutions on the cobordism $(\check  W,\upiota_W)$ as in Definition~\ref{defn:entry_m}.
\end{defn}
\begin{prop}
	For any $\xi$,
	\[
		(-1)^d\check\del(\check Y_+)\overrightarrow{m}(u \otimes \check\xi) = -\overrightarrow{m}(\delta \otimes \check\xi) +\overrightarrow{m}(u \otimes \check\del(\check Y_-)\check\xi).
	\]
	Therefore $\overrightarrow{m}$ descends to a homology-level map
	\[
		\overrightarrow{m}: H^d(\calU;\bbF_2) \otimes 
		\widehat{\HMR}_{\bullet}(\check Y_-,\upiota_-) \to 
		\widecheck{\HMR}_{\bullet}(\check Y_+,\upiota_+).
	\]
\end{prop}
\begin{proof}
	The identities to be proved are the same as the ones that appear in the proof of Proposition~\ref{prop:identities_m(uW)}, where the full list can be found in
	\cite[Lemma~25.3.6]{KMbook2007}, except that all terms $\bar m^s_s, \bar m^s_u, \bar m^u_s, \bar m^u_u$ are zero.
\end{proof}
\begin{defn}
	Assume $b^+_{-\upiota}(\check W) \ge 2$.
	Let $u$ be an element of $H^d(\boldsymbol{\calB}^{\sigma}(\check W, \upiota_W);\bbF_2)$.
	Suppose there is a bifold metric $g$ and a  perturbation $2$-form that satisfy the hypothesis of Proposition~\ref{prop:no_red_cylindend_closed}.
	Let $T_0$ be as in the conclusion of Proposition~\ref{prop:no_red_cylindend_closed}. 
	We define
	\[
		\overrightarrow{\HMR_{\bullet}}(u|\check W,\upiota_W)_{g,\omega}:
		\widehat{\HMR}_{\bullet}(\check Y_-,\upiota_-) \to 
		\widecheck{\HMR}_{\bullet}(\check Y_+,\upiota_+).
	\]
	as the operator
	\[
		\overrightarrow{m}(u\otimes -),
	\]
	where $\overrightarrow{m}$ is computed using a metric $g(T)$ for $T \ge T_0$.
\end{defn}
The following proposition is based on \cite[Proposition~27.3.4]{KMbook2007} and can be proved the same way.
\begin{prop}
	If $b^+_{-\upiota^*}(W) \ge 2$, then the map $\overrightarrow{\HMR_{\bullet}}(u|\check W,\upiota_W)_{g,\omega}$ does not depend on $g$ and $\omega$.
\end{prop}
\begin{proof}
Given any pair $(g_0,\omega_0)$ and $(g_1,\omega_1)$, there is a $P = [0,1]$ family of metrics and perturbations for which there is no reducibles, by Proposition~\ref{prop:family_2form_pert}.
Define the chain homotopy operator $\overrightarrow{K}:C^d(\calU;\bbF_2) \otimes \hat C_{\bullet}(\check Y_-,\upiota_-) \to \check C_{\bullet}(\check Y_+,\upiota_+)$ using parametrized moduli space $M([\fraka],\check W,\upiota_W,[\frakb])$ as in Proposition~\ref{prop:cob_map_indep_metric_pert}:
	\[
		\overrightarrow{K} = \begin{bmatrix}
			m^o_o(P) & m^u_o(P)\\
			m^o_s(P) & m^u_s(P)
		\end{bmatrix}
	\]
Indeed, $\overrightarrow{K}$ satisfies
\[
	\check\del \overrightarrow{K}(u \otimes \hat \xi) =
	-\overrightarrow{K}(\delta u \otimes \hat \xi) + 
	+\overrightarrow{K}(u \otimes \hat \del \hat \xi) 
	+ \overrightarrow{m}(u \otimes \hat\xi) -
	\overrightarrow{m}(u \otimes \hat\xi). 
	\qedhere
\]
\end{proof}
Finally, the closed real $4$-bifold invariant can be recovered from $\overrightarrow{\HMR_{\bullet}}(u|\check W,\upiota_W)$, via
the following counterpart of \cite[Proposition~27.4.1]{KMbook2007}.
Note however we have not developed Floer cohomology nor duality in this article.
The ordinary version is explained in \cite[Section~22.5]{KMbook2007}.
\begin{prop}
	\label{prop:closedinv_arrowmap}
	Let $(\check X,\upiota_X)$ be a closed, oriented real $4$-bifold with $b^+_{-\upiota_X}(\check X, \upiota-_X > 1$.
	Let $(\check W,\upiota_W)$ be the cobordism of $(S^3 \sqcup S^3,\upiota)$ to $(S^3 \sqcup S^3,\upiota)$ obtained by two pairs of balls from $X$.
	Let $1$ be the standard generator of $\widehat{\HMR}_{\bullet}(S^3 \sqcup S^3, \upiota)$ and $\check 1 \in \widecheck{\HMR^{\bullet}}(S^3 \sqcup S^3, \upiota) \cong \bbF_2[[U]]$ the generator for the Floer homology group.
	 Then
	 \[
	 	\frakm(u|X) = \langle \overrightarrow{\HMR_{\bullet}}(u|\check W,\upiota_W)(1),\check 1 \rangle,
	 \]
	where the angle brackets denote the $\bbF_2$-valued pairing between $\widehat{\HMR}_{\bullet}(S^3 \sqcup S^3, \upiota)$  and $\check 1 \in \widecheck{\HMR}_{\bullet}(S^3 \sqcup S^3, \upiota)$.
	\hfill \qedsymbol
\end{prop}

\subsection{Foam cobordisms, evaluations, and module structures}
\label{subsec:foamcob_eval_module}
Given a cobordism $(\Sigma,s_{\Sigma})$ between webs $(K_{\pm},s_{\pm})$ with $1$-sets
\[
	(\Sigma,s_{\Sigma}): (K_-,s_-) \to (K_+,s_+),
\]
the real cover $(\check W,\upiota_W)$ along $\Sigma^{\sfr}$ is a cobordism from the real covers of $(K_-,s_-)$ to $(K_+,s_+)$. 
Let $u$ be a cohomology class in $H^*(\boldsymbol{\mathcal B}^{\sigma}(\check W,\upiota_{W}))$. 
The \emph{cobordism map} $\HMR^{\circ}(u|\Sigma,s)$ is defined as the corresponding cobordism maps on the covers:
\[
	 \HMR^{\circ}(u|\Sigma,s) = 
	 \HMR^{\circ}(u|\check W, \upiota_W):
	\HMR^{\circ}_{\bullet}(K_-,s_-) \to \HMR^{\circ}_{\bullet}(K_-,s_-).
\]
If $u$ is not specified, then
$\HMR^{\circ}(\Sigma,s_{\Sigma})$ is defined to be $\HMR^{\circ}(1|\Sigma,s_{\Sigma})$.
Recall that a \emph{dot} is a point $\delta \in \Sigma$ lying on an open facet.
Given a $1$-set $s$ on $\Sigma$, either $\delta$ lies on a $\sfr$-face or a $\sfc$-face.
In the first case, we have the  evaluation map at $\delta$ 
\begin{equation*}
	ev_{\delta}: \mathcal G(\check W,\upiota_W) \to \{\pm 1\}, 
	\quad
	g \mapsto g(\delta),
\end{equation*}
which induces a cohomology class $\upsilon_{\delta} \in H^1(\boldsymbol{\mathcal B}^{\sigma}(W,\upiota_{W}); \bbF_2)$. Here, $\delta$ is viewed as a point in the preimage of $\delta$ under $\check W \to ([0,1] \times S^3, \Sigma^{\sfr})$.
The class $\upsilon_{\delta}$ is the first Stifel-Whitney class of a real line bundle $L_{\delta} \to \calB^{\sigma}(W,\upiota_W)$.
In the case when $\delta$ lies on a $\sfc$-face, we simply declare $\upsilon_{\delta} = 0$.
Furthermore, point $p$ on $\check W \setminus \tilde{\Sigma}$ defines an evaluation map
\[
ev_{p}:\calG(\check W,\upiota_W) \to \U(1),
\]
inducing a cohomology class $U \in H^2(\boldsymbol{\calB}^{\sigma}(\check W,\upiota_W);\bbF_2)$.
The class is independent of the point $p$, and well-defined even when $\Sigma$ is empty.
We extend the definition of $\HMR^{\circ}(\Sigma)$ to foams with dots as follows.
\begin{defn}
	\label{defn:umap}
	Let $\Sigma: K_- \to K_+$ be a foam cobordism and $\delta_1,\dots,\delta_n$ be dots.
	Suppose $s_{\Sigma}$ is a $1$-set.
	The \emph{dotted cobordism map} is defined as
	\[
		\HMR^{\circ}(\Sigma,s|\delta_1,\dots,\delta_n)= 
		\HMR^{\circ}(\upsilon_{\delta_1}, \dots, \upsilon_{\delta_n}|\Sigma,s)=
		\HMR^{\circ}(\upsilon_{\delta_1}, \dots, \upsilon_{\delta_n}|\check W,\upiota_{W}).
	\]
	Analogously, supposing $b^+_{-\upiota}(W) > 1$, we define $\overrightarrow{\HMR}(\Sigma,s_{\Sigma})$ as the corresponding $\overrightarrow{\HMR}$ map on the real double cover:
	\[
		\overrightarrow{\HMR}(\Sigma,s_{\Sigma}\big|\delta_1,\dots,\delta_n) = 	\overrightarrow{\HMR}(\upsilon_{\delta_1},\dots,\upsilon_{\delta_n}\big|\check W,\upiota_{W}):
		\widehat{\HMR}_{\bullet}(K_-,s_-) \to 
		\widecheck{\HMR}_{\bullet}(K_+,s_+).
	\]
\end{defn}
\begin{defn}
	\label{defn:foam_monopole_invariant}
	Let $\Sigma$ be a closed foam dotted with $(\delta_1,\dots,\delta_n)$.
	Assume $s_{\Sigma}$ is a $1$-set of ${\Sigma}$  such that 	
	\[
	\frac12b_1(\Sigma^{\sfr}) - \frac14 [\Sigma^{\sfr}]\cdot [\Sigma^{\sfr}]> 1.
	\]
	Let $(\check X,\upiota_X)$ be its real double cover and $(\fraks_{\Sigma},\uptau_{\Sigma})$ be a real bifold \spinc structure on $(X,\upiota_X)$.
	The \emph{monopole invariant}  is
	\[\frakm(\Sigma, \uptau_{\Sigma}|\delta_1,\dots,\delta_n)
	=
	\frakm(\upsilon_{\delta_1},\dots,\upsilon_{\delta_n}|\check X,\uptau_{\Sigma}).\]
\end{defn}
\begin{defn}
	Let $(\Sigma,s_{\Sigma})$ be a closed dotted foam with $1$-set
	such that 	$
	\frac12b_1(\Sigma^{\sfr}) - \frac14 [\Sigma^{\sfr}]\cdot [\Sigma^{\sfr}]> 1$. 
	The \emph{foam evaluation} of $\Sigma$ \emph{with respect to $s_{\Sigma}$} is the map
\[
\overrightarrow{\HMR}(\Sigma,s_{\Sigma}):
\widehat{\HMR}_{\bullet}(\emptyset) \to 
\widecheck{\HMR}_{\bullet}(\emptyset).
\]
\end{defn}
The foam evaluation of $(\Sigma,s_{\Sigma})$  recovers the monopole invariant via Proposition~\ref{prop:closedinv_arrowmap}.
Notice that the real double cover of the empty web $\emptyset \subset S^3$ is precisely the disjoint union of two $S^3$.
Operators $\HMR^{\circ}_{\bullet}(K)$ can be defined by applying dotted cobordism maps to the cylinder $[0,1] \times K$.
\begin{defn}
Let $K$ be a web $K$ and $\delta$ be a dot on $K \setminus \Ver(K)$. 
Let $s$ be the pullback $1$-set to the cylinder foam.
The operator $
	\upsilon_{\delta}:\HMR^{\circ}_{\bullet}(K) \to \HMR^{\circ}_{\bullet}(K)$	is defined as, for each $1$-set $s$, the dotted cobordism map
	\[
	\HMR^{\circ}(\upsilon_{\delta}|[0,1] \times K,s)=\HMR^{\circ}([0,1] \times K,s|\delta),
	\]
	where $\delta$ is understood as $\{\epsilon\} \times \delta$ for some small $\epsilon > 0$.
	We define the degree-$2$ operator $U: \HMR^{\circ}_{\bullet}(K) \to \HMR^{\circ}_{\bullet}(K)$ as
\[
 \HMR^{\circ}(U|[0,1] \times K, s).
\]
\end{defn}
By continuity, the class $\upsilon_{\delta}$ depends only on the component of $K^{\sfr}$.
\begin{proof}[Proof of Proposition~\ref{prop:powers_upsilon}]
For any pair of dots $\delta,\delta'$ both on $\sfr$-edges, the difference $\upsilon_{\delta} - \upsilon_{\delta'}$ is a mod-2 cohomology class of the Picard torus $\bbT$ of $\check W$. 
Thus their squares are independent of $\delta$, and 
\[
\upsilon_{\delta}^2 = \upsilon_{\delta'}^2 = U.
\]
The rest of the relations stated in the introduce are also straightforward.
Suppose $\delta_1,\delta_2,\delta_3$ are $3$ dots on $K$ that lie on three distinct edges that meet at a common vertex.
Then exactly two $\delta_i$'s lie on the  same component of $K^r$, and the third is zero.
So modulo two
\[
\upsilon_{\delta_1} + \upsilon_{\delta_2} + \upsilon_{\delta_3} = 0.
\]
The identity
\begin{equation}
\label{eqn:sum_of_double}
\upsilon_{\delta_1}\upsilon_{\delta_2} + \upsilon_{\delta_2}\upsilon_{\delta_3} + \upsilon_{\delta_3}\upsilon_{\delta_1} = U
\end{equation}
holds because exactly one of the summand is nonzero and a square. Finally, the triple product contains a zero factor, so
\[
\upsilon_{\delta_1}\upsilon_{\delta_2}\upsilon_{\delta_3}=0. \qedhere
\]
\end{proof}
\begin{rem}
	Equation~\eqref{eqn:sum_of_double} is formally similar to \cite[Prop.~5.7]{KMdefweb2019}. 
	In  \cite[Eq.~(15)]{KMweb} and  \cite[Eq.~(2)]{Khovanov2004sl3} are also counterparts, except the right hand sides are zero.
\end{rem}

\section{The framed monopole web homology}
\label{sec:framed} 
\subsection{Based webs and foams}
A \emph{based} web $(K,p)$ is a web $K$ with a base point $p \in K$. 
A \emph{based $1$-set} $s$ of a based web $(K,p)$ is a $1$-set of $K$ such that $p$ lies on an edge not belonging to $s$.
A \emph{based} foam cobordism from $(K_-,p)$ to $(K_+,p)$  is a foam cobordism $\Sigma \subset [0,1] \times S^3$ containing $\{p\} \times [0,1]$.
A \emph{based $1$-set} $s_{\Sigma}$ is a $1$-set of $\Sigma$ such that $\{p\} \times [0,1]$ lies in facet not belonging to $s_{\Sigma}$.

\subsection{The chain complex as a mapping cone}
The functor $\widetilde{\HMR}_{\bullet}(K,p)$ is defined on the category of based webs. 
The following construction is a counterpart of  \cite{Bloom2011}, which in turn is an analogue of $\widehat{\textit{HF}}$ in Heegaard Floer theory \cite{OzSz2004}.
See also~\cite{ljk2024SSKh} for the case of empty $1$-sets.
\begin{defn}
Let $(K,p)$ be a based web and $s$ be a based $1$-set.
Then the \emph{framed monopole web homology} $\widetilde{\HMR}_{\bullet}(K,p,s)$ is the homology of the following chain complex
\[
\widetilde C(K,p,s) =
\check C(K,s) \oplus \check C(K,s),
\quad
\widetilde\del
= \begin{bmatrix}
	\check\del & 0\\
	\check m(\upsilon_p|[0,1] \times K) & \check\del
\end{bmatrix},
\]
which is the mapping cone of $\check m(\upsilon_p|[0,1] \times K):\check C(K,s) \to \check C(K,s)$.
It fits into the following long exact sequence:
\[\begin{tikzcd}
	\cdots \ar[r] & \widetilde{\HMR}_{\bullet}(K,s) \ar[r] & \widecheck{\HMR}_{\bullet}(K,s) \ar[r,"\upsilon_p"] & \widecheck{\HMR}_{\bullet}(K,s) \ar[r] & \cdots.
\end{tikzcd}\]
We write $\widetilde{\HMR}_{\bullet}(K,p)$ as the direct sum of $\widetilde{\HMR}_{\bullet}(K,p,s)$  over all based $1$-sets $s$ of $(K,p)$.
\end{defn}

To describe the cobordism maps of  $\widetilde{\HMR}_{\bullet}$,
let us reformulate the definition of $\upsilon_p$ using cobordisms maps with additional unknot ends.
Let $\Sigma:(K_-,p_-) \to (K_+,p_+)$ be a based foam cobordism, containing $[0,1] \times p_-$.
Choose a based $1$-set $s_{\Sigma}$ of $\Sigma$ and let $s_{\pm}$ be the restriction of $1$-sets of $K_{\pm}$.
Here $p_-,p_+$ are identified with $0 \times p_-$ and  $1 \times p_-$, respectively.
Let $p_o = \frac14 \times p_-$ and let $\Sigma^{o}$ be the result of removing a small $2$-disc in $\Sigma$ centred at $p_o$.
The foam $\Sigma^{o}$ is a subset of $[0,1] \times S^3 \setminus B^4_o$, where $B^4_o$ is a small $4$-ball centred at $p_o$, intersecting $\Sigma$ at a small $2$-disc.
 It is equipped with an induced $1$-set $s_{\Sigma^o}$ and an additional boundary $(S^3, \U)$ viewed as an unknot contained in a small $S^3$.
 
In what follows, we will work over the real double covers.
Let $(\check W,\upiota_W)$, $(\check W^o,\upiota_W)$, $(\check Y_-,\upiota_-)$, $(\check Y_+,\upiota_+)$ be the real double branched covers of $(\Sigma,s_{\Sigma})$, $(\Sigma^o,s_{\Sigma^o})$,$(K_-,s_-)$, $(K_+,s_+)$, respectively.
Denote by $(\check S^3_o,\upiota_o)$ the boundary of $(\check B^4_o,\upiota_o)$, where the latter is the real cover of $(B^4_o,B^4_o \cap \Sigma ) \cong (B^4,B^2)$.
Let $(\check W^{o*},\upiota_W)$ be the real bifold adjoined with three cylindrical ends:
\[
\big((-\infty,0) \times \check Y_-,  \upiota_-
\big), 
\ 
\big((1,+\infty) \times \check Y_+ ,\upiota_+
\big), \
\big((0,\infty) \times S^3_o,  \upiota_o
\big).\] 
This is a slightly more general type of cobordism than what we have previously considered, because of the additional $(S^3_o,\upiota_o)$ end.
The ordinary monopole analogue of this setup is explained in~\cite[Chapter~3, Section~3]{LinFran2018}, initially due to~\cite{Bloom2011}.

Over the new cobordism $(\check W^{o*},\upiota_W)$, choose bifold metrics, real \spinc structures, and perturbations so that the chain complex of $(S^3_o,\upiota_o)$ is generated by reducible critical points $\{[\frakc_i]: i \in \bbZ\}$.
(This can be achieved as in~\cite[Section~14.1]{ljk2022}.)
Here,  $[\frakc_i]$ is boundary-stable whenever $i \ge 0$, and $\langle \upsilon_{p}, M(B^4_o, [\frakc_i]) \rangle$ is nonzero if and only if $i = -1$.
On the chain level, the dotted cobordism map $\check m(\upsilon_{p}|\check W, \upiota_W)$ is chain homotopic to the following map:
\[
	\begin{bmatrix}
		m^{uo}_o([\frakc_{-1}]\otimes \cdot) &
		m^{uu}_o([\frakc_{-1}] \otimes \bar\del^s_u(\cdot)) +
		\del^u_o\bar m_u^{us}([\frakc_{-1}] \otimes \cdot)\\
		m^{uo}_s([\frakc_{-1}] \otimes \cdot) &
		\bar m^{uu}_s([\frakc_{-1}] \otimes \bar\del^{s}_u(\cdot)) +\del^u_s\bar m_s^{us}([\frakc_{-1}] \otimes \cdot)
	\end{bmatrix}.
\]
The chain homotopy can be defined analogously as \cite[Proposition~13.11]{ljk2022}, using the fact that there are even number of isolated trajectories between $\frakc_i$ and $\frakc_{i-1}$.

The definition of $\widetilde{\HMR}(\Sigma)$ involves a one-parameter family of metrics and  a chain homotopy $\check H$ using the associated family moduli space.
Let $\check W_{--,++}$ be the compact punctured cobordism in the previous step.
Let $\check Y_{--}$ and $\check Y_{++}$ the boundary bifolds previously known as $\check Y_-$ and $\check Y_+$, respectively.
Every bifold involved in the discussion is equipped with a real involution, induced from the initial $\check W$.
All operations will be $\upiota_W$-invariant.

Let $\check Y_{-+}$ be a hypersurface of $\check W_{--,++}$ that is a ``pushed in'' copy of $\check Y_{--}$; that is, we take an inner copy of $\check Y_{--}$, choose an invariant path $\gamma$ to the punctured ball $B^4_o$ and a neighbourhood of $\gamma$ that contains $B^4_o$.
Let $\check Y_{-+}$ be the connected sum of inner collar and the boundary of the neighbourhood.
Let $\check W_{--,-+}$ be the subbifold bounded by $\check Y_{--}$ and $\check Y_{-+}$. In particular, the punctured balls are contained in  $\check W_{--,-+}$.
Similarly, we define the pushed in copy of $\check Y_{-+}$ of $\check Y_{++}$, and denote by $\check W_{+-,++}$ the cobordism between them.
Furthermore, let $\check W_{--,+-}$ be the region bounded by $\check Y_{--}$ and $\check Y_{+-}$, and let $\check W_{-+,++}$ bounded by $\check Y_{-+,++}$.
The cobordism with hypersurfaces is illustrated in Figure~\ref{fig:Yplusminus}.
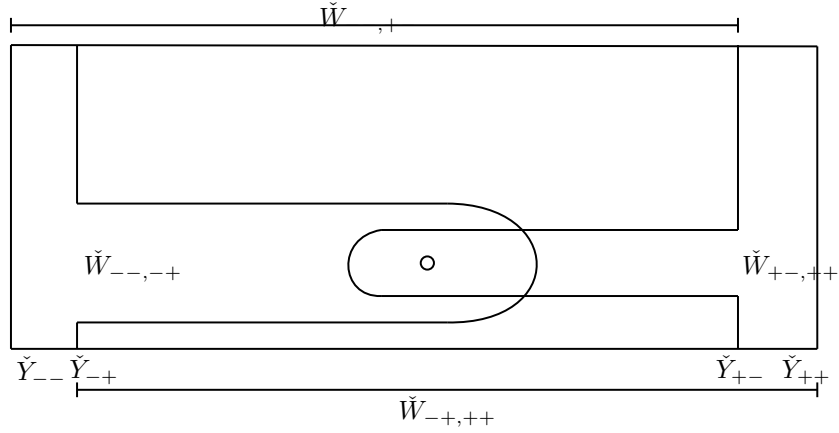
\begin{figure}[ht]
	\centering
	\tikzset{every picture/.style={line width=0.75pt}} 
	\begin{tikzpicture}[x=0.5pt,y=0.5pt,yscale=-1,xscale=1]
		
		\draw    (20,37) -- (630,38) ;
		\draw    (20,37) -- (20,267) ;
		\draw    (20,267) -- (630,267) ;
		\draw    (630,38) -- (630,267) ;
		\draw    (70,37) -- (70,157) ;
		\draw   (330,202) .. controls (330,199.24) and (332.24,197) .. (335,197) .. controls (337.76,197) and (340,199.24) .. (340,202) .. controls (340,204.76) and (337.76,207) .. (335,207) .. controls (332.24,207) and (330,204.76) .. (330,202) -- cycle ;
		\draw    (70,157) -- (350,157) ;
		\draw    (70,247) -- (350,247) ;
		\draw    (70,247) -- (70,267) ;
		\draw    (570,37) -- (570,177) ;
		\draw    (300,177) -- (570,177) ;
		\draw    (300,227) -- (570,227) ;
		\draw    (570,227) -- (570,267) ;
		\draw    (350,157) .. controls (439.33,157.33) and (441.33,248.67) .. (350,247) ;
		\draw    (300,177) .. controls (265.33,182.67) and (268.67,228) .. (300,227) ;
		\draw    (20,22) -- (570,22) ;
		\draw [shift={(570,22)}, rotate = 180] [color={rgb, 255:red, 0; green, 0; blue, 0 }  ][line width=0.75]    (0,5.59) -- (0,-5.59)   ;
		\draw [shift={(20,22)}, rotate = 180] [color={rgb, 255:red, 0; green, 0; blue, 0 }  ][line width=0.75]    (0,5.59) -- (0,-5.59)   ;
		\draw    (70,298) -- (630,298) ;
		\draw [shift={(630,298)}, rotate = 180] [color={rgb, 255:red, 0; green, 0; blue, 0 }  ][line width=0.75]    (0,5.59) -- (0,-5.59)   ;
		\draw [shift={(70,298)}, rotate = 180] [color={rgb, 255:red, 0; green, 0; blue, 0 }  ][line width=0.75]    (0,5.59) -- (0,-5.59)   ;
		
		\draw (22,270.4) node [anchor=north west][inner sep=0.75pt]    {$\check Y_{--}$};
		\draw (62,269.4) node [anchor=north west][inner sep=0.75pt]    {$\check Y_{-+}$};
		\draw (551,269.4) node [anchor=north west][inner sep=0.75pt]    {$\check Y_{+-}$};
		\draw (601,269.4) node [anchor=north west][inner sep=0.75pt]    {$\check Y_{++}$};
		\draw (251,2.4) node [anchor=north west][inner sep=0.75pt]    {$\check W_{--,+-}$};
		\draw (311,300.4) node [anchor=north west][inner sep=0.75pt]    {$\check W_{-+,++}$};
		\draw (73,190.4) node [anchor=north west][inner sep=0.75pt]    {$\check W_{--,-+}$};
		\draw (571,190.4) node [anchor=north west][inner sep=0.75pt]    {$\check W_{+-,++}$};
	\end{tikzpicture}
	\caption{Hypersurfaces in a (punctured) cobordism.}
	\label{fig:Yplusminus}
\end{figure}
Now, $\check Y_{-+}$ and $\check Y_{+-}$ intersect at an $S^2$, whose neighbourhood is modelled on
$S^2  \times (-\epsilon,+\epsilon) \times (-\epsilon,+\epsilon)$.
Let $P = [-\infty, +\infty]$ be a $1$-parameter family of metrics, that are stretched along $\check Y_{-+}$ on $[-\infty,0]$ and along $\check Y_{+-}$ on $[0,\infty]$, respectively.
Given critical points $[\fraka]$ and $[\frakb]$, define a  family moduli space
\[
	M([\fraka], \check W^*, \upiota_W,[\frakb])_P =
	\bigcup_{p \in P}\bigcup_z 
	M_z([\fraka], [\frakc_{-1}], \check W^{**}, \upiota_W,[\frakb]).
\]
Counting zero-dimensional moduli spaces of the form $M([\fraka], \check W^*, \upiota_{W},[\frakb])$ and inserting different types of critical points, we obtain eight operators $H^{u*}_*([\frakc_{-1}] \otimes \cdot)$.
These operators in turn can be used to define entries of a homotopy 
\[
\check H(\upsilon_p|\check W_{--,++},\iota_{W}): \check C_{\bullet}(\check Y_-,\upiota_-) \to \check C_{\bullet}(\check Y_+,\upiota_+),
\]
via the formula (cf. \cite[pg. 44]{Bloom2011}):
\begin{align*}
\check H(\upsilon_p|\check W_{--,++},\iota_{W})&=
	\begin{bmatrix}
		H^{uo}_o(\frakc_{-2}\otimes \ \cdot) &
		H^{uu}_o(\frakc_{-2} \otimes \delbar^s_u(\cdot)) +\del^u_o \bar H^{us}_u(\frakc_{-2} \otimes \ \cdot)
		\\
		H^{uo}_s(\frakc_{-2}\otimes \ \cdot) &
		H^{us}_s(\frakc_{-2}\otimes \ \cdot)+H^{uu}_s(\frakc_{-2} \otimes \ \delbar^s_u(\cdot)) +\del^u_s \bar H^{us}_u(\frakc_{-2} \otimes \ \cdot)
	\end{bmatrix}\\
	&\ + 
	\begin{bmatrix}
		0 &
		m^u_{o,--,+-} \bar m^{us}_{s,+-,++}(\frakc_{-2} \otimes \ \cdot)
		+m^u_{o,--,-+} \bar m^{us}_{s,+-,++0}(\frakc_{-2} \otimes \ \cdot) \\
		0 &
		m^{u}_{o,--,+-} \bar m^{us}_{s,+-,++}(\frakc_{-2} \otimes \ \cdot)
		+m^u_{o,--,-+} \bar m^{us}_{s,+-,++}(\frakc_{-2} \otimes \ \cdot)
	\end{bmatrix}.
\end{align*}
\begin{defn}
Let $(\Sigma,s_{\Sigma}):(K_-,p_-,s_-) \to (K_+,p_+,s_+)$ be a based foam cobordism.
Let $(\check W,\upiota_W)$ be its real double cover as above.
Over the chain level, $\widetilde{\HMR}(\Sigma,s_{\Sigma})= \widetilde{\HMR}(\check W,\upiota_W)$ is defined as
	\[
	\widetilde m(\check W,\upiota_W) = 
	\begin{bmatrix}
		\check m(\check W_{--,+-},\upiota_W) & 0\\
		\check H(\upsilon_p|\check W_{--,++},\upiota_W) &
		\check m(\check W_{-+,++},\upiota_W)
	\end{bmatrix}.
	\]
\end{defn}
\begin{rem}
	The  proofs of well-definedness of re-interpreted $\upsilon_p$ and $\tilde m(\Sigma,s_{\Sigma})$ as chain maps are completely analogous to \cite[Remark~8.3]{Bloom2011}. In particular, they can be extracted from the proof of \cite[Lemma~4.8-9]{ljk2024SSKh}.
\end{rem}
\section{Adjunction and Excision}
\label{sec:adj_excis}
The adjunction inequality in three dimension can be generalized to bifolds.
\begin{prop}
\label{prop:adj_orb}
	Let $\check F$ be an smoothly embedded, oriented genus-$g$ bifold surface in $\check Y$, having $n$ bifold points.
	Suppose
	\[
	\deg(K_{\check F})= -\chi(\check F) = 2g - 2 + \frac{n}{2} > 0.
	\]
	If $\mathfrak s$ is a bifold \spinc structure satisfying
	\[
	\left|\langle c_1(\mathfrak s), [\check F]\right| > 2g - 2 + \frac{n}{2},
	\]
	then exists an orbifold metric on $
	\check Y$ for which the unperturbed Seiberg-Witten equations admit no bifold solutions. 
\end{prop}
\begin{prop}		
\label{prop:adj_orb_real}
	Suppose $(\check Y,\upiota)$ is a real bifold, and that $\check F = \check F_1 \sqcup \check F_2$ is a disjoint union of two identical genus-$g$ bifold surfaces with bifold points. Suppose $\upiota$ exchanges $\check F_{i}$ and
	\[
	\deg(K_{F_1})= -\chi(\check F_1) = 2g - 2 + \frac{n}{2} > 0.
	\]
	If $(\fraks,\uptau)$ is a real bifold \spinc structure satisfying
	\[
	\left|\langle c_1(\mathfrak s), [\check F_1]\right| = \left|\langle c_1(\mathfrak s), [\check F_2]\right| > 2g - 2 + \frac{n}{2},
	\]
	then there exists an $\upiota$-invariant bifold metric on $\check Y$ having no unperturbed Seiberg-Witten solutions.
\end{prop}
\begin{proof}[Proof of Proposition~\ref{prop:adj_orb} and~\ref{prop:adj_orb_real}]
	The proof to the first statement is identical to the standard neck-stretching argument in \cite[Proposition~40.1.1]{KMbook2007}.
	In particular, we used the fact $\check\Sigma$ admits a bifold metric with volume $1$ and constant scalar curvature
	\[
	\scal_{\Sigma}=-8\pi\left(g-1+\frac{n}{4}\right).
	\]
	The second proposition is proved the same way, while keeping the neck-stretching process equivariant.
\end{proof}
Suppose $(\check Y,\upiota)$ is a real bifold and $\check F$ be a \emph{sub-bifold disjoint from $\iota(\check F)$}.
Suppose $F$ is closed, connected, oriented, and satisfies
\[
	\chi(\check F) <  0.
\]
If $\check F$ is connected, let $\calS(\check Y|F)$ be the set of isomorphisms classes of real bifold \spinc structures  on $(\check Y,\upiota)$ satisfying
\[
	\langle c_1(\fraks), [\check F]\rangle = -\chi(\check F) =
	2\genus(F) - 2 + \#(\Sing^{\sfc} \cap F).
\]
If $\check F$ contains multiple components, we set $\calS (\check Y|F)$ to be the intersection of $\calS$ of all components.

\begin{defn}
In view of the adjunction inequality, we introduce the subgroup
\[
	\HMR_{\bullet}(\check Y,\upiota|
	F) = 
	\bigoplus_{(\fraks,\uptau) \in \calS(\check Y|F)}
	\HMR_{\bullet}(\check Y, \upiota;\fraks,\uptau).
\]
Let $(K,s)$ be a web with a $1$-set, and $(\check Y,\upiota)$ be its real double cover.
Suppose $F \subset S^3$ is a connected surface disjoint from $K^{\sfr}$ and $\widetilde F_1 \subset Y$ is a component of the preimage of $F$.
Let
\[
	\HMR_{\bullet}(K,s|\widetilde  F_1) = 
	\bigoplus_{(\fraks,\uptau) \in \calS(\check Y|\widetilde F_1)}
	\HMR_{\bullet}(\check Y|\widetilde  F_1;\fraks,\uptau).
\]
\end{defn}
Given a cobordism $(\check W,\upiota_W): (\check Y_-,\upiota_-) \to (\check Y_+,\upiota_+)$, we define $\calS(\check W|F_W)$ to be the set containing all \spinc structures achieving adjunction equality on cobordism $\check W$ at $\check F_W$.
Suppose $\check F_{\pm}$ is a bifold surface on $\check Y_{\pm}$, and $\check F_W$ contains $\check F_- \cup \check F_+$ on the boundary.
By restricting to real bifold \spinc structures in  $\calS(\check W|F_W)$ we define the corresponding cobordism map
\[
	\HMR(\check W,\upiota_W| F_W):\HMR_{\bullet}(\check Y_-, \upiota_-|F_-) \to \HMR_{\bullet}(\check Y_+,\upiota_+|F_+).
\]
Let $(\Sigma,s_{\Sigma}):(K_-,s_{-}) \to (K_+,s_+)$ be a cobordism. 
Denote the corresponding real double covers by$\check W:\check Y_- \to \check Y_+$.
Assume $F_{\Sigma}$ is a bifold surface in $[0,1] \times S^3$ whose boundary is contained in $\Sigma^{\sfr}$.
Suppose $\widetilde F_{\Sigma} \subset \check W$ is a bifold surface in the preimage of $F_{\Sigma}$ that contains $\widetilde F_{\pm}$ on the boundary, where $\widetilde F_{\pm}$ is a component of the preimage of $F_{\pm}$.
We define
\[
\HMR(\Sigma|\widetilde F_{\Sigma}):\HMR_{\bullet}(K_-,s_-|\widetilde F_-) \to \HMR_{\bullet}(K_+,s_+|\widetilde F_+)
\]
to be $\HMR(\check W| \widetilde F_{\Sigma})$.

Before moving on to the next subsection, let us consider the case when the separating surface admits positive curvature.
The following proposition is similar in spirit to \cite[Proposition~40.1.3]{KMbook2007} for $\HM_{\bullet}$ with certain local coefficient.
\begin{prop}
\label{prop:adj_pos}
	Let $(\check Y,\upiota)$ be a real $3$-bifold and $\check \Sigma$ be a sub-bifold surface.
	Suppose $\check \Sigma = \check\Sigma_1 \sqcup \check\Sigma_2$ where $\check\Sigma_1 \cong \check\Sigma_2$ is a bifold $2$-sphere and $\upiota$ swaps $\check\Sigma_i$'s.
	Suppose the number of bifold points on $\check\Sigma_1$ is either $1$ or $3$.
	Then
	\[
		\HMR_{\bullet}(\check Y, \upiota) = 0.
	\]
\end{prop}
\begin{proof}
	The type of neck stretching argument goes back to \cite{KM1994thom}. 
	For each $T > 0$, we 
	choose an $\upiota$-invariant bifold metric on $\check Y$ containing an isometric copy  of $[-T,T] \times \check\Sigma$, equipped with the product metric $dt^2 + h_{\Sigma}$, such that $h_{\Sigma}$ has constant positive scalar curvature.
	Assume also $h_T$ is independent of $T$ on complement $\check Y_0$ of the neck region.
	Suppose for a real \spinc structure $(\fraks,\uptau)$, we have 
	\[
		\HMR_{\bullet}(\check Y,\uptau) \ne 0.
	\]
	By metric independence of $\HMR$, for every $T > 0$, there is a $3$-dimensional Seiberg-Witten solution.
	Let $(A_T,\Phi_T)$ be a family of 4d Seiberg-Witten solutions  by pulling back the 3d solutions to
	$\check X = S^1 \times \check Y$. 
	Since the scalar curvature of $h_T$ is uniformly bounded across $T$, there is a uniform $C^{0}$ bound below on the spinor $\Phi_T$.
	Moreover, the energy on the neck is bounded by a $T$-independent constant:
	\[
	\calE^{top}_{S^1 \times [-T,T] \times \check \Sigma} (A_T,\Phi_T) = 0-\calE^{top}_{S^1 \times \check Y_0}(A_T,\Phi_T)
	\le -\frac14\int_{S^1 \times \check Y_0}(|\Phi|^2+(s/2))^2 + \int_{S^1 \times \check Y_0} \frac{s^2}{16}
	\]
	One can extract a converging subsequence of $(A_T,\Phi_T)$ to obtain a 3d solution $(B,\Psi)$ on $\reals \times \check\Sigma$ using the basic compactness property \cite[Theorem~5.1.1]{KMbook2007}.
	But the spinor $\Psi$ can be shown to be zero by using the differential inequality $\Delta |\Psi|^2 \le -|\Psi|^2$, see \cite[Equation~(4.22)]{KMbook2007} and \cite[Theorem~5.1.1]{KMbook2007}.
	We have a contradiction, as the equation implies $F_{B} = 0$, while $c_1(\fraks)$ must evaluate to $\pm 1/2$ over $\check\Sigma_1$ .
\end{proof}
\begin{defn}
Let $K$ be a web.
An \emph{embedded bridge} $e$ is an edge of $K$ for which there exists an embedded $2$-sphere that meets $K$ exactly at a single point on $e$.	
\end{defn}
Since $K^{\sfr}$ consists of cycles, an embedded bridge must be coloured $\sfc$ for all $1$-sets.
A straightforward application of Proposition~\ref{prop:adj_pos} is the following.
\begin{thm}
\label{thm:vanish_bridge}
If a web $K$ has an embedded bridge, then $\HMR_{\bullet}(K) = 0$.\qed 
\end{thm}
The case of a $2$-sphere with $3$ bifold points in Proposition~\ref{prop:adj_pos} translates to the following corollary.
\begin{cor}
\label{cor:vanish_3bridges}
Let $(K,s)$	be a web with $1$-set.
Suppose there is an embedded two sphere whose intersection with $K$ consists of exactly $3$ points, all lying on the same $\sfc$-edge. Then $\HMR_{\bullet}(K,s) = 0$.
\qed 
\end{cor}

\subsection{The product bifold $S^1 \times \check \Sigma$}
For $3$-manifolds with $S^1$-actions, the Seiberg-Witten solutions naturally reduce to the vortex equations.
The presentation here follows Mrowka--Ozsv\'{a}th--Yu~\cite{MOY}.

Let $\check \Sigma$ be a bifold Riemann surface and $\Sing^{\sfc}$ be the subset of bifold points.
The canonical \spinc bundle $W_c$ over  $\check \Sigma$ is the rank-$2$ hermitian bundle
$
\C \oplus K_{\check\Sigma}^{-1},
$
where $K_{\check\Sigma}$ is the orbifold canonical bundle and $\C$ is trivial line bundle.
Any other \spinc bundle is a tensor product of $W_c$ with an orbifold line bundle $L$:
\[
L \oplus (L \otimes K_{\check \Sigma}^{-1}).
\]
Let $n = \# \Sing^{\sfc}$ be the number of bifold points on $\check \Sigma$.
As in Section~\ref{subsec:cohom_line_bundles}, around each singular point $x_i$, we define a bifold line bundle $E_i$ locally as a quotient of $\C \times D$:
\[
	(\lambda, z) \mapsto (-\lambda,-z),
\]
glued to the trivial bundle by the transition function $z$.
Any bifold line bundle $L$ can be written as a tensor product 
\[
	|L| \otimes E^{-\beta_1} \otimes \cdots \otimes E^{-\beta_n}.
\]
The line bundle $|L|$ is a smooth bundle over the underlying surface, its degree $\deg|L|$ is the \emph{background Chern number}.

By \cite[Proposition~2.12]{MOY} the Picard group of bifold line bundles over $\check \Sigma$ is isomorphic to
the subgroup of $\frac12\bbZ \times (\bbZ/2)^{n}$, consisting of elements $(c, \beta_1,\dots, \beta_n) \in \frac12\bbZ \times (\bbZ/2)^{n}$ satisfying
\[	
	c \equiv \sum_{i=1}^n \frac{\beta_i}{2} \mod \bbZ.
\]

\begin{defn}
Given a bifold \spinc structure $L \otimes W_c$. 
Let $B_0$ be a unitary bifold connection on $L$.
Suppose
$\alpha_0$ and $\beta_0$ are bifold sections of $L$ and $L \otimes K_{\check \Sigma}^{-1}$, respectively.
	The \emph{K\"{a}hler vortex equations} for the triple $(B_0,\alpha_0,\beta_0)$ are 
	\begin{align*}
		2F_{B_0} - F_{K_{\Sigma}} &= i(|\alpha_0|^2-|\beta_0|^2)\vol_{\check \Sigma},\\
		\delbar_{B_0} \alpha  &= 0,\\
		\delbar_{B_0}^* \beta  &= 0,\\
		\alpha_0 = 0  &\text{ or } \beta_0 = 0.
	\end{align*}
We denote by $\mathcal M^+_v(L)$ the moduli space of bifold solutions to the K\"{a}hler vortex equations, whose $\beta_0$-factors are zero.
Similarly, let $\mathcal M^-_v(L)$ be moduli spaces consisting triples of the form $(B_0,0,\beta_0)$.
There is a natural isomorphism $\mathcal M^+_v(L) \cong \calM_v^-(K_{\Sigma}\otimes L^{-1})$.
\end{defn}
It was known since Taubes~\cite{Taubes1980vortex} that the moduli spaces of vortices can be interpreted as spaces of divisors on the Riemann surface.
An orbifold version of the correspondence was written down in \cite{MOY}:
\begin{thm}[\cite{MOY},Theorem~5]
Let $L$ be an orbifold line bundle of background Chern number $e$. 
The moduli spaces of vortices $\mathcal M^+_v(L)$ is empty if 
	\[
	\deg(L) > \frac12 \deg(K_{\check \Sigma}),
	\]
	and it is naturally diffeomorphic to the $e$-th symmetric product of $\Sigma$, if 
	\[
	\pushQED{\qed}
	\deg(L) < \frac12 \deg(K_{\check \Sigma}).\qedhere
	\pushQED{\qed}
	\]
\end{thm}

\begin{defn}
A \emph{real bifold Riemann surface} is a pair $(\check \Sigma, \upiota_{\Sigma})$ of a bifold Riemann surface $\check \Sigma=(\Sigma, \Sing^{\sfc})$ and
an anti-holomorphic involution $\upiota_{\Sigma}: \check\Sigma \to \check\Sigma$ that preserves the bifold structure.	
\end{defn}

\begin{thm}
	Let $(\check\Sigma,\upiota)$ be a real bifold Riemann surface.
	Let $L$ be a real orbifold line bundle of background Chern number $e$.
	Then the moduli spaces of vortices $\mathcal M^+_v(L)$ is empty if $
	\deg(L) > \frac12 \deg(K_{\check \Sigma})$.
	Furthermore, $M^+_v(L)$ is naturally diffeomorphic to the fixed point locus $(\text{Sym}^e(\Sigma))^{\upiota}$ of the $\upiota$-action on $\text{Sym}^e(\Sigma)$, if 
	\[
	\pushQED{\qed}
	\deg(L) < \frac12 \deg(K_{\Sigma}). \qedhere 
	\pushQED{\qed}
	\]
\end{thm}
It is well-known that Seiberg-Witten solutions on a product $S^1 \times \check \Sigma$ are circle-invariant.
The particular version of the result we need is the following Proposition, which be derived by, for example, adapting the arguments in \cite[Proposition~3.1]{MunozWang2005SWS1S}.
\begin{prop}
	Let $\check Y = S^1 \times \check \Sigma$. Assume $\fraks$ is a bifold \spinc structure on $\check Y$ such that is isomorphic to the pullback of a \spinc structure $L \oplus (L \otimes K^{-1}_{\check\Sigma})$ on $\check \Sigma$.
	Then the moduli space of unperturbed Seiberg-Witten solutions $\calM_{sw}(\check Y,\fraks)$ is isomorphic to $\mathcal M^+_v(L)$.
	\qed 
\end{prop}
Switching notations, let $\check F$ be a bifold Riemann surface.
For a \spinc structure on $S^1 \times \check F$ to admit a Seiberg-Witten solution, it is necessarily that $\langle c_1(\fraks), [S^1 \times \gamma] \rangle = 0$, for a curve $\gamma$ in $\Sigma$, by the adjunction inequality.
Based on the description above, we can compute the Floer homologies for the product \spinc structure, satisfying the adjunction \emph{equality}.
There are two cases, depending on whether the real locus is nonempty.
\begin{lem}
\label{lem:HMR_adj_equality_swap}
Suppose $\check F$ be a bifold Riemann surface	$\chi(\check F) < 0$, such that $\check F = \check F_1 \sqcup \check F_2$, where each $\check F_i$ is identified with a connected bifold $\check F_o$.
Let $\check Y_i$ be the product $\check Y = S^1 \times \check F$ and $\check Y = \check Y_1 \sqcup \check Y_2$ be the disjoint union.
Let $\upiota$ be the swapping involution, i.e. it exchanges the two $\check Y_i$'s and acts by identity under identification $\check Y_i \cong S^1 \times \check F_o$.
Then
\[
	\HMR_{\bullet}(\check Y|F_1) =
	\bbF_2,
\]
\end{lem}
\begin{proof}
	There exists only one compatible real structure, namely the identity lift, see Section~\ref{subsec:case_real_loci}.
	Since the \spinc structure is non-torsion, there is no reducible generator to $\HMR$.
	Furthermore, there exists a unique irreducible generator of $\HMR_{\bullet}(\check Y|F_1)$, namely the canonical solution to the vortex equations over $\check F_1$.
\end{proof}
\begin{lem}
Let $(\check F, \upiota_F)$ be a connected real bifold Riemann surface such that $2\text{genus}(F) - 2 + \#\Sing^{\sfc} > 0$.
Let $\check Y = S^1 \times \check Y$ and $\upiota = \upiota_{S^1} \times \upiota_F$ be a product involution such that $\upiota_{S^1}$ is the reflection along two points on $S^1$.
Let $\fraks_0$ be the pullback of the canonical \spinc structure on $\check F$.
Then
\[
	\HMR_{\bullet}(\check Y|\check F) =\bigoplus_{\uptau}\HMR_{\bullet}(\check Y;\fraks_0, \uptau) 
	=
	\bigoplus_{\uptau} \bbF_2,
\]
where the direct sum is taken over compatible lifts $\uptau$.
In particular, if $Y/\upiota = S^3$, then the lift is unique and we have
\[
\HMR_{\bullet}(\check Y|F) =
 \bbF_2.
\]
\end{lem}

\subsection{An excision theorem}
Let us prove a version of Floer's excision theorem, following \cite[Section~3]{KM2010excision}.

Let $(\check W,\upiota_W): (\check Y_-,\upiota_-) \to (\check Y_+,\upiota_+)$ be a real cobordism.
Assume $\check F_{\pm}$ is a connected subbifold surface in $\check Y_{\pm}$ disjoint from the real locus such that
\[
	\chi(\check F_{\pm}) < 0.
\]
Suppose $\check W$ contains a subbifold
\[
	\check Z = \check G \times S^1
\]
where $\check G$ is a disjoint union of two identical bifold surfaces $\check G_1, \check G_2$ identified with $\check G_o$.
Assume the restriction of $\upiota_W$ on $\check Z$ is the swapping involution in Lemma~\ref{lem:HMR_adj_equality_swap}.
To obtain a new cobordism $\check W^{\dag}$, we first cut along $\check Z$ to obtain a real bifold $
\check W'$. 
The new boundary components is a union of two $\check G \times S^1$, each preserved by the induced involution $\upiota_{W'}$.

Let $\check W^{\dag}$ be result of attaching two copies of $\check G \times D^2$ to $\check W'$. 
Consider the following two collections of surfaces on the cobordisms
\[
	F_{\check W} = (\check F_- \cup \check F_+ \cup \check G_1) \subset W, \quad
	F_{\check W^{\dag}} = (\check F_- \cup \check F_+ \cup \check G_1) \subset W^{\dag}.
\]
\begin{prop}
\label{prop:excision_prepare}
The maps $\HMR(\check W|\check F_W)$ and $\HMR(\check W^{\dag}|\check F_{W^{\dag}})$ are the same.
\end{prop}
\begin{proof}
	This the analogue of \cite[Proposition~2.5]{KM2010excision}.
	The key is that $\HMR_{\bullet}(\check Z, \upiota_{Z}) \cong \Ftwo$ which follows from Lemma~\ref{lem:HMR_adj_equality_swap}.
	Let us recall the argument in \cite{KM2010excision}.
	We view the real bifold $\check W'$ as a cobordism from $\check Y_- \cup \check Z$ to $\check Y_+ \cup \check Z$.
	By gluing $[0,1] \times \check Z$ and $(D^2 \sqcup D^2) \times G$, we obtain $W$ and $W^{\dag}$, respectively.
	It therefore suffices to prove that $\check W_1 := [0,1] \times \check Z$ and $(D^2 \sqcup D^2) \times G$, as cobordism from $\check Z$ to $
	\check Z$, induce the same map $\HMR_{\bullet}(\check Z|\check G_1) \to \HMR_{\bullet}(\check Z|\check G_1)$.
	
	Indeed, the first map is the identity, and to see the second is also the identity, it remains to show $\HMR_{\bullet}(D^2 \times G) = 1$. 
	We remove a copy of $\check G$ from the interior (recall $\check G$ is a disjoint union of two surfaces). The resulting real bifold is a cobordism $W_2: \check Z \sqcup \check Z \to \check Z$. The induced map $\HMR_{\bullet}(\check Z|\check G_1) \otimes \HMR_{\bullet}(\check Z|\check G_1) \to \HMR_{\bullet}(\check Z|\check G_1)$ recovers $\HMR_{\bullet}(\check W_1)$ upon evaluating $\HMR_{\bullet}(D^2 \times \check G)$. 
	So $\HMR_{\bullet}(\check W_1)$ is an isomorphism, and $\HMR_{\bullet}(D^2 \times \check G) = 1$.
\end{proof}
The setup for Floer's excision is as follows. 
In the remaining part of this section, we will suppress the ``check'' symbol, that signified bifolds as opposed to manifolds, in order to avoid potential confusion with the ``tilde''.	
 
Suppose $(Y,\upiota)$ is a closed real $3$-bifold, such that $Y/\upiota$ has either one or two components.
Let $Y_-, Y_+$ be the preimage of $ Y \to Y/\upiota$ of two components of $Y/\upiota$.
Assume $\Sigma_-$ and $\Sigma_+$ are two isomorphic real subbifold surfaces of $Y$, having no real points.
Write $\Sigma = \Sigma_- \cup \Sigma_+$.
In particular, $\Sigma_{\pm}$ is a disjoint union $\Sigma_{\pm,1} \sqcup \Sigma_{\pm,2}$.
If $Y/\upiota$ is connected, we assume $ \Sigma_{\pm}$ represent independent homology classes.
If $Y/\upiota$ is disconnected, assume each component $\Sigma_{\pm,i}$ is nonseparating.
We fix an orientation-preserving diffeomorphism $h: \Sigma_- \to \Sigma_+$ such that $h(\Sigma_{-,i}) = \Sigma_{+,i}$.
Cutting $Y$ along $\Sigma$, we obtain a real bifold $Y'$ with four pairs of boundary components (eight connected surfaces in total):
\[
	\del Y' = \Sigma_- \cup (- \Sigma_-) \cup  \Sigma_+ \cup (-\Sigma_+).
\]
If $Y/\upiota$ has two components, then so does $Y'$.
Let $Y' = Y'_- \cup Y'_+$ and let $\widetilde Y$ be the real bifold obtained by gluing boundary components $\Sigma_-$ to $-\Sigma_+$ and $\Sigma_+$ to $-\Sigma_-$.
Unlike the setting in \cite{KM2010excision}, the resulting bifold $\widetilde Y$ is connected if and only if $Y$ is.

Denote the image of $\Sigma_- = -\Sigma_+$ by $\widetilde \Sigma_-$ and the image of $\Sigma_+ = -\Sigma_-$ by $\widetilde \Sigma_+$.
Also, we denote $\Sigma_i$ by $\Sigma_{-,i} \cup \Sigma_{+,i}$. 
Finally, we write $\widetilde \Sigma = \widetilde \Sigma_- \cup \tilde \Sigma_+$ and $\widetilde \Sigma_i = \widetilde \Sigma_{-,i} \cup \widetilde \Sigma_{+,i}$.

\begin{thm}
\label{thm:excision}
	Suppose $\widetilde Y$ and $\widetilde F$ are obtained from above, and assume $\chi_{orb}(\check \Sigma_-) = \chi_{orb}(\check \Sigma_+) < 0$.
	Then there is an isomorphism
	\[
		\HMR_{\bullet}(Y|\Sigma) =
		\HMR_{\bullet}(\widetilde Y|\widetilde \Sigma).
	\]
\end{thm}
\begin{proof}
	The proof strategy goes back to Floer, although our setup are different from \cite{BraamDonaldson1995FloerMem} and \cite{KM2010excision}.
	First, we have a cobordism $(W,\upiota_W):\widetilde Y \to Y$.
	Regardless whether $Y/\upiota$ is connected, there exists a codimension-0 real subbifold $\mathbf U$ of $W$ that admits a projection map $\pi: \mathbf U \to U$ onto a saddle $U$, whose fibre is $\Sigma_-$.
\begin{figure}[h!]
	\centering
	\includegraphics[width=0.35\linewidth]{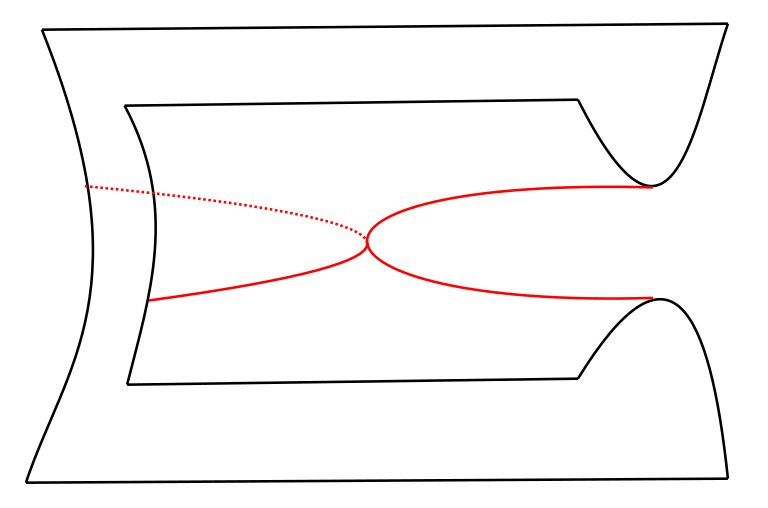}
	\caption{The saddle surface $U$.}
	\label{fig:saddle}
\end{figure}

See Figure~\ref{fig:doublesaddle} for a schematic picture of $W$ when $Y/\upiota$ is disconnected, where each of the blue and red curves correspond to $(\text{curve} \times \Sigma_{-,1})$.
The involution $\upiota_{{W}}:{W} \to {W}$ interchanges the blue and the red regions. 

\begin{figure}[h!]
	\centering
	\includegraphics[width=0.6\linewidth]{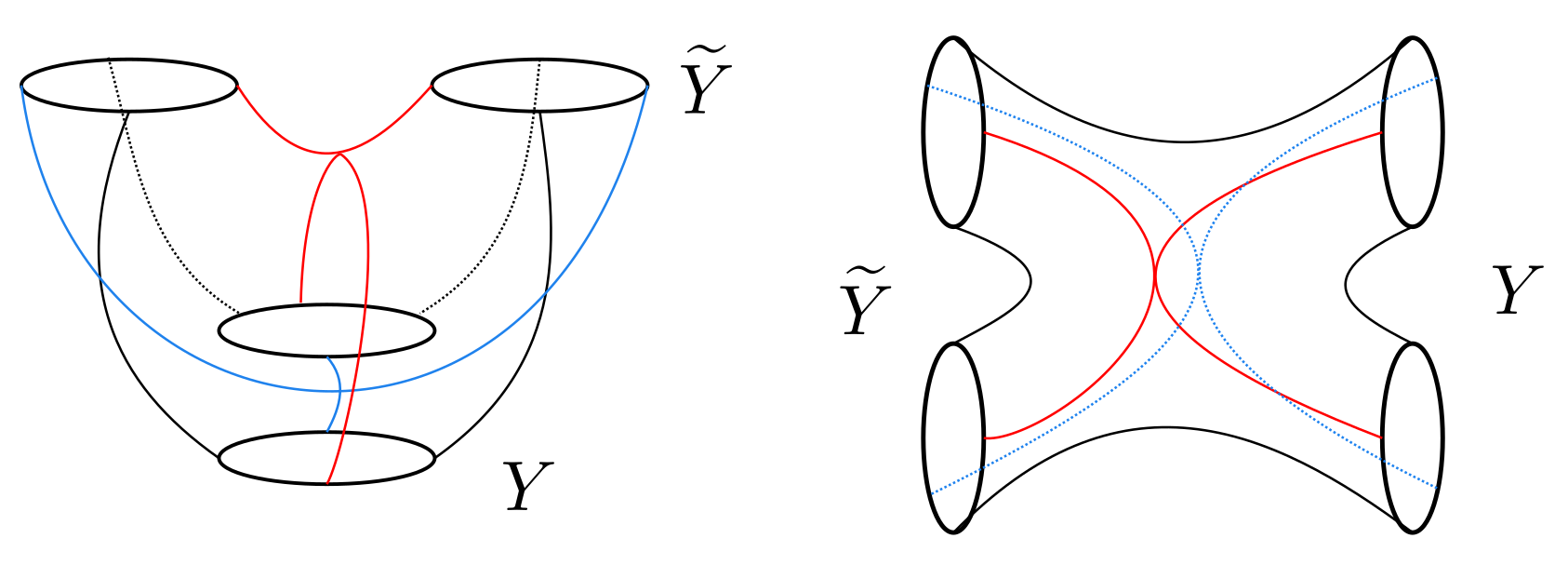}
	\caption{Two schematic views of the cobordism ${W}$. (Drawn as doubles of saddles.)}
	\label{fig:doublesaddle}
\end{figure}
	Similarly, we have a cobordism $\bar W$ that goes the other way.
	It suffices to prove the following two cobordism maps are mutual inverses
	\begin{align*}
		\HMR(W,\upiota_W):
		\HMR_{\bullet}(\widetilde Y, \upiota_{\widetilde Y}|\widetilde \Sigma_1) 
		&\to
		\HMR_{\bullet}( Y, \upiota_{Y}| \Sigma_1), \\
		\HMR(\bar W,\upiota_{\bar W}):
		\HMR_{\bullet}( Y, \upiota_{ Y}| \Sigma_1) 
		&\to
		\HMR_{\bullet}(\widetilde Y, \upiota_{\widetilde Y}|\widetilde \Sigma_1).
	\end{align*}
	
	Let $X$ be the composition of $W$ and $\bar W$, as a real cobordism from $\widetilde Y$ to $\tilde Y$. 
	Since $\Sigma_1$ and $\widetilde\Sigma_1$ are homologous in $X$, the map induced by $X$ factors through $\HMR_{\bullet}(Y,\upiota_Y|\Sigma_1)$.
	Moreover, $X$ contains a codimension-0 real sub-bifold (with corners) $\mathbf V \cong \Sigma_- \times V$ that maps onto $V$ by $\pi: \mathbf V \to V$, where $V$ is the union of two saddles as Figure~\ref{fig:composite_saddles}.
	\begin{figure}[h]
	\centering
	\includegraphics[width=0.37\linewidth]{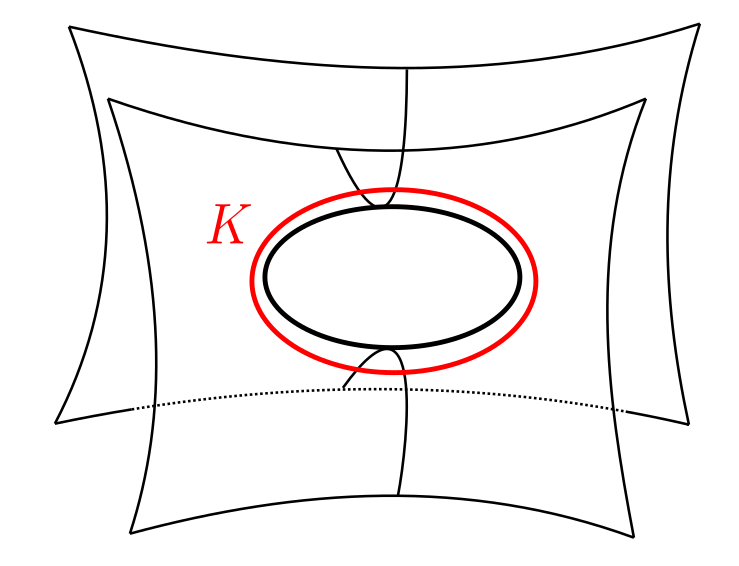}
	\caption{The surface $V$ with curve $k$.}
	\label{fig:composite_saddles}
	\end{figure}
	\begin{figure}[h]
	\centering
	\includegraphics[width=0.5\linewidth]{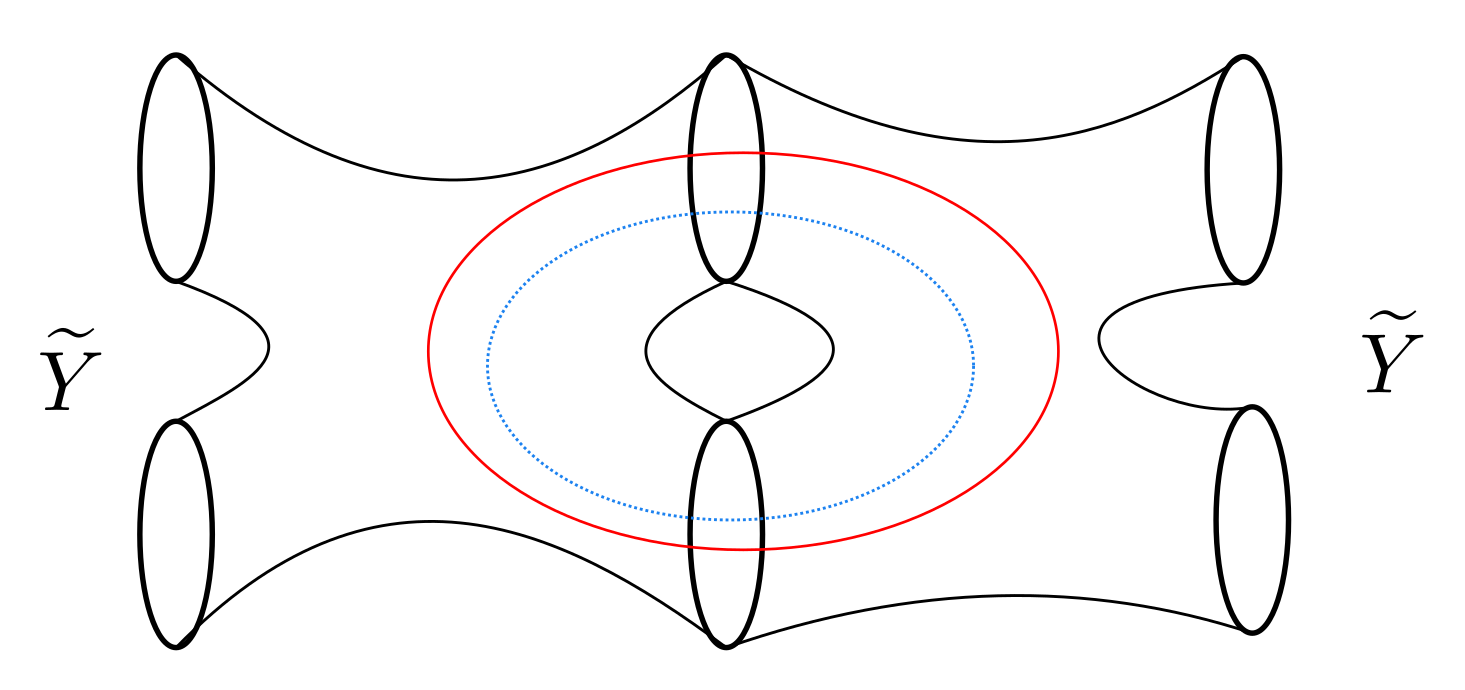}
	\caption{Schematic picture of the composition $X$ of two cobordisms.}
	\end{figure}
	In particular, there is a closed curve $k$ in $V$, whose preimage $\mathbf K = \pi^{-1}(k)$ is, equivariantly with respect to the real involution, isomorphic to
	\[
		\Sigma_- \times k.
	\]
	Cutting along $\mathbf K$, we obtain a real bifold with boundary $X'$.
	Let $X^*:\widetilde Y \to \widetilde Y$ be the cobordism constructed by gluing two copies of $\Sigma_- \times D^2$, along $\del D^2 \cong k$.
	However, $X^*$ is nothing but the product cobordism, so $\HMR(X^*,\upiota_{X^*})$ is simply the identity map.
	In addition, $X^*$ and $X$ falls into the scenario of Proposition~\ref{prop:excision_prepare}, so we deduce also that 
	\[
		\HMR(X,\upiota_X)  = \HMR(W,\upiota_W) \circ \HMR(\widetilde W,\upiota_{\widetilde W}) = \Id.
	\]
	The opposite composite $W \circ \bar W$ is similar.
	We will consider a map from a codimension-0 real subbifold onto a new surface $\widetilde U$, cut along the preimage of a curve $\widetilde k \subset \widetilde U$, and apply Proposition~\ref{prop:excision_prepare}.
\end{proof}

\subsection{Application to mapping tori}
\label{subsec:mapping_tori}
As an illustration of the excision theorem, consider the following class of webs built from braids.
Let $\gamma$ be an $n$-strand braid such that $n$ is odd and $n \ge 5$.
Suppose $\gamma$ is embedded in the set
\[
	\{(x,y,z) \in \reals^3: x^2 + y^2 \le 1, 0 \le z \le 1\}
\]
in way that the boundary points lie on the unit circle $U_1, U_0$ in the $z = 1$ and $z=0$ plane, respectively.
Let $K_{\gamma}$ be the union of $\gamma$ and $U_1,U_0$, thought of as a web in $\reals^3 \cup \infty = S^3$.
Let $s_{\gamma}$ be the $1$-set such that $U_1,U_0$ are coloured $\sfr$ and the braid is coloured $\sfc$.
The goal is to show that $\HMR_{\bullet}(K_{\gamma}, s_{\gamma})$ is nonzero at the top first Chern class. We make the statement precise as follows.

The real double cover $Y_{\gamma}$ of $s_{\gamma}$ is $(S^1 \times S^2, \tilde K^{\sfc})$ such that $\widetilde K^{\sfc}$ is an $n$-braid.
Let $\Sigma_{\gamma}$ be a $2$-sphere in $S^3$ that intersects $K_{\gamma}$ at five points. 
Assume that $\Sigma_{\gamma}$ bounds a 3-ball $B^3$ where $(B^3, B^3 \cap K_{\gamma})$ looks like the trivial $n$-ball.
This can be achieved, for example, by lowering $S^2$ below the $z = \epsilon$ plane for $\epsilon > 0$ small.
Let $\widetilde{\Sigma}_+$ be the preimage of $\Sigma $ in $Y$, and label the components as $\widetilde{\Sigma}_{+,1},\widetilde{\Sigma}_{+,2}$, respectively.
\begin{prop}
	\label{prop:excision_app}
$\HMR_{\bullet}(K_{\gamma},s_{\gamma}|\widetilde\Sigma_{+,1}) = \bbF_2$.
\end{prop}
\begin{proof}
	Let $\gamma'$ be the inverse braid to $\gamma$.
	Let $Y = Y_{\gamma} \sqcup Y_{\gamma'}$ be the disjoint union of the two real double covers of $\gamma,\gamma'$.
	In particular, this is a union of a mapping torus and its inverse.
	Let $\Sigma_{-} = \Sigma_{-,1} \cup \Sigma_{-,2}$ be the preimage of $\Sigma_{\gamma'}$ in $Y_{\gamma'}$. 
	We label $\Sigma_{-,i}$, $i = 1,2$ so that the bifold $\widetilde Y$ from cutting and regluing as in the excision theorem result in the two real covers
	\[
		Y_{\Id} \sqcup Y_{\gamma \circ \gamma'}
	\]
	of the trivial $n$-braid $\Id$, and $\gamma \circ \gamma'$ which is isotopic to the trivial $n$-braid. 
	By Lemma~\ref{lem:HMR_adj_equality_swap}, $\HMR_{\bullet}(Y_{\gamma \circ \gamma'}|\widetilde\Sigma_{+,1}) \cong \HMR_{\bullet}(Y_{\Id}|\widetilde\Sigma_{-,1}) = \bbF_2$.
	By Theorem~\ref{thm:excision},  $\HMR_{\bullet}(Y_{\gamma}|\Sigma_{-,1}) \otimes \HMR_{\bullet}(Y_{\gamma'}|\Sigma_{+,1})$ is rank-$1$.
\end{proof}

\section{Examples}
\label{sec:example}
\begin{exmp}[Unknot]
\label{exmp:unknot}
Let $\U_1$ be the unknot and $p \in \U_1$ be a basepoint.
There are two $1$-sets for the unknot $\U_1$: either $\U_1$ is coloured by $\sfr$ or $\sfc$.
Let $s$  be the unique based $1$-set such that $\U_1$ is $\sfr$. As in \cite[Prop.~14.4]{ljk2022}, there is a unique \spinc structure $\uptau$ and we have:
\[
	\widetilde{\HMR}_*(\U_1,p,s)=\bbF_2,\quad 
	\widecheck{\HMR}_*(\U_1,s,\uptau) = \bbF_2[\upsilon^{-1},\upsilon]/\bbF_2[\upsilon], \quad
	\overline{\HMR}_*(\U_1,s,\uptau) = \bbF_2[\upsilon^{-1},\upsilon], \quad
	\widehat{\HMR}_*(\U_1,s,\uptau) = \bbF_2[\upsilon],
\] 
where $\deg \upsilon = -1$. 

Suppose $s'$ is the $1$-set where $\U_1$ is coloured $\sfc$.
The real double cover $\check Y$ is the disjoint union of $(S^3,\U_1)$ where $\upiota$ is the swapping involution.
The computation is essentially a bifold version of the computation of ordinary monopole Floer homology of $(S^3, \U_1)$.
Choose a positive scalar curvature (PSC) bifold metric on one copy of two $(S^3, \U_1)$ and pull it back to the second copy by $\upiota$.
There are two choices of bifold \spinc structures by Corollary~\ref{cor:bifold_spinc_S3} and the real structure is simply the identity map to the conjugate \spinc structure on the second copy.
By the usual Weitzenb\"{o}ck argument, there exists a unique reducible Seiberg-Witten solution $(B,0)$.
Let $\frakq$ be a small admissible perturbation such that spectrum of $D_{B,\frakq}$ is simple. 
The rest of the calculation is the same as \cite[Section~22.7]{KMbook2007}.
The critical points can be labelled by $\{\fraka_i: i \in \bbZ\}$ corresponding to eigenvectors of the perturbed operators.
For $i \ge 0$, the critical point $\fraka_i$ is boundary stable and otherwise boundary-unstable.
Moreover, $\gr([\fraka_i],[\fraka_{i-1}]) = 2$ unless $i=0$ where $\gr([\fraka_0],[\fraka_{-1}]) = 1$. For parity reasons, all differentials vanish, and 
\[
	\widecheck{\HMR}_*(\U_1,s',\uptau) = \bbF_2[U^{-1},U]/\bbF_2[U], \quad
	\overline{\HMR}_*(\U_1,s',\uptau) = \bbF_2[U^{-1},U], \quad
	\widehat{\HMR}_*(\U_1,s',\uptau) = \bbF_2[U],
\]
where $\deg U = -2$. 
Since $s'$ is not a based $1$-set, we have 
\[
	\widetilde{\HMR}_*(\U_1,p) = 	\widetilde{\HMR}_*(\U_1,p,s) = \bbF_2.
\]
\end{exmp}
\begin{exmp}[Unlinks]
	Let $\U_n$ be the $n$-component unlink, and choose a basepoint $p$.
	Suppose $s_k$ is a $1$-set containing $k$ $\sfr$-circles.
	The real double cover $\check Y$ has underlying $3$-manifold $\sharp_{k-1}(S^1 \times S^2)$ and can be equipped with an $\upiota$-invariant bifold metric having positive scalar curvature.
	Assume $(\fraks, \uptau)$ is the unique torsion real $\spinc$ structure.
	Note that a non-torsion $\spinc$ structure gives trivial Floer homology because of positive scalar curvature.
	The web Floer homology can be computed the same way as \cite[Proposition~14.3]{ljk2022}:
	\[
	\widehat{\HMR}_*(\U_n,s_k,\uptau) = \frac{\bbF_2[\upsilon_1,\dots,\upsilon_n]}{\upsilon_i^2=\upsilon_j^2}
	\cong \Lambda[x_1,\dots,x_{k-1}] \otimes \bbF_2[\upsilon].
	\]
	where $\deg(x_i)=-1$ and $\Lambda[x_1,\dots,x_k]$ is the exterior algebra over $\bbF_2$. The two other flavours are analogous.
	As for the framed version, suppose $s_k$ is a based $1$-set. 
	Then  $k \ge 1$ and the circle containing $p$ is coloured $\sfr$.
	We have
	\[
		\widetilde{\HMR}_*(\U_n,p,s_k) = \bbF_2^{2^{k-1} \times 2^{n-k}} = \bbF_2^{2^{n-1}},
	\]
	where $2^{k-1}$ corresponds to the rank of $\Lambda[x_1,\dots,x_{k-1}]$.
	Notice for a given $1$-set having $k$ $\sfr$-edges, there are $2^{n-k}$ choices of bifold \spinc structures corresponding to $2^{n-k}$  isotropy types of the $\sfc$-circles.
	Notice also the lifts of the $\sfc$-circles are null-homologous on the real covers.
	There are $\binom{n-1}{k-1}$ choices of $s_k$.
	The total rank is therefore
	\[
		\dim \widetilde{\HMR}_*(\U_n,p) = \sum_{k=1}^n \binom{n-1}{k-1} 2^{n-1} = 4^{n-1}.
	\]
	
	For comparison, the number of Tait colourings of $\U_n$ is $3^n$.
	The discrepancy between these  quantities is due to isotropy types.
	Given a based $1$-set $s$, we choose representatives of the $2^{n-k}$ real bifold \spinc structures that can be obtained by changing the isotropy types on the $\sfc$-circles.
	Let $T(s)$ be the set of representatives.
	Then
	\[
	\sum_{s \in \{\text{based 1-sets}\}}\sum_{\uptau \in T(s)}\dim \widetilde{\HMR}_*(\U_n,p,s,\uptau) = 
	 \sum_{k=1}^n \binom{n-1}{k-1} 2^{k-1}=3^{n-1}.
	\]
\end{exmp}
\begin{figure*}[h!]
	\centering
	\begin{subfigure}[t]{0.5\textwidth}
		\centering
		\includegraphics[height=1.2in]{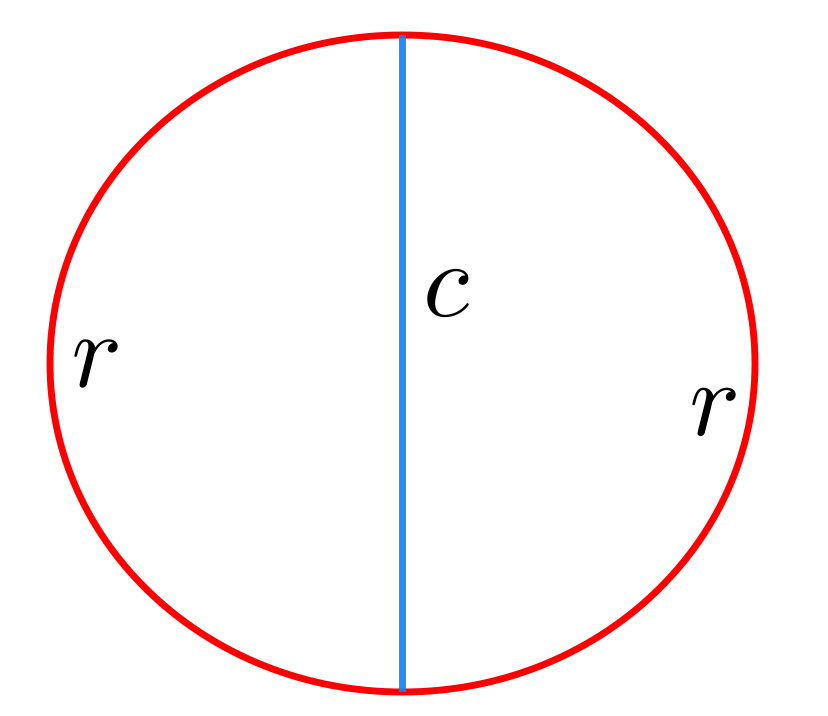}
		\caption{The theta graph}
		\label{fig:theta}
	\end{subfigure}%
	~ 
	\begin{subfigure}[t]{0.5\textwidth}
		\centering
		\includegraphics[height=1.2in]{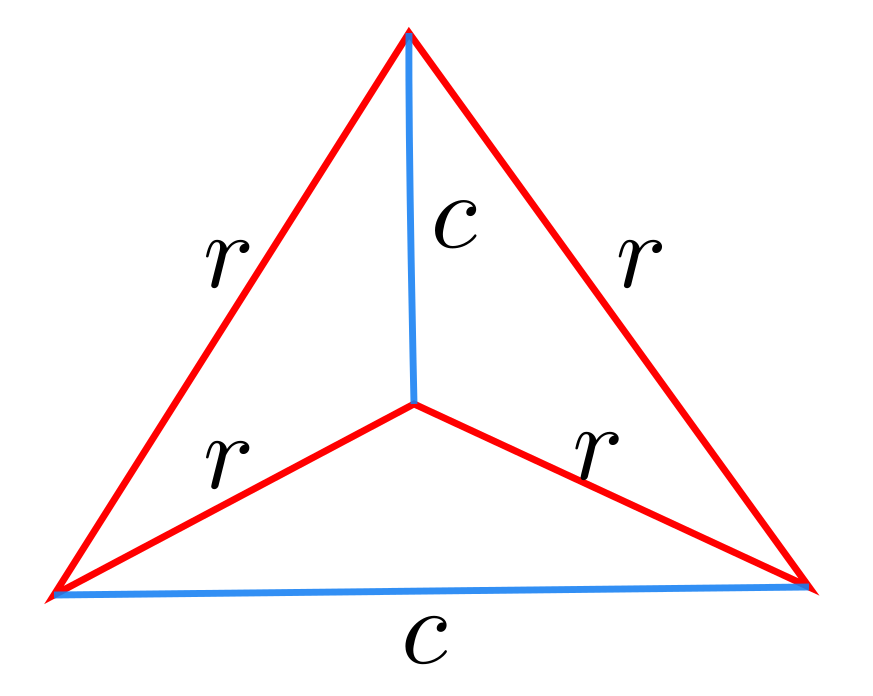}
		\caption{The tetrahedron graph.}
	\end{subfigure}
	\caption{Examples of $1$-sets.}
\end{figure*}
\begin{exmp}
	[Theta graph] 
	Let $\Theta$ be the theta graph and $p \in \Theta$ be a basepoint.
	The bifold $(S^3,\Theta)$ is the global quotient of  the $V_4$ on $(\reals^3 \cup \infty)$ as diagonal matrices in $\SO(3)$ whose entries are $\{\pm 1\}$.
	A $1$-set is equivalent to a choice of an edge, of which there are three possibilities (see Figure~\ref{fig:theta}).
	The three cases are isomorphic.
	Note $\Theta$ has 6 Tait colourings.
	
	The real double cover $\check Y$ is diffeomorphic to $(S^3, \U_1)$.
	(More concretely, $\tilde \Theta^{\sfr}$ can be thought as the $x_3$-axis in $\reals^3$, and $\U_1$ is the union of the positive $x_1$-axis and $x_2$-axis, together with $\infty$.)
	Equip $(S^3,\U_1)$ with an $\upiota$-invariant bifold metric of positive scalar curvature. 
	(It suffices to take the round metric using the model above.)
	There are two bifold \spinc structures over $(S^3,\U_1)$ corresponding to two isotropy data.
	For a given bifold \spinc structure $\fraks$, there is a unique real structure $\uptau$.
	
	Assume the spectrum of the Dirac operator is simple, and there are no irreducible critical points.
	The computation of $\HMR^{\circ}$ is the same as $\HMR^{\circ}(\Theta^{\sfr})$ without a bifold singularity.
	The chain complex for $\HMR^{\circ}(\Theta,s,\uptau)$ is generated by $\{\fraka_i\}$, such that $\gr([\fraka_i],[\fraka_{i-1}]) = 1$ and $\gr([\fraka_0],[\fraka_{-1}]) = 0$.
	Similar to~\cite[Proposition~14.4]{ljk2022},
\[
	\widecheck{\HMR}_*(\Theta,s,\uptau) = \bbF_2[\upsilon^{-1},\upsilon]/\bbF_2[\upsilon], \quad
	\overline{\HMR}_*(\Theta,s,\uptau) = \bbF_2[\upsilon^{-1},\upsilon], \quad
	\widehat{\HMR}_*(\Theta,s,\uptau) = \bbF_2[\upsilon],
\]
where $\deg \upsilon = -1$.
Hence $\widetilde{\HMR}_*(\Theta,p,s,\uptau)=\bbF_2$.
There is a unique based $1$-set  and two relevant real bifold \spinc structures. 
We conclude
\[
	\widetilde{\HMR}_*(\Theta,p) = \bbF_2 \oplus  \bbF_2 .
\]
\end{exmp}

\begin{exmp}
	[Theta graph $\sqcup$ based unknot] 
	Let $(\Theta',p')$  be the disjoint union of $\Theta$ with an unknot $\U_1$ with a basepoint $p' \in \U_1$.
	Let $s$ be a based $1$-set (there are $3$ in total).
	The real double cover $\check Y$ is diffeomorphic to $S^1 \times S^2$, whose bifold locus is a great circle in one of the $S^2$ slice.
	By the arguments above,
	\[
		\widecheck{\HMR}_*(\Theta', s) \cong \widecheck{\HMR}_*(\U_1  \sqcup \U_1, s) 
		\cong \frac{\bbF_2[x]}{x^2} \otimes \frac{\bbF_2[\upsilon^{-1},\upsilon]}{\bbF_2[\upsilon]}.
	\]
	Due to the three choices of based $1$-sets and two choices of isotropy types  for each $1$-set, $\widetilde{\HMR}_*(\Theta', p') \cong \bbF_2^{12}$.
	Finally, if for each based $1$-set we choose only one real bifold \spinc structure, then the resulting direct sum of web homology groups has rank $6$, agreeing with the number of Tait colourings.
\end{exmp}
\begin{exmp}
	[Tetrahedron] 
	Let $T$ be the $1$-skeleton of the tetrahedron.
	There are three isomorphic $1$-sets, related by a rotation symmetry of the plane.
	Given such a $1$-set $s$, the real cover $\check Y$ is $(S^3,\U_2)$, equipped with an involution which fixes an unknot that meets each component of $\U_2$ at one point.
There again exists an invariant PSC metric and 
\[
	\widehat{\HMR}_*(T,s) = \bbF_2[\upsilon]^{\oplus 2},
\]
where each bifold isotropy type contributes one $\bbF_2[\upsilon]$-tower.
\hfill$\diamond$
\end{exmp}
Let the \emph{$n$-prism} $L_n \subset \reals^3$ be the planar web with consisting of two cocentric $n$-gons and $n$ straight line segments connecting the inner and outer vertices.
Alternatively, $L_n$ is the web $K_{\Id_n}$ in Section~\ref{subsec:mapping_tori} when $\Id_n$ is the trivial $n$-braid.
The set $s_n = s_{\Id_n}$ of $n$ line segments is a $1$-set (coloured blue in Figure~\ref{fig:sampleL2}-\ref{fig:sampleL5}).
The corresponding real cover $\check Y_n$ is $S^1 \times S^2_n$, where $S^2_n$ is the $2$-sphere with $n$ bifold points.
The covering involution $\upiota:\check Y_n \to \check Y_n$ exchanges the two hemispheres bounded by the the great circle, and acts by $-1$ on $S^1 \cong \reals/\bbZ$.
Let us consider the prism webs from $n = 1$ to $5$.
\begin{exmp}[$n = 1$]
\label{exmp:L1}
This example contains an embedded bridge.
Theorem~\ref{thm:vanish_bridge} implies
	\[
	\HMR_*(L_1) = 0.
	\]
\end{exmp}
\begin{exmp}[$n = 2$]
There are three $1$-sets.
\begin{figure}[h!]
	\centering
	\includegraphics[height=0.13\textheight]{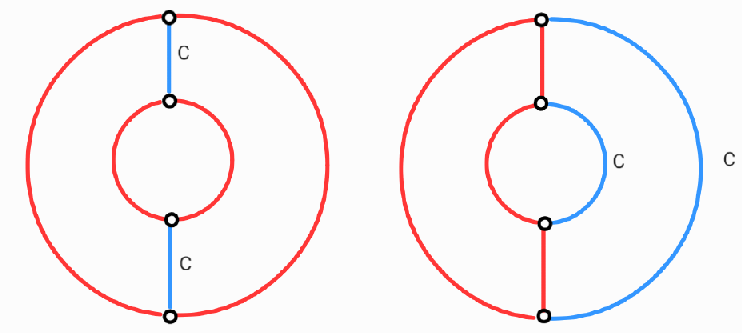}
	\caption{Sample $1$-sets of $L_2$.}
	\label{fig:sampleL2}
\end{figure}
For the $1$-set $s_2$, it suffices to consider the reducible solutions.
There is a unique torsion real orbifold spin\textsuperscript{c} structure on $\check Y_2$.
There is an $S^1$-family of flat $\U(1)$-connections with holonomy $-1$ around each of the singular strand, preserved by $\upiota$. 
The flat connections are parametrized by the holonomy along the $S^1$ factor, and arguing with the $\upiota$-invariant PSC metric as in \cite[Section~22.7]{KMbook2007},
\[
	\widehat{\HMR}_{*}(L_2,s_2) = \bbF_2[\upsilon] \oplus \bbF_2[\upsilon]\langle -1 \rangle.
\]
There are two more $1$-sets, both topologically the same as an unknot with two arcs attached, as in Figure~\ref{fig:sampleL2}.
Choose one such $1$-set $s'$.
The real cover is $S^3$ singular along the 2-component unlink $\U_2$.
Since there exists a PSC bifold metric on $(S^3,\U_2)$,
\[
	\widehat{\HMR}_*(L_2,s') = \bbF_2[\upsilon]^{\oplus 2},
\]
where each $\bbF[\upsilon]$-tower comes from a choice of isotropy of bifold \spinc structure.
\end{exmp}
\begin{exmp}[$n = 3$]
	\begin{figure}[h!]
		\centering
		\includegraphics[height=0.13\textheight]{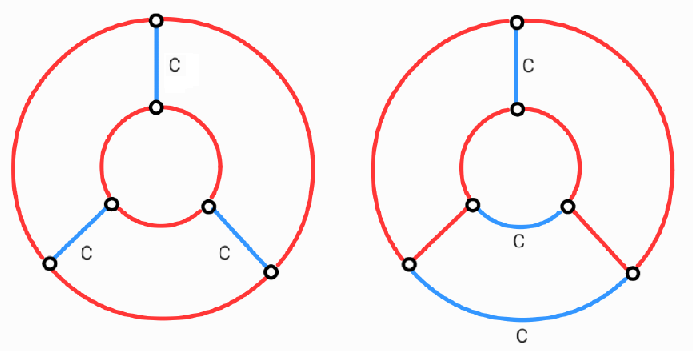}
		\caption{Sample $1$-sets of $L_3$.}
		\label{fig:sampleL3}
	\end{figure}
	For the $1$-set $s_3$, Corollary~\ref{cor:vanish_3bridges} gives
	\[
	\HMR_*(\check Y_3,s_3) = 0.
	\]
	There are three $1$-sets that are topologically a single real circle with $3$ $\sfc$-arcs attached, as in Figure~\ref{fig:sampleL3}.
	Each of these $1$-sets $s'$ contributes two $\bbF[\upsilon]$-towers. We obtain
	\[
		\widehat{\HMR}_*(L_3) = \bbF_2[\upsilon]^{\oplus 6}.
	\]
Consider $L_3 \sqcup \U_1$ with a basepoint $p' \in \U_1$. 
Every $1$-set of $L_3$ corresponds to a based $1$-set of $L_3 \sqcup \U_1$.
We have $	\widehat{\HMR}(L_3 \sqcup \U_1,s') = \bbF_2[x]/(x^2) \otimes \bbF_2[\upsilon]^{\oplus 6}$.
There are $6$ Tait colourings of $L_3$, and 
\[
		\widetilde{\HMR}_*(L_3 \sqcup \U_1, p') = \bbF_2^{\oplus 12}.
	\]
\end{exmp}
\begin{exmp}[$n = 4$]
		\begin{figure}[h!]
		\centering
		\includegraphics[height=0.13\textheight]{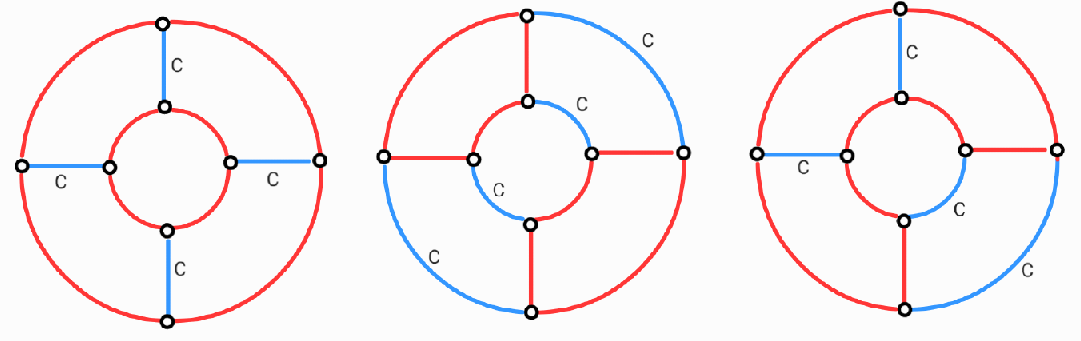}
		\caption{Sample $1$-sets of $L_4$.}
		\label{fig:sampleL4}
	\end{figure}
	The $1$-set $s_4$ is a borderline case.
	Since $S^2_4$ admits a flat metric, the unperturbed Seiberg-Witten equation admits only reducible solutions.
Consider the unique torsion \spinc structure. 
The orbifold flat connections is parametrized by $S^1$ whenever $n$ is even.
The situation is similar to the three-torus in \cite[Section~37]{KMbook2007}.
Our case is in fact simpler as the bifold flat connections have nontrivial holonomy around the $4$ points and the
twisted Dirac operator has trivial kernel.
We again have
\[
\widehat{\HMR}_{*}(L_4, s_4)= \mathbb F_2[\upsilon] \oplus \mathbb F_2[\upsilon]\langle -1\rangle.
\]
There are two more $1$-sets isomorphic to $s_4$, and four $1$-sets isomorphic to a real circle with $4$ $\sfc$-arcs attached as in Figure~\ref{fig:sampleL4}.
Hence
\[
	\widehat{\HMR}_*(L_4) = \left(\mathbb F_2[\upsilon] \oplus \mathbb F_2[\upsilon]\langle -1\rangle\right)^{\oplus 3} 
	\oplus \bbF_2[\upsilon]^{\oplus 4 \times 2}.
\]
\end{exmp}
\begin{exmp}[$n = 5$]
\label{exmp:L5}
	\begin{figure}[h!]
	\centering
	\includegraphics[height=0.13\textheight]{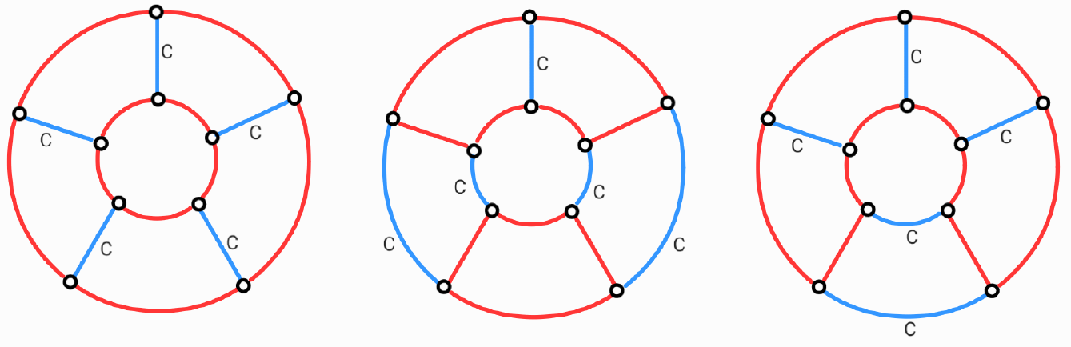}
	\caption{Sample $1$-sets of $L_5$.}
	\label{fig:sampleL5}
\end{figure}
	For $n \ge 5$,  the bifold $S^2_n$ admits a metric of negative curvature.
By Proposition~\ref{prop:excision_app}, 
	\[
	\HMR_*(L_5,s_5) = \mathbb F_2 \oplus \mathbb F_2,
	\]
where each $\bbF$ summand arises from an irreducible generator.
Figure~\ref{fig:sampleL5} illustrates $s_5$ and two other kinds of $1$-sets for $L_5$.
There exist $5$ more $1$-sets isomorphic to a $\sfr$-circle with 5 $\sfc$-arcs.
The last family of $1$-sets are those isomorphic to two real circles with 4 $\sfc$-arcs between, where one of the circles has another $\sfc$-arc attached.
In other words, the complement of this arc is isomorphic to $s_4$.
\end{exmp}

\begin{exmp}
	[Handcuff]
	Let the handcuff be a 2-component unlink joined by a standard arc.
	This is in fact the prism $L_1$ above, and
	\[
		\HMR^{\circ}_*(\text{Handcuff}) = 0.
	\]
\end{exmp}
\begin{figure}[h!]
	\centering
	\includegraphics[width=0.4\linewidth]{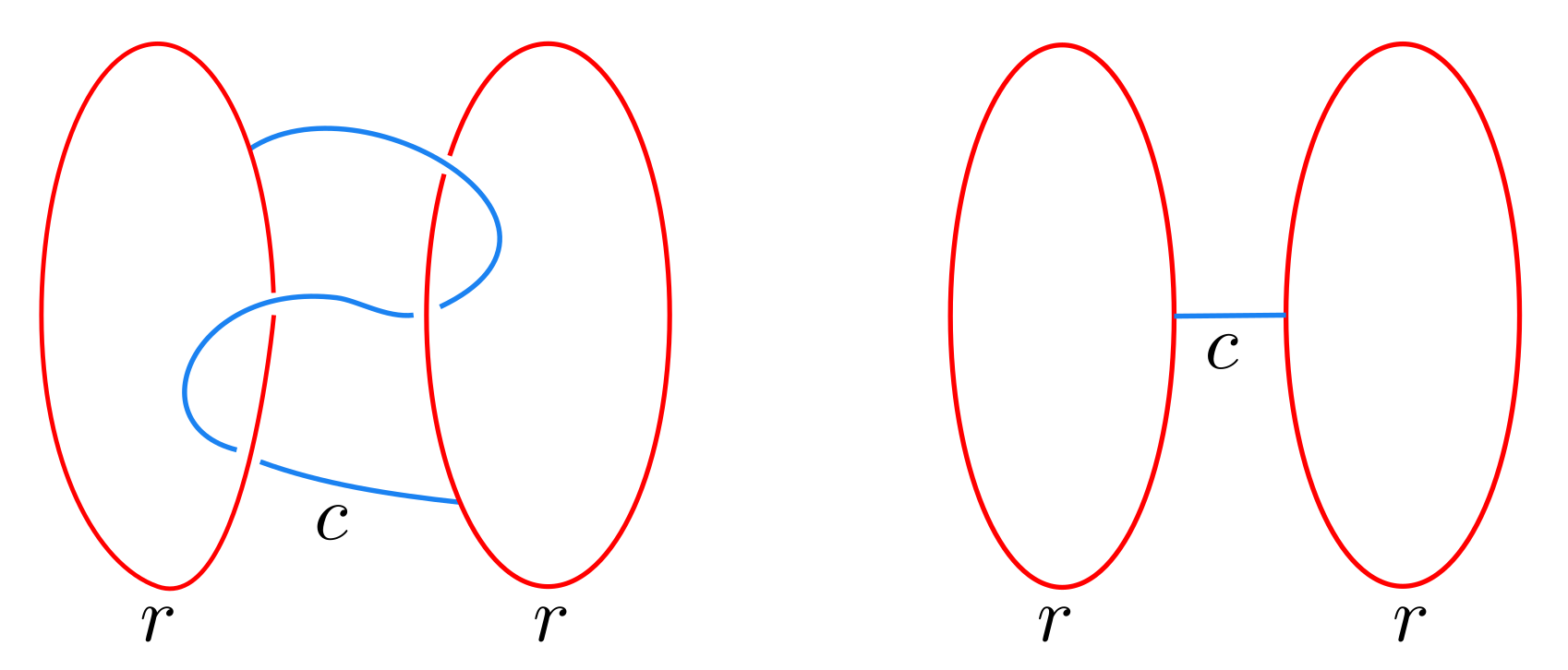}
	\caption{The handcuff and a twisted handcuff.}
	\label{fig:handcuffs}
\end{figure}
The examples so far are planar. 
Let us consider some genuinely spatial examples.
\begin{exmp}
	[Twisted handcuff]
	As in Figure~\ref{fig:handcuffs}, a twisted handcuff is the union of a $2$-component unlink and an arc joining the components in a way that the arc intersects the separating $2$-sphere of the components $3$ times.
	By Corollary~\ref{cor:vanish_3bridges},
	\[
		\HMR_*^{\circ}(\text{Twisted handcuff}) = 0.
	\]
\end{exmp}

\begin{figure}[h!]
	\centering
	\includegraphics[width=0.2\linewidth]{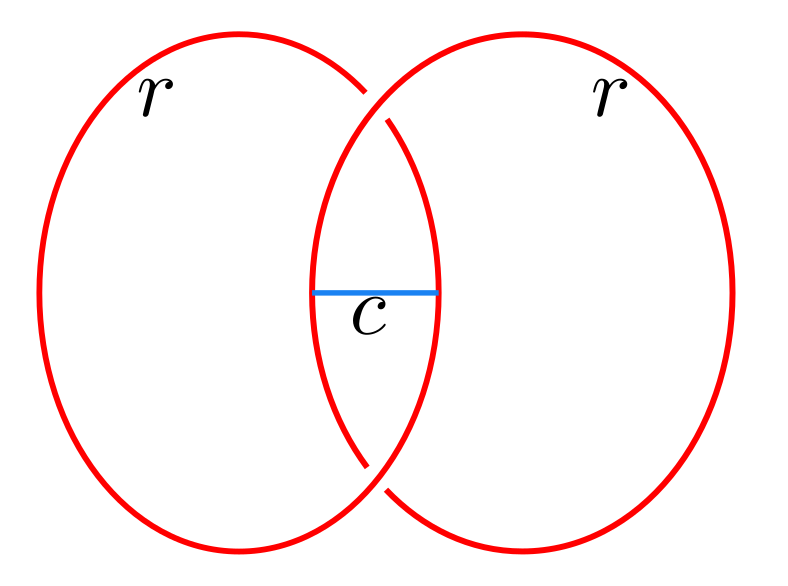}
	\caption{The Hopf handcuff.}
	\label{fig:hopfhandcuff}
\end{figure}
\begin{exmp}
	[Hopf handcuff]
Let $H$ be the web obtained by joining the two components of the Hopf link via a small arc.
The real double cover is $\mathbb{RP}^3$, and an $\upiota$-invariant PSC bifold metric can be obtained by pulling back a PSC bifold metric on $S^3$ via the covering $\mathbb{RP}^3 \to S^3$, singular along an unknot.
The lift $\tilde H^{\sfc}$ is the generator of the fundamental group of $\mathbb{RP}^3$. 
There are $4$ torsion real bifold \spinc structures, each admitting a unique reducible generator, so
\[
\widehat{\HMR}_*(H) = \bbF_2[\upsilon]^{\oplus 4}.
\]
\end{exmp}
	\begin{figure}[h!]
	\centering
	\includegraphics[width=0.27\linewidth]{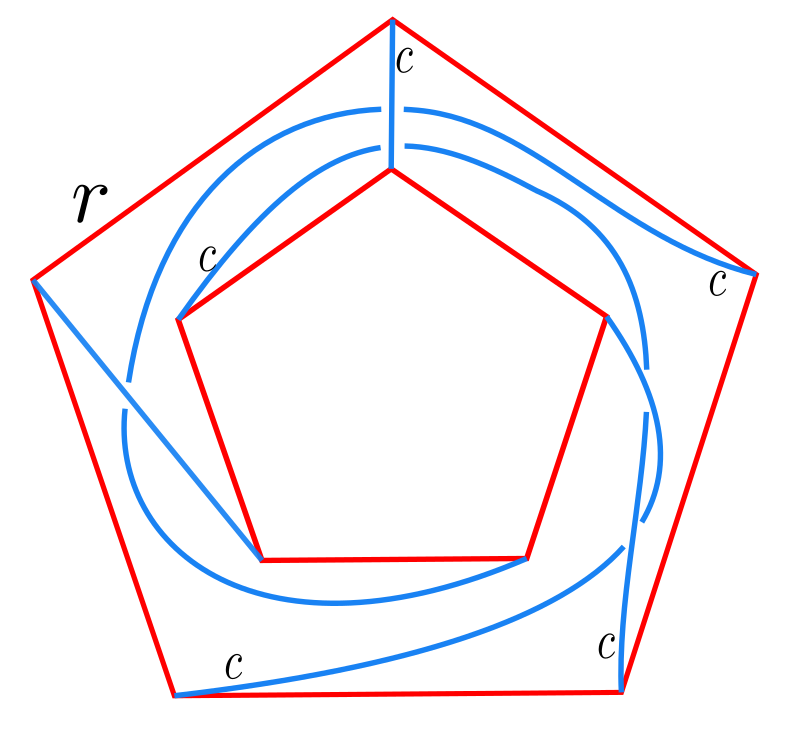}
	\caption{An emebedding of the Petersen graph.}
	\label{fig:peterson}
\end{figure}
\begin{exmp}[Petersen]
Let $P$ be the Petersen graph embedded in $S^3$ as in Figure~\ref{fig:peterson}.
This is a special case of the construction in Section~\ref{subsec:mapping_tori}.
So by the adjunction inequality and Proposition~\ref{prop:excision_app}, 
\[
\HMR_*(P,s_{\gamma}) = \bbF_2 \oplus \bbF_2,
\] 
where $s_{\gamma}$ is the $1$-set consisting of the five arcs connecting the two pentagons.
Note that $P$ is not $3$-colourable and non-planar.
\end{exmp}

\section{Remarks}
\label{sec:remarks}
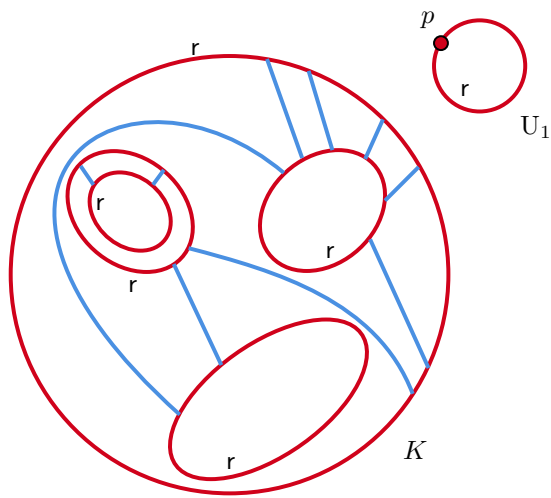
\begin{figure}[htp]
\centering
\tikzset{every picture/.style={line width=0.75pt}} 
\begin{tikzpicture}[x=0.75pt,y=0.75pt,yscale=-1,xscale=1]
	\draw  [color={rgb, 255:red, 208; green, 2; blue, 27 }  ,draw opacity=1 ][line width=1.5]  (43.5,141.88) .. controls (43.5,80.92) and (92.92,31.5) .. (153.88,31.5) .. controls (214.83,31.5) and (264.25,80.92) .. (264.25,141.88) .. controls (264.25,202.83) and (214.83,252.25) .. (153.88,252.25) .. controls (92.92,252.25) and (43.5,202.83) .. (43.5,141.88) -- cycle ;
	\draw  [color={rgb, 255:red, 208; green, 2; blue, 27 }  ,draw opacity=1 ][line width=1.5]  (73.26,110) .. controls (68.72,93.16) and (78.7,79.5) .. (95.54,79.5) .. controls (112.39,79.5) and (129.72,93.16) .. (134.26,110) .. controls (138.8,126.84) and (128.82,140.5) .. (111.98,140.5) .. controls (95.13,140.5) and (77.8,126.84) .. (73.26,110) -- cycle ;
	\draw  [color={rgb, 255:red, 208; green, 2; blue, 27 }  ,draw opacity=1 ][line width=1.5]  (170.26,109.5) .. controls (174.68,92.66) and (191.92,79) .. (208.77,79) .. controls (225.61,79) and (235.68,92.66) .. (231.26,109.5) .. controls (226.84,126.34) and (209.6,140) .. (192.75,140) .. controls (175.91,140) and (165.83,126.34) .. (170.26,109.5) -- cycle ;
	\draw  [color={rgb, 255:red, 208; green, 2; blue, 27 }  ,draw opacity=1 ][line width=1.5]  (133.22,204.78) .. controls (149.25,182.54) and (180.29,164.5) .. (202.54,164.5) .. controls (224.79,164.5) and (229.83,182.54) .. (213.79,204.78) .. controls (197.76,227.03) and (166.72,245.07) .. (144.47,245.07) .. controls (122.22,245.07) and (117.18,227.03) .. (133.22,204.78) -- cycle ;
	\draw  [color={rgb, 255:red, 208; green, 2; blue, 27 }  ,draw opacity=1 ][line width=1.5]  (84.01,110) .. controls (81.07,99.09) and (87.53,90.25) .. (98.44,90.25) .. controls (109.35,90.25) and (120.57,99.09) .. (123.51,110) .. controls (126.45,120.91) and (119.99,129.75) .. (109.08,129.75) .. controls (98.17,129.75) and (86.95,120.91) .. (84.01,110) -- cycle ;
	\draw [color={rgb, 255:red, 74; green, 144; blue, 226 }  ,draw opacity=1 ][line width=1.5]    (224.75,124) -- (254.25,189) ;
	\draw [color={rgb, 255:red, 74; green, 144; blue, 226 }  ,draw opacity=1 ][line width=1.5]    (125.75,136.5) -- (149.75,187.5) ;
	\draw [color={rgb, 255:red, 74; green, 144; blue, 226 }  ,draw opacity=1 ][line width=1.5]    (78.75,87) -- (85.25,96.5) ;
	\draw [color={rgb, 255:red, 74; green, 144; blue, 226 }  ,draw opacity=1 ][line width=1.5]    (128.75,212.5) .. controls (-5.25,92) and (103.75,25) .. (181.25,90.5) ;
	\draw [color={rgb, 255:red, 74; green, 144; blue, 226 }  ,draw opacity=1 ][line width=1.5]    (133.25,128.5) .. controls (177.5,140) and (228.25,155) .. (246.25,202) ;
	\draw [color={rgb, 255:red, 74; green, 144; blue, 226 }  ,draw opacity=1 ][line width=1.5]    (115.25,96.5) -- (120.75,89) ;
	\draw [color={rgb, 255:red, 74; green, 144; blue, 226 }  ,draw opacity=1 ][line width=1.5]    (172.75,33) -- (190.5,83) ;
	\draw [color={rgb, 255:red, 74; green, 144; blue, 226 }  ,draw opacity=1 ][line width=1.5]    (193.75,39) -- (205.75,79) ;
	\draw [color={rgb, 255:red, 74; green, 144; blue, 226 }  ,draw opacity=1 ][line width=1.5]    (231.25,63.5) -- (222.25,83.5) ;
	\draw [color={rgb, 255:red, 74; green, 144; blue, 226 }  ,draw opacity=1 ][line width=1.5]    (249.25,87.5) -- (232.25,104.5) ;
	\draw  [color={rgb, 255:red, 208; green, 2; blue, 27 }  ,draw opacity=1 ][line width=1.5]  (257,36.5) .. controls (257,23.8) and (267.3,13.5) .. (280,13.5) .. controls (292.7,13.5) and (303,23.8) .. (303,36.5) .. controls (303,49.2) and (292.7,59.5) .. (280,59.5) .. controls (267.3,59.5) and (257,49.2) .. (257,36.5) -- cycle ;
	\draw  [fill={rgb, 255:red, 208; green, 2; blue, 27 }  ,fill opacity=1 ] (257,25) .. controls (257,23.07) and (258.57,21.5) .. (260.5,21.5) .. controls (262.43,21.5) and (264,23.07) .. (264,25) .. controls (264,26.93) and (262.43,28.5) .. (260.5,28.5) .. controls (258.57,28.5) and (257,26.93) .. (257,25) -- cycle ;
	
	\draw (101.25,143.4) node [anchor=north west][inner sep=0.75pt]    {$\sfr$};
	\draw (84.75,101.9) node [anchor=north west][inner sep=0.75pt]    {$\sfr$};
	\draw (200.75,125.9) node [anchor=north west][inner sep=0.75pt]    {$\sfr$};
	\draw (150.75,232.4) node [anchor=north west][inner sep=0.75pt]    {$\sfr$};
	\draw (132.75,23.4) node [anchor=north west][inner sep=0.75pt]    {$\sfr$};
	\draw (240,224.4) node [anchor=north west][inner sep=0.75pt]    {$K$};
	\draw (299,60.4) node [anchor=north west][inner sep=0.75pt]    {$\text{U}_{1}$};
	\draw (249,7.4) node [anchor=north west][inner sep=0.75pt]    {$p$};
	\draw (268.75,44.4) node [anchor=north west][inner sep=0.75pt]    {$\sfr$};
\end{tikzpicture}
\caption{\label{fig:planar_web}
An even $1$-set of $K$ with $\sfr$-cycles labelled.
The top-right corner is an unknot $\U_1$ based at $p$.
}

\end{figure}
\subsection{Tait colourings}
A \emph{Tait colouring} of $K$ is a function from $\Edg(K)$ to a 3-element set of ``colours'' such that the colours of the three edges adjacent to any vertex are distinct.
It is known~\cite{Tait1884} that the existence of Tait colourings for planar trivalent graphs is equivalent to the four-colouring theorem.
Let $\Tait(K)$ he the number of Tait colourings of $K$.

Recall that a \emph{$1$-set} for a web $K$ is a colouring of $\Edg(K)$ by a 2-element set $\{\sfr, \sfc\}$ such that around each vertex there are exactly two edges coloured by $\sfr$ and one edge coloured by $\sfc$.
A $1$-set $s$ is \emph{even} if the homology class $[K^{\sfc}]$ is zero in $H_1(\reals^3,K^{\sfr};\bbF_2)$.
Equivalently, every $\sfr$-cycle has an even number of endpoints of $\sfc$-arcs.
Given a $1$-set $s$, denote by $n(s)$ the number of cycles in the complementary $2$-set.
We have the following fact about number of Tait colourings and $1$-sets:
\begin{lem}
	\[
	\pushQED{\qed} 
	\Tait(K) = \sum_{s \in \{\text{even $1$-sets}\}} 2^{n(s)}. \qedhere
	\pushQED{\qed} 
	\]
\end{lem}
\begin{lem}
\label{lem:red_even}
	Let $K$ be a planar web.
	Then $\HMR^{\circ}_{*}(K,s)$ admits a reducible generator if and only if $s$ is even.
	Moreover, let $s$ be an even $1$-set and $Y = \dbcv_2(S^3, K^{\sfr})$ be the real cover of $(K,s)$.
	Then $b^1(Y) = n(s)-1$.
\end{lem}
\begin{proof}
	Observe that $K$ consists of a union of unknotted cycles $K^{\sfr}$ and arcs $K^{\sfc}$ (cf. Figure~\ref{fig:planar_web}).
	Hence the real cover $Y$ has underlying $3$-manifold $\#_{n-1}(S^1 \times S^2)$ having $b^1=n(s)-1$.	
Furthermore, $s$ is even if and only if the lift $\tilde K^{\sfc}$ in $Y $ is zero in $H_1(Y;\bbF_2)$.
	The bifold $(\check Y, \tilde K^{\sfc})$ admits a torsion \spinc structure if and only if $\tilde K^{\sfc}$ is null-homologous mod 2.
\end{proof}
Adding an unknot $\U_1$ has the effect of increasing $b^1$ of the real cover by $1$.
We have $H^*(\dbcv_2(S^3, K^{\sfr} \sqcup \U_1)) \cong \Lambda[x_1,\dots,x_n]$, which has rank $2^{n(s)}$.  
In the case when $K$ is the n-component unlink, theta graph, or tetrahedral graph, we have
\begin{equation}
	\label{eq:M_lowerbound}
	\dim \widetilde{\HMR}_*(K\sqcup \U_1, p) \ge \Tait(K).
\end{equation}
We expect this lower bound of $\dim \widetilde{\HMR}_*(K \sqcup \U_1)$ to hold for planar web $K$ and $p \in \U_1$.
This would follow from an exact triangle in $\widetilde{\HMR}_*$ (cf. Section~\ref{sec:exact_triangle}).
It is known that $\dim J^{\sharp}(K) \ge \Tait(K)$~\cite{KMdefweb2019}.

The bound~\eqref{eq:M_lowerbound} is by no means sharp even in the case of $n$-component unlinks.
There is already a discrepancy between the rank of $\widetilde{\HMR}(K)$ and $\Tait(K)$ due to the choices of bifold \spinc structures.
By restricting to a suitable subset of real bifold \spinc structures, we expect a version of $\HMR$ whose rank is equal to the Tait colourings of the planar unlinks, and at least $\Tait(K)$ in general.

The Seiberg-Witten version of Kronheimer and Mrowka's program is to construct a monopole Floer homology whose rank agrees with $\Tait(K)$ and has non-vanishing properties for planar bridgeless webs.
We will investigate further non-vanishing properties of $\HMR$.

We make an observation about $\widetilde{\hmr}_{*}$ ~\cite{ljk2024SSKh}:
\begin{prop}
	\label{prop:hmrKr}
	Let $K$ be a planar web and $s$ be an even $1$-set.
	Then $\dim \widetilde{\hmr}_*(K^{\sfr} \sqcup \U_1)=2^{n(s)}.$
\end{prop}
\begin{proof}
	Since $K^{\sfr}$ is an $n$-component unlink, this proposition follows from $\widetilde{\hmr}_*(K^{\sfr} \sqcup \U_1)=2^{n(s)}$, since the homology of the $n$-dimensional torus which has rank $2^{n}$.
	This equality is noted in \cite[Example~2.4]{ljk2024SSKh}.
\end{proof}

\subsection{Relations in monopole web homologies}
Kronheimer and Mrowka \cite{KMdefweb2019} defined a deformation $J^{\sharp}(K;\Gamma)$ of $J^{\sharp}(K)$, whose rank $\Rank_{R}J^{\sharp}(K;\Gamma)$ is equal to the number of Tait colourings, whenever $K$ is planar.
Here, $\Gamma$ is a local coefficient system over a ring $R$, and $\Rank_{R}J^{\sharp}(K;\Gamma)$ can be shown be equal to $\Rank_{R}J^{\sharp}(K;\Gamma')$ by introducing another local system $\Gamma'$ from inverting an operator $P$.
Similar to $\HMR$, the homology $J^{\sharp}(K;\Gamma')$ admits a direct sum decomposition over $1$-sets.
\[
	J^{\sharp}(K;\Gamma') = \bigoplus_{s \in \{1-\text{sets}\}} J^{\sharp}(K,s;\Gamma')
\]
($J^{\sharp}(K,s;\Gamma')$ was denoted as $V(K,s)$ in~\cite[Section~5.1]{KMdefweb2019}.)
By contrast, neither $J^{\sharp}$ nor $J^{\sharp}(\cdot,\Gamma)$ admits such a decomposition.
In light of Proposition~\ref{prop:hmrKr}, $J^{\sharp}(K;\Gamma')$ is analogous to $\bigoplus_{\{\text{1-sets}\}} \widetilde{\hmr}(K^{\sfr} \sqcup \U_1)$.

\begin{figure*}[ht!]
	\centering
	\begin{subfigure}[t]{0.4\textwidth}
		\centering
		\includegraphics[height=1.2in]{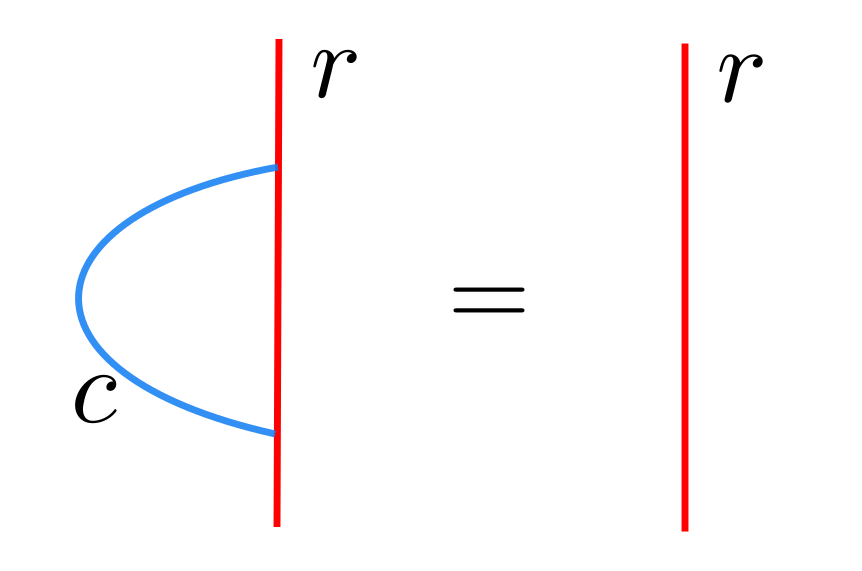}
		\caption{Removing a bigon.}
		\label{fig:remove_blue}
	\end{subfigure}%
	~ 
	\begin{subfigure}[t]{0.37\textwidth}
		\centering
		\includegraphics[height=1.2in]{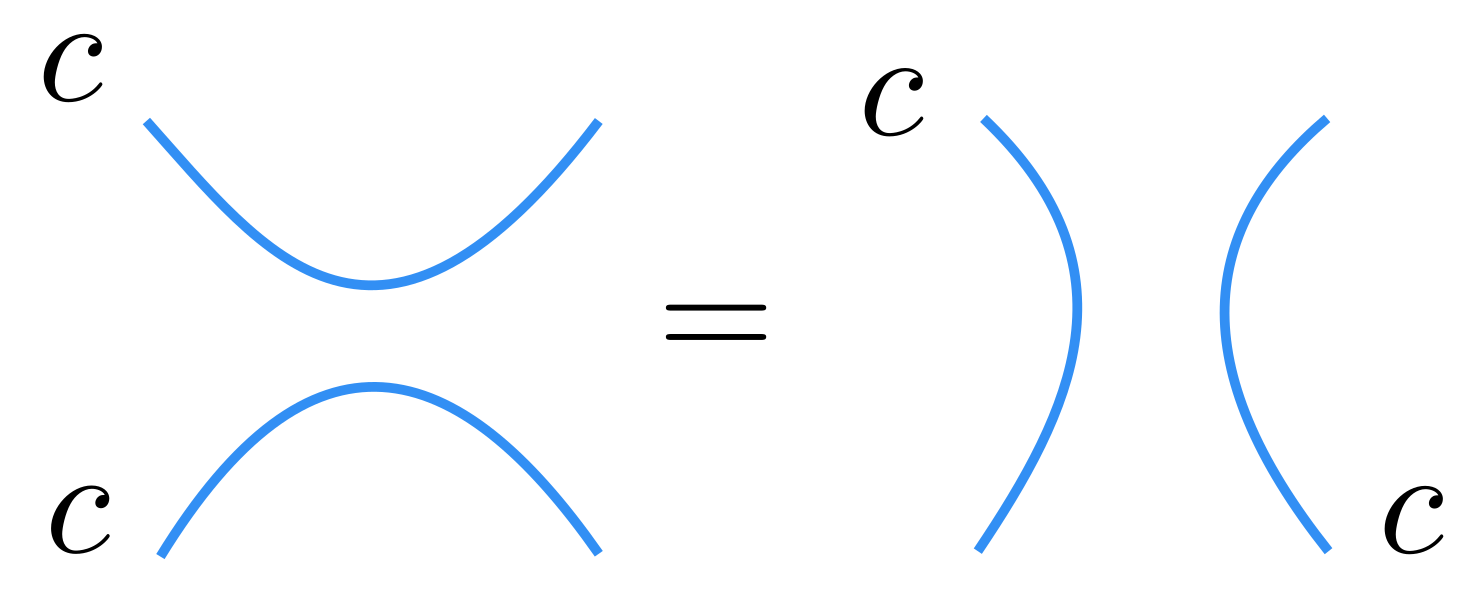}
		\caption{Rotating $\sfc$-arcs.}
		\label{fig:rotate_blue}
	\end{subfigure}
	\caption{Two relations.}
	\label{fig:two_relations}
\end{figure*} 
The relations illustrated in Figure~\ref{fig:two_relations} are satisfied by 
$J^{\sharp}(K,s;\Gamma')$~\cite[Corollary~5.15]{KMdefweb2019}.
The two relations state that two webs with $1$-sets that differ locally as in Figure~\ref{fig:remove_blue} or Figure~\ref{fig:rotate_blue} have isomorphic $J^{\sharp}(K,s;\Gamma')$.
Again, we used the convention that an edge $e$ is coloured $\sfc$ if $e \in s$ and coloured $\sfr$ if $e \not\in s$. 

The versions of the relations without $1$-sets
were observed earlier in $\mathfrak{sl}(3)$-homology~\cite[Figure~8]{Khovanov2004sl3} and also in $J^{\sharp}$~\cite[Proposition~6.5]{KMweb}.
However, the bigon removal relation is not satisfied for $\HMR$.
Indeed, with the addition of a $\sfc$-arc on the left of Figure~\ref{fig:remove_blue}, there are twice as many bifold \spinc structures for isotropy reasons.

\subsection{Exact triangles}
\label{sec:exact_triangle}
Kronheimer and Mrowka proved an exact triangle in $J^{\sharp}$ \cite{KMwebtriangle}, where the usual unoriented skein triple has additional $V_4$-singularities as in Figure~\ref{fig:web_triangle}.
\begin{figure}[h!]
	\centering
	\includegraphics[width=0.5\linewidth]{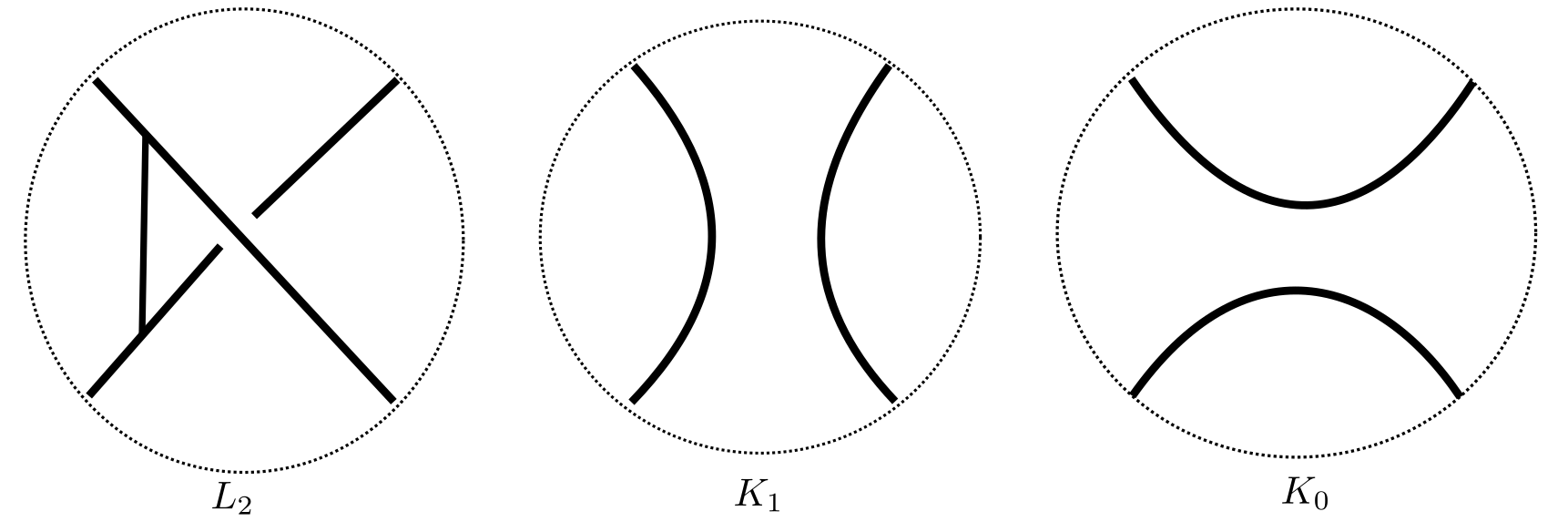}
	\caption{An exact triangle in $J^{\sharp}$}
	\label{fig:web_triangle}
\end{figure}
The $1$-set decomposition of $J^{\sharp}(K;\Gamma')$ is compatible with the exact triangle~\cite[Proposition~3.3]{KMdefweb2019}:
\[
	\cdots \to 
	J^{\sharp}(L_2,s_2;\Gamma') \to
	J^{\sharp}(K_1,s_1;\Gamma') \to 
	J^{\sharp}(K_0,s_0;\Gamma') \to 
	\cdots.
\]
One hopes that there exists an exact triangle in $\HMR$ and the triangle decomposes over the $1$-sets too, as the definition has $1$-sets built in.
However, the naive form of exact triangle shown below would fail if the two relations in Figure~\ref{fig:two_relations} hold.
\[
	``\cdots \to 
	\HMR^{\circ}(L_2,s_2) \to
	\HMR^{\circ}(K_1,s_1) \to 
	\HMR^{\circ}(K_0,s_0) \to 
	\cdots''.
\]
Indeed, start with the $5$-prism $L_5$ with the $1$-set $s_5$ as in Section~\ref{sec:example}.
The exact triangle implies that a pair of arcs bridging pentagons can be changed to a pair of arcs that starts and ends at the same pentagons.
Moreover, any arc that start and end at the same $\sfr$-edge can be removed without changing the homology. 
Successive applications of the exact triangle shows that $\HMR_{\bullet}(L_5,s)$ is isomorphic to $\HMR_{\bullet}(\text{Handcuff}) = 0$, which is a contradiction to the non-triviality of $\HMR_{\bullet}(L_5,s_5)$.
In the case with no trivalent points and when every edge is coloured by $\sfr$, an (unoriented skein) exact triangle was verified in \cite{ljk2023triangle}.

\bibliographystyle{alpha}
\bibliography{sw_webs.bib}
\Addresses

\end{document}